\documentclass[11pt,reqno]{amsart}
\usepackage{amsmath}
\usepackage{multicol}
\usepackage{amsfonts}
\usepackage{amssymb}
\usepackage{graphicx}
\usepackage{amsthm,graphicx,color}
\usepackage{a4wide}
\usepackage{epstopdf}%

\usepackage{dsfont}
\allowdisplaybreaks

\usepackage[colorlinks=true]{hyperref}
\hypersetup{linkcolor=red,citecolor=blue,filecolor=dullmagenta,urlcolor=blue} 

\usepackage{wrapfig}
\usepackage{tikz}
\usetikzlibrary{arrows,calc,decorations.pathreplacing}
\definecolor{light-gray1}{gray}{0.90}
\definecolor{light-gray2}{gray}{0.80}
\definecolor{light-gray3}{gray}{0.60}
\definecolor{myred1}{RGB}{255, 0, 0}
\definecolor{myyellow1}{RGB}{255, 255, 219}
\definecolor{mygreen1}{RGB}{0, 255, 0}
\definecolor{mygreen2}{RGB}{0, 126, 0}
\definecolor{myblue1}{RGB}{0, 0, 255}

\newcommand{\R} {\mathbb R}
\newcommand{\cuad}{{\sqcap\kern-.68em\sqcup}}

\newcommand{\be}{\begin{equation}}
\newcommand{\ee}{\end{equation}}

\definecolor{darkgreen}{rgb}{0.2,0.7,0.1}

\newcommand{\sech}{\mathop{\mbox{\normalfont sech}}\nolimits}

\newcommand{\N}{\mathbb{N}}

\newcommand{\al}{\alpha}
\newcommand{\bt}{\beta}
\newcommand{\ga}{\gamma}

\def\bm{\left( \begin{array}{cc}}
\def\endm{\end{array}\right)}

\newcommand{\ba}{\begin{equation*}}
\newcommand{\ea}{\begin{equation*}}
\newcommand{\bea}{\begin{eqnarray}}
\newcommand{\eea}{\end{eqnarray}}
\newcommand{\bee}{\begin{eqnarray*}}
\newcommand{\eee}{\end{eqnarray*}}
\newcommand{\ben}{\begin{enumerate}}
\newcommand{\een}{\end{enumerate}}

\newtheorem{theorem}{Theorem}[section]
\newtheorem{proposition}{Proposition}[section]
\newtheorem{corollary}{Corollary}[section]
\newtheorem{lemma}{Lemma}[section]

\theoremstyle{definition}
\newtheorem{definition}{Definition}[section]

\theoremstyle{remark}
\newtheorem{remark}{Remark}[section]
 
 \numberwithin{equation}{section}

\title[Nonexistence and uniqueness of breathers]{Uniqueness of quasimonochromatic breathers for the generalized Korteweg-de Vries and Zakharov-Kuznetsov models}

\author[Faya]{Jorge Faya}
\address{Instituto de Ciencias F\'isicas y Matem\'aticas, Facultad de Ciencias, Universidad Austral de Chile, Valdivia, Chile.}
\email{jorge.faya@uach.cl}
\thanks{J.F.'s work is partially supported by FONDECYT 1231250 and 1221076.}

\author[Figueroa]{Pablo Figueroa}
\address{Instituto de Ciencias F\'isicas y Matem\'aticas, Facultad de Ciencias, Universidad Austral de Chile, Valdivia, Chile.}
\email{pablo.figueroa@uach.cl}

\author[Mu\~noz]{Claudio Mu\~noz}  
\address{Departamento de Ingenier\'{\i}a Matem\'atica and Centro
de Modelamiento Matem\'atico (UMI 2807 CNRS), Universidad de Chile, Casilla
170 Correo 3, Santiago, Chile.}
\email{cmunoz@dim.uchile.cl}
\thanks{C.M. was partially funded by Chilean research grants ANID 2022 Exploration 13220060, FONDECYT 1191412, 1231250, and Basal CMM FB210005 and MathAmSud WAFFLE 23-MATH-18. Part of this work was carried out while C.M. was visiting the ICFM at U. Austral, Valdivia Chile. He would like to thank the Institute and the organizers for their warming hospitality and support.}

\author[Poblete]{Felipe Poblete}
\address{Instituto de Ciencias F\'isicas y Matem\'aticas, Facultad de Ciencias, Universidad Austral de Chile, Valdivia, Chile.}
\email{felipe.poblete@uach.cl}
\thanks{F.P.'s work is partially supported by ANID Exploration project 13220060, ANID project FONDECYT 1221076 and MathAmSud WAFFLE 23-MATH-18.}


\begin{document}

\begin{abstract}
Consider the generalized Korteweg-de Vries (gKdV) equations with power nonlinearities $q=2,3,4\ldots$ in dimension $N=1$, and the Zakharov-Kuznetsov (ZK) model with integer power nonlinearities $q$ in higher dimensions $N\geq 2$. Among these power-type models, the only conjectured equation with space localized time periodic breathers is the modified KdV (mKdV), corresponding to the case $q=3$ and $N=1$. Quasimonochromatic solutions were introduced by Mandel \cite{RainerMandel2021} to show that sine-Gordon is the only scalar field model with breather solutions among this class.  In this paper we consider smooth generalized quasimonochromatic solutions of arbitrary size for gKdV and ZK models and provide a rigorous proof that mKdV is the unique power-like model among them with spatially localized breathers of this type. In particular, we show the nonexistence of breathers of this class in the ZK models. The method of proof involves the use of the naturally coherent algebra of Bell's polynomials to obtain particularly distinctive structural elliptic PDEs satisfied by breather-like quasimonochromatic solutions. A reduction of the problem to the classification of solutions of these elliptic PDEs in the entire space is performed, and de Giorgi type uniqueness results are proved in this particular case, concluding the uniqueness of the mKdV breather, and the nonexistence of localized smooth breathers in the ZK case. No assumption on well-posedness is made, and the power of the nonlinearity is arbitrary. 
\end{abstract}

\maketitle

\section{Introduction}

\subsection{Setting} Let $N\geq 1$.  In this paper we consider the model 
\begin{equation}\label{mKdVN}
\partial_t u + \partial_{x_1} (\Delta u +2u^q)=0,
\end{equation}
for a given real-valued function $u=u(t,x)$. Here $t\in\R$, $x=(x_1,x_2,\ldots,x_N)\in\R^N$, and $q\in\{2,3,4,\ldots\}$. The Laplacian operator $\Delta$ is defined as $\Delta=\sum_{j=1}^N \partial_{x_j}^2.$ A particular case of this equation is the completely integrable, modified Korteweg-de Vries (mKdV) equation \cite{HIROTA1,Wadati} ($N=1$, $q=3$)
\begin{equation}\label{mKdV}
\partial_t u + \partial_x(\partial_x^2 u +2u^3)=0,
\end{equation}
which has been extensively studied during past years \cite{KS3}. Another interesting case is represented by the KdV model \cite{KdV}, obtained for $N=1$ and $q=2$. Both KdV and mKdV have an impressive range of mathematical structure, that will be reviewed below. If $N\geq 2$, \eqref{mKdVN} is known as the Zakharov-Kuznetsov (ZK) model \cite{ZK,Kuznetsov}. A detailed bibliographical description on the generalized KdV models can be found in Linares and Ponce \cite{LP}. Depending on the dimension $N$, the Cauchy problem associated to \eqref{mKdVN} may be either well-posed or ill-posed in standard Sobolev spaces. Indeed, \eqref{mKdV} in the case $N=1$ is locally well-posed in $H^1$ \cite{KPV1} and globally well-posed if $q\leq4$. Blow up is present if $q\geq 5$, see \cite{MM3}. In the case where $N\geq 2$, \eqref{mKdVN} is the Zakharov-Kuznetsov and may have blow-up solutions in dimension $N=2$ and $q=3$ \cite{Farah2}. Local and global well-posedness in dimension $N=2$ is well-understood, see \cite{Kinoshita:Arxiv:2019} and references therein. The situation for $N\geq 3$ is still not well-understood, but local well-posedness is well-known \cite{Herr:Kinoshita:Arxiv:2020}. Dynamical long time behavior and collision results in the ZK case have been obtained in \cite{Farah2,CMPS,PV1,PV2}. 

\medskip

In this paper, we are interested in the existence and nonexistence of breathers for the general model \eqref{mKdVN}. We understand breathers as localized in space, time periodic smooth solutions. Among the equations represented in \eqref{mKdVN},  mKdV \eqref{mKdV} is well-known by having explicit breather solutions. Indeed, let $\alpha,\beta >0$. From \cite{Wadati}, \cite[p. 139]{Lamb} and \cite{AM}, the space aperiodic mKdV breather solution is  given by the 4-parameter expression
\begin{equation}\label{BBB}
u=B(t,x;\alpha,\beta,x_1,x_2)= -2 \partial_x \arctan \left( \frac{\beta}{\alpha} \frac{\sin (2\alpha x+\delta t +x_1)}{\cosh(2\beta x+\gamma t+x_2)}\right),
\end{equation}
where  $\gamma:=8\beta(3\alpha^2-\beta^2)$, $\delta:=8\alpha(\alpha^2-3\beta^2)$, and $x_1,x_2$ are free real-valued shifts. In applications, breathers appear in the description of the movement of planar curves \cite{GoPe,Na1,Na2}, and lack of uniformly continuous well-posedness \cite{KPV2,Ale1}. Its natural structure has been understood using numerical \cite{AGV}, PDE \cite{AM,Munoz,AleProy1,AleProy2} and Inverse Scattering methods \cite{OW,Ale,Ale0,ChenLiu}, revealing that their structural stability/instability is characterized by universal tools \cite{Munoz2,AFM}.

\medskip

Following Mandel \cite{RainerMandel2021}, who described \emph{quasimonochromatic} solutions in scalar field theories, we seek to describe arbitrary size breathers as particular solutions of \eqref{mKdVN}. Unfortunately, the proofs in \cite{RainerMandel2021} do not apply to the gKdV/ZK cases due to the lack of a closed related algebraic structure as present in scalar field models.  To explain this point in detail, first notice that in the case $\gamma=0$, the breather solution $B$ \eqref{BBB} has the  representation
\begin{equation}\label{ansatzlamb}
B(t,x)= -2 \partial_x \arctan \left( \sqrt{3} \frac{\sin (2\alpha x-32\alpha^2 t +x_1)}{\cosh(\sqrt{3} \alpha x+x_2)}\right),
\end{equation}
o more generally
\begin{equation}\label{BBBB}
B(t,x)= \partial_x \big( F(p(x)\sin \left(2\alpha(x+m t)) \right) \big).
\end{equation}
Notice that with no loss of generality we can always assume $F(0)=0$, and if $u$ is solution, $-u$ and  $\lambda u(\lambda^3 (t+t_0), \lambda (x+x_0)) $, $\lambda>0$, $t_0,x_0\in\mathbb R$ are also solutions as well. These are the so-called classical symmetries in mKdV and must be considered in any uniqueness result. Replacing \eqref{BBBB} in \eqref{mKdVN} ($N=1$) leads to a complicated equation for $F$ and $p$ having a strongly perturbed structure compared with the one appearing in Mandel's work. Therefore, finding that \eqref{ansatzlamb} is the unique quasimonochromatic breather seems not clear. In this work we will introduce new techniques that will help us to overcome this difficulty.


\begin{definition}[see \cite{RainerMandel2021}]\label{mono}
 We call a function $u:\mathbb R \times \mathbb{R}^N \to \mathbb R$ a quasimonochromatic breather of \eqref{mKdVN} if
 \begin{equation}\label{ansatz}
\begin{aligned}
 &~{} u(t,x) =\partial_{x_1} \left( F(p(x)\sin (2\alpha(x_1+m t))) \right),\\
 &~{} \hbox{with} \quad x_1\in\R, ~x=(x_1,\ldots,x_N)\in \mathbb R^N,~ t\in \mathbb R,
 \end{aligned}
 \end{equation}
is a smooth solution of \eqref{mKdVN}, where $m\in \mathbb R$, $\alpha>0$, and $F\in C^\infty(\mathbb R)$, $p\in C^2(\mathbb R)$ are nontrivial functions such that $F(0)=0$ and $p(x), \nabla p(x), D^2 p(x)\to 0$ as $|x|\to \infty$. Finally, we say that $u$ is trivial if $u\equiv 0$. This condition can be obtained e.g. by choosing $F\equiv 0$ or constant.
\end{definition}

Notice that standard solitons do not enter in the quasimonochromatic family considered above. Under Definition \ref{mono} we will show that \eqref{mKdVN} has \eqref{ansatzlamb} as the unique quasimonochromatic solutions in any dimension. 

\begin{remark}
The condition on $p$ at infinity is important, since mKdV ($N=1$, $q=3$) possesses periodic in space breathers, see e.g. \cite{KKS2,AMP2} for further details. Additionally, one has \emph{line breathers} $u(t,x)=B(t,x_1)$, natural extensions of \eqref{ansatzlamb}. 
\end{remark}

\begin{theorem}[$N=1$ case]\label{MT1} Assume $F$ real analytic in a neighborhood of $0$ with $F(0)=0$, and $N=1$. Then the breather solution \eqref{ansatzlamb} is the unique nonzero solution to gKdV (up to scaling and shifts symmetries of the equation) satisfying the quasimonochromatic form \eqref{ansatz}.
\end{theorem}

In the higher dimensional case $N\geq 2$, the situation is strongly restrictive. 


\begin{theorem}[$N\geq 2$ case]\label{MT2}
Assume $F$ real analytic in a neighborhood of $0$, with $F(0)=0$. If $N\ge 2$, there is no nontrivial quasimonochromatic breather solution of \eqref{mKdVN}. 
\end{theorem}

Theorems \ref{MT1} and \ref{MT2} are rigidity results and are in concordance with the results obtained in \cite{MuPo} and \cite{MMPP}, where it was proved using different methods that pure-power subcritical gKdV and ZK models do not possess breathers. In the last two works, two conditions were needed: data in a suitable Sobolev space (usually $L^2$), and the global well-posedness of the model in that Sobolev space. Here we only require enough regularity of the exact solution, and sufficient decay of $p(x)$, without need of well-posedness theory.  Because of this, we can consider any pure power, despite the absence or lack of well-posedness. No particular size is required as well. Additionally, the method of proof combines elementary techniques and classical results in elliptic theory, showing a promising potential for applications in other dispersive models. Structural uniqueness properties have also been recently obtained in \cite{ACMPT} for the KP-II model, where the key solutions are nonlocalized in $\mathbb R^2$ and given by line solitons and line multisolitons. Nonexistence of breathers has been recently addressed in \cite{KMM,GG}.

\begin{remark}[On the conditions on $F$] 
The condition $F$ real analytic is necessary to avoid the case of nonzero functions with exactly zero Taylor expansion at zero. One example of this type is $F(s)=e^{-1/s^2}$, $F(0)=0$, for which all its derivatives vanish.
\end{remark}

\begin{remark}[On Gardner breathers]
If the nonlinearity is different from the pure power, there are well-known localized breathers. Indeed, the Gardner model is
\begin{equation}\label{Gardner}
\partial_t u + \partial_{x}(\partial_x^2 u +  u^2 + \mu u^3) = 0, \quad  \mu>0.
\end{equation}
This model possesses (stable) breather solutions \cite{Alejo}. Indeed,  if  $\al, \bt>0$ are such that $\Delta:=\al^2+\bt^2-\frac2{9\mu}>0$, then
\begin{equation*}\label{BGe}
B(t,x) :=   2\sqrt{\frac 2\mu}\partial_x\arctan\left(\frac{\mathcal G(y_1,y_2)}
{\mathcal F(y_1,y_2)}\right),
\end{equation*}
where
\begin{equation*}
\begin{aligned}
\mathcal G (y_1,y_2)& :=   \frac{\bt\sqrt{\al^2+\bt^2}}{\al\sqrt{\Delta}}\sin(\al y_1) -\frac{\sqrt2 \bt}{3\sqrt{\mu}\Delta} e^{\beta y_2},\\
\mathcal F(y_1,y_2) & :=   \cosh(\bt y_2)-\frac{\sqrt2 \bt}{3\sqrt{\mu}\al\sqrt{\Delta}} \frac{[\al\cos(\al y_1)-\bt\sin(\al y_1)]}{\sqrt{\al^2+\bt^2}},
\end{aligned}
\end{equation*}
and $y_1 = x+ \delta t$, $y_2 = x+ \ga t$ , $\delta := \al^2-3\bt^2$, $\ga :=3\al^2-\bt^2$, is a smooth non decaying solution to \eqref{Gardner}. The structure of the Gardner breather \eqref{BGe} does not follow the quasimonochromatic structure proposed in Theorems \ref{MT1} and \ref{MT2}, becoming the proof of uniqueness an interesting open problem.
\end{remark}

\begin{remark}
Assuming more complex structures, of ``polychromatic''  type, may lead to the entrance of new solutions into the scene, such as e.g. multibreathers. A first polychromatic structure is the one naturally given by the Gardner breather \eqref{BGe}. In that sense, the forthcoming step should be the study of suitable structures composed of at most two oscillatory independent parts.
\end{remark} 

\subsection{Ideas of proof} The proof of Theorems \ref{MT1} and \ref{MT2} is based in elementary ideas having in principle no correlation with a well-posedness theory. Mandel's work \cite{RainerMandel2021} remains as the main basis. Let us briefly explain the new ideas. First of all,  we rewrite \eqref{mKdVN} as 
\[
u(t,x)= \partial_{x_1} v(t,x), \quad v(t,x):=F(p(x)\sin (2\alpha(x_1+m t))),
\]
where $\alpha>0$ and $m\in\R$ are free parameters. Therefore (assuming that $v$ tends to zero as $x\to \infty$) $v$ satisfies the PDE
\begin{equation}\label{mZKqv}
  	\partial_t v+\partial_{x_1}(\Delta v)+2(\partial_{x_1}v)^q =0.
\end{equation}
This new equation has some important advantages with respect to the original model. First, it conserves the Airy structure, and translates the spatial derivative inside de polynomial. The term $F(p(x)\sin (2\alpha(x_1+m t)))$ is the heart of the quasimonochromatic structure. We will introduce Bell's polynomials (Section \ref{Sec:2}) to translate this structure into new conditions that must be satisfied either by $F$ or $p$. These special polynomials are key to preserve the algebra of \eqref{mZKqv} thanks to its nice relation with the derivative of compositions of functions, known as the Fa\`a di Bruno's formula. With this structure on hand, every operation can be expressed in terms of new Bell's polynomials of higher order (Section \ref{sec:algebra}). After this is done, an important part of this work will be devoted to show that
\[
F^{(2n)}(0)=0,  \quad \hbox{and} \quad F^{(1)}(0)=0 \implies F^{(n)}(0)=0,  \quad n\geq 0.
\]
(Section \ref{sec:propa1}.) This precisely tells us that $F$ is odd and cannot have zero derivative at zero. A second part of the proof is devoted to show key elliptic PDEs satisfied by $p$. In the case $q$ even the situation is simple and one can quickly discard this case by proving that the only allowed nontrivial $p$ are singular.  The case $q$ odd requires care since no particular bad condition is present and the case $q=3$, $N=1$ is integrable. We will prove a strongly rigidity property satisfied by any solution $p$ along the $x_1$ variable:
\[
\partial_{1}^{(2)}p(x) + \left( \frac12m +4 \alpha^{2}\right)p(x)- 4 \alpha^{2}\left( \frac{F^{(3)}(0)}{F^{(1)}(0)} \right)p(x)^{3}=0.
\]
This almost ensures that $p$ is the reciprocal of a hyperbolic cosine solution. This will be done in the case $N=1$ after finding that for any dimension, the Laplacian for the $x_c= (x_2,\ldots,x_N)$ coordinates satisfies
\[
\Delta_c p(x)- \left( \frac12m+16\alpha^{2}  \right)p(x)  +16\alpha^{2}\left(  \frac{F^{(3)}(0)}{F^{(1)}(0)}\right)p(x)^{3}  +2^{q}\alpha^{q-1} \left(F^{(1)}(0)\right)^{q-1} p(x)^{q}  =0.
\]
These two new equations will be used to discard existence of quasimonochromatic breathers in the case $N\geq 2$, and the uniqueness of the mKdV breather in $N=1$. In the 1D case, the idea is to use the second equation above to quickly fix the parameters and conclude that $p$ is an hyperbolic cosine. In the case $N\geq 2$, we follow some of the ideas by Mandel \cite{RainerMandel2021} and use elliptic PDE theory to discard any localized solution for $p$.

\subsection*{Organization of this paper} This paper is organized as follows: In Section \ref{Sec:2} we provide the elementary techniques needed for the proof of main results, including Bell's polynomials. Next, Section \ref{sec:algebra} is devoted to the rigorous description of the algebra associated to Bell's polynomials. Section \ref{sec:propa1} describes the propagation of the quasimonochromatic structure of mKdV breathers. Section \ref{adicional} expands further this algebra considering additional properties held by the ZK-Bell mixed algebra. In Section \ref{N1} we prove Theorem \ref{MT1} in the odd case. Finally, Section \ref{Ngen} is devoted to the proof of Theorem \ref{MT2}, and Theorem \ref{MT1} in the even case.


\section{Preliminaries}\label{Sec:2}

\subsection{Bell's polynomials}
Let $n\in\N-\{0\}$, $k\in\N$ and $k\leq n$. For $x_{1},x_{2},\dots,x_{n-k+1}\in\R$, let us introduce the Bell's polynomials \cite{Bell}  of order $n$ and index $k$ as follows: 
\[
B_{n,0}(x_{1},x_{2},\dots,x_{n-k+1})=0,
\]
and
\begin{equation}\label{Bell}
\begin{aligned}
& B_{n,k}(x_{1},x_{2},\dots,x_{n-k+1})\\
& \quad :=\sum \frac{n!}{j_{1}!j_{2}!\dots j_{n-k+1}!}\left( \frac{x_{1}}{1!}\right)^{j_{1}} \left( \frac{x_{2}}{2!}\right)^{j_{2}}\dots \left( \frac{x_{n-k+1}}{(n-k+1)!}\right)^{j_{n-k+1}}, \quad k\geq 1,
\end{aligned}
\end{equation}
where the sum is taken on all the sequences $(j_{1},j_{2},\ldots,j_{n-k+1})$ of nonnegative integers such that
\begin{equation}\label{indices}
\begin{cases}
j_1+j_2+\cdots+j_{n-k+1}=k, \\
j_1+2j_2+ 3j_3+\cdots+ (n-k+1) j_{n-k+1}=n. 
\end{cases}
\end{equation}
Being $n$ and $k$ fixed, one denotes $({\bf x_{(v)}:v}) := (x_{1},x_{2},\dots,x_{n-k+1})$, where ${\bf v}={\bf v}(n,k)$. In short, one has
\begin{equation}\label{Bell_compacto}
	\begin{aligned}
		B_{n,k}(x_{1},x_{2},\dots,x_{n-k+1})=: &~{} B_{n,k}(\bf{x_{(v)}:v})\\
		= &~{} \sum \frac{n!}{j_{1}!j_{2}!\dots j_{n-k+1}!}\prod_{a=1}^{n-k+1}\left( \frac{x_{a}}{a!}\right)^{j_{a}},
	\end{aligned}
\end{equation}
with indices $(j_1,\ldots,j_{n-k+1})$ following \eqref{indices}. Some well-known Bell's polynomials needed along this paper are
\begin{equation}\label{bells}
\begin{aligned}
& B_{n,1}(x_1,\ldots,x_n)= x_n, \quad B_{2,2}(x_1)=x_1^2,\\
& B_{3,2}(x_1,x_2)=3x_1x_2, \quad B_{3,3}(x_1)=x_1^3,\\
& B_{4,2}(x_1,x_2,x_3)=3x_2^2+4x_1x_3, \quad B_{4,3}(x_1,x_2)=6x_1^2x_2, \quad B_{4,4}(x_1)=x_1^4.
\end{aligned}
\end{equation}
Notice that all the presented examples represent homogenous polynomials of degree $k$.

\medskip

Additionally, for $a\in\N$ and $j\in\{1,2,\ldots,N\}$, we denote
\[
\partial_j^{(a)} := \partial_{x_j}^a.
\]

\begin{lemma}[Computational properties of Bell's polynomials]\label{Bell_lema1}
Let $n\in\N-\{0\}$, $k\in\N$ and $k\leq n$. Let $B_{n,k}(\bf{x_{(v)}:v})$ be a Bell's polynomial of order $n$ and index $k$. Let $F:\mathbb R \to \mathbb R$ and $G:\mathbb R^N \times \mathbb R \mapsto \mathbb R$ be sufficiently smooth. Then the following derivative of the composition is satisfied: for any $j\in\{t,1,\ldots,N\}$,
\[
\partial_{j}^{(n)}( F^{(\ell)}(G(x,t))) =\sum_{k=1}^{n}F^{(\ell+k)}(G(x,t))B_{n,k} \left( {\bf \partial_{j}^{(v)}G(x,t): v} \right),
\]
where, 
\begin{equation}\label{vector_de_derivadas}
 \left( {\bf \partial_{j}^{(v)}G(x,t): v} \right) = \left( \partial_j^{(1)} G(x,t),   \partial_j^{(2)} G(x,t), \ldots,  \partial_j^{(n-k+1)} G(x,t) \right).
\end{equation}
\end{lemma}

\begin{proof}
The previous results are usually known as the Fa\`a di Bruno's formula. See \cite{DB1,DB2}.
\end{proof}

\subsection{Rewriting of the problem}
Recall that we are seeking for quasimochromatic solutions to gKdV and ZK models 
\begin{equation}\label{A1}
	\partial_t u+\partial_{x_1}\left( \Delta u +2u^{q}\right)=0, \quad N\geq 1, \quad q=2,3,\ldots
\end{equation}
of the form
\begin{equation}\label{A11}
	u(t,x)= \partial_{x_1}F(p(x)\sin (2\alpha(x_1+m t))),
\end{equation}
where $\alpha>0$ and $m\in\R$ are free parameters. We will always assume that $p$ is nontrivial, otherwise the solution is trivial. Hence, from \eqref{A1} and \eqref{A11}, and under the decaying condition on $p$ described in Definition \ref{mono}, we have that
\begin{equation}\label{def_v}
	v(t,x):=F(p(x)\sin (2\alpha(x_1+m t)))
\end{equation}
will satisfy the simplified equation
\begin{equation}\label{A2}
	\partial_t v+\partial_{x_1}(\Delta v)+2(\partial_{x_1}v)^q =0.
\end{equation}
From now on, we shall concentrate our efforts in showing uniqueness/nonexistence for the model \eqref{A2}. For simplicity, let 
\begin{equation}\label{G_s}
G(x,t): =p(x)s(x_{1},t), \qquad s(x,t):=s(x_{1},t) =\sin (2\alpha(x_1+m t)),
\end{equation}
(the change in the order $(t,x)$ to $(x,t)$ is motivated by the elliptic character of the forthcoming techniques). Finally, recall that $B_{n,k}:=B_{n,k}(\bf{x_{(v)}:v})$ denotes the Bell's polynomial of order $n$ and index $k$, as introduced in \eqref{Bell_compacto}.
\begin{lemma}\label{eq_simplificada_0}
Let $q\in\{1,2,\ldots\}$, and $\alpha>0$ and $m\in\R$ free parameters. Consider $v$ as in \eqref{def_v}, and $G$, $s$ as in \eqref{G_s}. Then the following expansion holds true: 
\begin{equation}\label{eq_simplificada}
	\begin{split}
		&\partial_t v+\partial_{x_1}(\Delta v)+2(\partial_{x_1}v)^q\\
		& = F^{(1)}(G(x,t))p(x)\partial^{(1)}_{t}s(x_{1},t)+2\left( F^{(1)}(G(x,t))\partial^{(1)}_{1}G(x,t)\right)^{q}\\
		&\quad +\sum_{k=1}^{3}F^{(k)}(G(x,t))B_{3,k}\left( \bf{\partial^{(v)}_{1}G(x,t):v} \right)\\
		&\quad +\sum_{j=2}^{N}\sum_{k=1}^{2}\sum_{\substack{a_{1}+a_{2}+a_{3}=1 \\ 0\leq a_1,a_2,a_3\leq 1}}	\partial^{a_{1}}_{1} \left( F^{(k)}(G(x,t)) \right)\partial^{a_{2}}_{1} \left( B_{2,k} \left( {\bf \partial^{(v)}_{j}p(x):v } \right) \right)\partial^{a_{3}}_{1}\left(s(x_{1},t)^{k} \right).
	\end{split}
\end{equation}
\end{lemma}

\begin{proof}
We compute using the Chain rule:
\begin{equation*}
	\begin{split}
		\partial_{t}v(t,x)&= \partial_{t}( F\circ G) (x,t)\\
		&= F^{(1)}(G(x,t))\partial^{(1)}_{t}G(x,t)= F^{(1)}(G(x,t))p(x)\partial^{(1)}_{t}s(x_{1},t).
	\end{split}
\end{equation*}
Next we will compute $\partial_{x_1}(\Delta v)$. If $j\neq 1$, then thanks to Lemma \ref{Bell_lema1} and the fact that Bell polynomials are homogenous of degree k,
\begin{equation*}
	\begin{split}
		(\partial^{2}_{j}v)(t,x)&=\partial^{2}_{j}(F\circ G)(x,t)\\
		&=\sum_{k=1}^{2}F^{(k)}(G(x,t))B_{2,k}\left( {\bf \partial^{(v)}_{j}G(x,t):v } \right)\\
		&=\sum_{k=1}^{2}F^{(k)}(G(x,t))B_{2,k}\left( {\bf \partial^{(v)}_{j}p(x):v } \right)s(x_{1},t)^{k}.
	\end{split}
\end{equation*}
Hence, using again Lemma \ref{Bell_lema1}, 
\[
\begin{aligned}
& \partial_{1}\partial^{2}_{j} (F\circ G)(x,t) \\
 &~{} = \sum_{k=1}^{2} \partial_{1} \left( F^{(k)}(G(x,t))B_{2,k}\left( {\bf \partial^{(v)}_{j}p(x):v } \right) s(x_{1},t)^{k} \right)\\
 &~{} = \sum_{k=1}^{2}\sum_{a_{1}+a_{2}+a_{3}=1}	\partial^{a_{1}}_{1} \left( F^{(k)}(G(x,t)) \right)\partial^{a_{2}}_{1}\left( B_{2,k} \left( {\bf \partial^{(v)}_{j}p(x,t):v } \right) \right)\partial^{a_{3}}_{1} \left(s(x_{1},t)^{k} \right).
\end{aligned}
\]
On the other hand,
\begin{equation*}
	\begin{split}
		\partial^{3}_{1}v(t,x)&= \partial^{3}_{1}(F\circ G)(x,t)\\
		&=\sum_{k=1}^{3}F^{(k)}(G(x,t))B_{3,k} \left( {\bf \partial^{(v)}_{1}G(x,t):v} \right).
	\end{split}
\end{equation*}
Finally, by classical chain rule,
\begin{equation*}
		(\partial_{1}v(x,t))^{q} = (\partial_{1}(F\circ G) (x,t))^{q}=\left( F^{(1)}(G(x,t))\partial^{(1)}_{1}G(x,t)\right)^{q}.
\end{equation*}
Gathering the previous computations, and replacing, we get \eqref{eq_simplificada}.
\end{proof}

In the following sections, we will study the consequences associated to the identity \eqref{eq_simplificada}.

\section{Algebra of Bell's polynomials}\label{sec:algebra}

The purpose of this section is to describe how the gKdV-ZK quasimonochromatic breather algebra interacts with Bell's polynomials. The symbol
\[
M(x,t)\Big|_{(x,t_{0})}
\]
indicates evaluation of the function $M$ at the particular point $(x,t_{0})\in\mathbb R^2$.

\subsection{Basic properties} Our first result concerns the product law evaluated at a sequence of particular times. 

\begin{lemma}\label{product_law0}
Let $a,b$ be nonnegative integers, and $t_{0}\in\mathbb{R}$  such that $s(x_{1},t_{0})=0$. Then the following identities are satisfied:
\begin{enumerate}
	\item[$(i)$] If $a\leq h$ and   $a$ and $h$ have the same parity,
	\begin{equation}\label{product_law00}
		\begin{split}
			&\partial_{1}^{(h)}\left(s(x_{1},t)^{a}\cdot \partial_{t}s(x_{1},t)^{b} \right) \Big|_{(x,t_{0})} = \left(\partial_{1}s(x_{1},t_{0}) \right)^{a+b} A^{1}_{a,b,h},
		\end{split}
	\end{equation}
	where $A^{1}_{a,b,h}$ is the explicit constant
	\begin{equation}\label{A1b}
A^{1}_{a,b,h}= (-1)^{ \frac12(h-a)}(2\alpha)^{h-a}\sum_{g=a}^{h}\sum_{\sum_{i=1}^{a} c_{i}=g \atop \sum_{j=1}^{b} d_{j}=h-g}\binom{h}{c_{1},\dots,c_{a},d_{1},\dots,d_{b}},
	\end{equation}
and where the sum extends over all $a$-tuples $(c_{1},\dots,c_{a})$ and $b$-tuples $(d_{1},\dots,d_{b})$   of non-negative integers  with
\begin{equation*}
	c_{1}+\dots+c_{a}=g \quad \hbox{ and } \quad	d_{1}+\dots+d_{b}=h-g
\end{equation*}
such that
\begin{equation*}
	c_{j} \mbox{ is odd } \forall j\in \{ 1,\dots,a\}, \quad  \mbox{ and } \quad d_{j} \mbox{ is even } \forall j\in \{ 1,\dots,b\}.
\end{equation*}	
	
	\item[$(ii)$]  In any other case, 
	\[
	\partial_{1}^{(h)}\left(s(x_{1},t)^{a}\cdot \partial_{t}s(x_{1},t)^{b} \right) \Big|_{(x,t_{0})}=0.
	\]
\end{enumerate} 
\end{lemma}

\begin{proof}
Let $a,b$ be nonnegative integers. First, recall that by the classical Leibniz rule,
\begin{equation}\label{ejemplo1}
\partial_{1}^{(g)} \left( s(x_{1},t)^{a} \right) =\sum_{\sum_{i=1}^{a} c_{i}=g}\binom{g}{c_{1},\dots,c_{a}} \prod_{i=1}^{a}\partial_{1}^{(c_{i})}s(x_{1},t) ,
\end{equation}
where the sum extends over all $a$-tuples $(c_{1},\dots,c_{a})$ of non-negative integers with $c_{1}+\dots+c_{a}=g$.  Similarly, for $h\geq g,$
\[
\partial_{1}^{(h-g)} \left( \partial_{t}s(x_{1},t)^{b} \right) =\sum_{\sum_{j=1}^{b} d_{j}=h-g}\binom{h-g}{d_{1},\dots,d_{b}}\prod_{j=1}^{b}\partial_{1}^{(d_{j})}\partial_{t}s(x_{1},t).
\]
Here the sum extends over all $b$-tuples $(d_{1},\dots,d_{b})$ of non-negative integers with $d_{1}+\dots+d_{b}=h-g$.  Since $a,b$ are nonnegative integers, and using again the classical 1D Leibniz formula
\begin{equation}\label{intermedio}
	\begin{split}
		\partial_{1}^{(h)}\left(s(x_{1},t)^{a}\cdot \partial_{t}s(x_{1},t)^{b} \right) = &~{} \sum_{g=0}^{h}\binom{h}{g} \partial_{1}^{(g)} \left( s(x_{1},t)^{a} \right) \partial_{1}^{(h-g)} \left( \partial_{t}s(x_{1},t)^{b} \right)\\
		=&~{}\sum_{g=0}^{h}\binom{h}{g}\left(\sum_{\sum_{i=1}^{a} c_{i}=g}\binom{g}{c_{1},\dots,c_{a}} \prod_{i=1}^{a}\partial_{1}^{(c_{i})}s(x_{1},t) \right) \\
		& \qquad \times \left( \sum_{\sum_{j=1}^{b} d_{j}=h-g}\binom{h-g}{d_{1},\dots,d_{b}}\prod_{j=1}^{b}\partial_{1}^{(d_{j})}\partial_{t}s(x_{1},t)\right).
	\end{split}
\end{equation}
Let $t_{0}$ is such that $s(x_{1},t_{0})=0$. Let us assume that  particular $a$-tuples $(c_{1},\dots,c_{a})$ and $b$-tuples $(d_{1},\dots,d_{b})$  with
\begin{equation*}
	c_{1}+\dots+c_{a}=g \quad \mbox{ and } 	\quad d_{1}+\dots+d_{b}=h-g
\end{equation*}
are such that

\begin{equation*}
	\prod_{i=1}^{a}\partial_{1}^{(c_{i})}s(x_{1},t_{0})\neq0,\quad \mbox{ and } \quad \prod_{j=1}^{b}\partial_{1}^{(d_{j})}\partial_{t}s(x_{1},t_{0})\neq0 .  
\end{equation*}
Then we have that
\begin{equation*}
	 c_{j} \mbox{ is odd } \forall j\in \{ 1,\dots,a\}, \quad  \mbox{ and } \quad d_{j} \mbox{ is even } \forall j\in \{ 1,\dots,b\}.
\end{equation*}
So, if we write $c_{i}=2\ell_{c_{i}}+1$ and $d_{j}=2\ell_{d_{j}}$, we have:
\[
	\sum_{i=1}^{a}c_{i}=\sum_{i=1}^{a}(2\ell_{c_{i}}+1)=\sum_{i=1}^{a}(2\ell_{c_{i}})+a=g,
\]
and
\[
\sum_{j=1}^{b}d_{j}=2\ell_{d_{j}}=h-g.
\]
In particular, $g-a$ and $h-g$ are even numbers  and $a$ and $h$ have the same parity. Moreover,
\begin{equation*}
	\begin{split}
		\prod_{i=1}^{a}\partial_{1}^{(c_{i})}s(x_{1},t)= &	\prod_{i=1}^{a}\partial_{1}^{(2\ell_{c_{i}} +1)}s(x_{1},t) \\
		=& 	\prod_{i=1}^{a}\partial_{1} \left( \partial_{1}^{(2\ell{c_{i}} )}s(x_{1},t) \right) \\
		=&	\prod_{i=1}^{a}(-1)^{\ell_{c_{i}}}(2\alpha)^{2\ell_{c_{i}}}\partial_{1}s(x_{1},t)\\
		= &~{} \left(\partial_{1}s(x_{1},t) \right)^{a} \prod_{i=1}^{a}(-1)^{\ell_{c_{i}}}(2\alpha)^{2\ell_{c_{i}}}= \left(\partial_{1}s(x_{1},t) \right)^{a} (-1)^{ (g-a)/2}(2\alpha)^{g-a},
	\end{split}	
\end{equation*}
and 
\begin{equation*}
	\begin{split}
		\prod_{j=1}^{b}\partial_{1}^{(d_{j})}\partial_{t}s(x_{1},t)= 	& \prod_{j=1}^{b}\partial_{t}(\partial_{1}^{(2\ell_{d_{j}})}s(x_{1},t)) \\
		=& 	\prod_{j=1}^{b} (-1)^{\ell_{d_{j}}}(2\alpha)^{2\ell_{d_{j}}}\partial_{t}s(x_{1},t) \\
		= &  ~{}\left( \partial_{t}s(x_{1},t) \right)^{b}\prod_{j=1}^{b} (-1)^{\ell_{d_{j}}}(2\alpha)^{2\ell_{d_{j}}}=\left(\partial_{t}s(x_{1},t)\right)^{b}(-1)^{ (h-g)/2}(2\alpha)^{h-g}.
	\end{split}
\end{equation*}
Hence, from \eqref{intermedio}, if either $a > h$ or $a$ and $h$ have different parity, then we have that
\begin{equation*}
	\begin{split}
		\partial_{1}^{(h)}\left(s(x_{1},t)^{a}\cdot \partial_{t}s(x_{1},t)^{b} \right) =0.
	\end{split}
\end{equation*}
In general, using \eqref{intermedio}, 
\begin{equation*}
	\begin{split}
		&\partial_{1}^{(h)}\left(s(x_{1},t)^{a}\cdot \partial_{t}s(x_{1},t)^{b} \right) \Big|_{(x,t_{0})} \\
		&=\sum_{g=a}^{h}\binom{h}{g}\left(\sum_{\sum_{i=1}^{a} c_{i}=g}\binom{g}{c_{1},\dots,c_{a}}  \partial_{1}s(x_{1},t)^{a} (-1)^{ (g-a)/2}(2\alpha)^{g-a}\right)\times \\
		&\quad \left(\sum_{\sum_{j=1}^{b} d_{j}=h-g}\binom{h-g}{d_{1},\dots,d_{b}} \partial_{t}s(x_{1},t_{0})^{b}(-1)^{ (h-g)/2}(2\alpha)^{h-g}\right) \\
		&=\left(\partial_{1}s(x_{1},t_{0}) \right)^{a+b}\left( (-1)^{ (h-a)/2}(2\alpha)^{h-a}\sum_{g=a}^{h}\sum_{\sum_{i=1}^{a} c_{i}=g \atop \sum_{j=1}^{b} d_{j}=h-g}\binom{h}{c_{1},\dots,c_{a},d_{1},\dots,d_{b}} \right).
	\end{split}
\end{equation*}
where the sum extends over all $a$-tuples $(c_{1},\dots,c_{a})$ and $b$-tuples $(d_{1},\dots,d_{b})$  with
\begin{equation*}
	c_{1}+\dots+c_{a}=g\quad  \mbox{ and } 	\quad d_{1}+\dots+d_{b}=h-g
\end{equation*}
such that
\begin{equation*}
	c_{j} \mbox{ is odd } \forall j\in \{ 1,\dots,a\},\quad \mbox{ and } \quad d_{j} \mbox{ is even } \forall j\in \{ 1,\dots,b\}.
\end{equation*}
This proves \eqref{product_law00} with $A^{1}_{a,b,h}$ as in \eqref{A1b}.
\end{proof}

\begin{lemma}\label{PEB}
	Let $t_{0}$ be any time such that $s(x_{1},t_{0})=0$. One has 
	\begin{equation}\label{PEB_0}
		B_{n,k} \left( {\bf\partial_{t}^{(v)} s(x_{1},t_{0}) :  v} \right)= A_{n,k}^{0}\partial_{t}s(x_{1},t_{0})^{k},
	\end{equation}
where $A_{n,k}^{0}$ is a constant which does not depend on $x$ and $t_{0}$.  Moreover,
\begin{enumerate}
\item[$(i)$] If $n-k$ is odd then $A_{n,k}^{0}=0$.
\item[$(ii)$] If $n-k$ is even,  $A_{n,k}^{0}\neq 0$ and it is given by
\begin{equation}\label{A_nk}
A_{n,k}^{0}:=(-1)^{(n-k)/2}(2\alpha m)^{n-k}\sum \frac{n!}{j_{1}!j_{2}!\cdots j_{n-k+1}!}\left(\prod_{i=0}^{\lfloor \frac{n-k}{2} \rfloor}\left( (2i+1)!\right)^{-j_{2i+1}}\right),
\end{equation}
where the sum is taken over all sequences $j_{1},\dots,j_{n-k+1}$  of nonnegative integers which satisfy the constitutive properties \eqref{indices} and also satisfy 
\begin{equation}\label{suma0k}
 \sum_{i=1}^{\lfloor \frac{n-k+1}{2} \rfloor} j_{2i}=0 \quad \mbox{ and } \quad \sum_{i=0}^{\lfloor \frac{n-k}{2} \rfloor}j_{2i+1}=k.
\end{equation}
\end{enumerate}
\end{lemma}

\begin{proof}
From \eqref{Bell} and \eqref{Bell_compacto},
\begin{equation*}
\begin{split}
&B_{n,k} \left( {\bf\partial_{t}^{(v)} s(x_{1},t): v} \right)\\
&=\sum \frac{n!}{j_{1}!j_{2}!\cdots j_{n-k+1}!}\left(\frac{\partial^{(1)}_{t} s(x_{1},t)}{1!} \right)^{j_{1}}\left(\frac{\partial^{(2)}_{t} s(x_{1},t)}{2!} \right)^{j_{2}}\cdots \left(\frac{\partial^{(n-k+1)}_{t} s(x_{1},t)}{(n-k+1)!} \right)^{j_{(n-k+1)}} \\
&=\sum\frac{n!}{j_{1}!j_{2}!\cdots j_{n-k+1}!}\prod_{i=1}^{n-k+1}\left(\frac{\partial^{(i)}_{t} s(x_{1},t)}{i!} \right)^{j_{i}},
\end{split}
\end{equation*}
where the sum is taken following \eqref{indices}. Splitting in even and odd derivatives, and recalling \eqref{G_s},
\begin{equation*}
\begin{split}
&B_{n,k} \left( {\bf\partial_{t}^{(v)} s(x_{1},t): v} \right)\\
&=\sum \frac{n!}{j_{1}!j_{2}!\cdots j_{n-k+1}!}\prod_{i=1}^{\lfloor \frac{n-k+1}{2} \rfloor}\left(\frac{\partial^{(2i)}_{t} s(x_{1},t)}{(2i)!} \right)^{j_{2i}}\prod_{i=0}^{\lfloor \frac{n-k}{2} \rfloor}\left(\frac{\partial^{(2i+1)}_{t} s(x_{1},t)}{(2i+1)!} \right)^{j_{2i+1}}\\
&=\sum \frac{n!}{j_{1}!j_{2}!\cdots j_{n-k+1}!}\prod_{i=1}^{\lfloor \frac{n-k+1}{2} \rfloor}\left(\frac{(-1)^{i}(2\alpha m)^{2i}s(x_{1},t)}{(2i)!} \right)^{j_{2i}}\prod_{i=0}^{\lfloor \frac{n-k}{2} \rfloor}\left(\frac{(-1)^{i}(2\alpha m)^{2i}\partial_{t}s(x_{1},t) }{(2i+1)!} \right)^{j_{2i+1}}\\
&=\sum \frac{n!}{j_{1}!j_{2}!\cdots j_{n-k+1}!}\left( \prod_{i=1}^{\lfloor \frac{n-k+1}{2} \rfloor}\left(\frac{(-1)^{i}(2\alpha m)^{2i}}{(2i)!} \right)^{j_{2i}}\prod_{i=0}^{\lfloor \frac{n-k}{2} \rfloor}\left(\frac{(-1)^{i}(2\alpha m)^{2i}}{(2i+1)!} \right)^{j_{2i+1}}\right) \times\\
& \qquad \hspace{3.5cm}s(x_{1},t)^{	S_{j_{i}}^{e}}\cdot \partial_{t}s(x_{1},t)^{S_{j_{i}}^{o} }.
\end{split}
\end{equation*}
Here 	
\begin{eqnarray*}
	S_{j_{i}}^{e}:=\sum_{i=1}^{\lfloor \frac{n-k+1}{2} \rfloor} j_{2i} \quad \mbox{ and } \quad S_{j_{i}}^{o}:=\sum_{i=0}^{\lfloor \frac{n-k}{2} \rfloor}j_{2i+1}.
\end{eqnarray*}
So, if $t_{0}\in\mathbb{R}$ is  such that $s(x_{1},t_{0})=0$, then a term in the previous sum is different from zero if and only if $j_{1},\dots,j_{n-k+1}$ are such that 
\begin{equation*}
	S_{j_{i}}^{e}=0, \quad  S_{j_{i}}^{o}=k \quad \mbox{ and } \quad \sum_{i=0}^{\lfloor \frac{n-k}{2} \rfloor}(2i+1)j_{2i+1}=n.
\end{equation*}
This is nothing but \eqref{suma0k}. Therefore
\begin{eqnarray*}
2\sum_{i=0}^{\lfloor \frac{n-k}{2} \rfloor}ij_{2i+1}+\sum_{i=0}^{\lfloor \frac{n-k}{2} \rfloor}j_{2i+1}=n,
\end{eqnarray*}
i.e.
\begin{eqnarray*}
	2\sum_{i=0}^{\lfloor \frac{n-k}{2} \rfloor}ij_{2i+1}=n-k.
\end{eqnarray*}
We conclude that if  $n-k$  is odd then $B_{n,k}\left( {\bf\partial_{t}^{(v)} s(x_{1},t): v} \right) \Big|_{(x,t_{0})}= 0$. It follows that
\begin{equation*}
	\begin{split}
		&	B_{n,k} \left( {\bf\partial_{t}^{(v)} s(x_{1},t): v} \right) \Big|_{(x,t_{0})}\\
		&=\sum \frac{n!}{j_{1}!j_{2}!\cdots j_{n-k+1}!}\left(\prod_{i=0}^{\lfloor \frac{n-k}{2} \rfloor}\left(\frac{(-1)^{i}(2\alpha m)^{2i}}{(2i+1)!} \right)^{j_{2i+1}}\right)\partial_{t}s(x_{1},t_{0})^{S_{j_{i}}^{o} }\\
		&=\partial_{t}s(x_{1},t_{0})^{k }(-1)^{(n-k)/2}(2\alpha m)^{n-k}\sum \frac{n!}{j_{1}!j_{2}!\cdots j_{n-k+1}!}\left(\prod_{i=0}^{\lfloor \frac{n-k}{2} \rfloor}\left( (2i+1)!\right)^{-j_{2i+1}}\right).
	\end{split}
\end{equation*}
This concludes the proof of \eqref{PEB_0} after checking that $A^0_{n,k}$ is given by \eqref{A_nk}.	
\end{proof}
The following result will be not necessary, but it is left for future computations.
\begin{corollary}
	Let $t_{0}$  such that $s(x_{1},t_{0})=0$, then 
\begin{equation*}
	B_{n,k}({\bf\partial_{1}^{(v)} s(x_{1},t_{0}): v} )=A_{1,n,k}\partial_{1}s(x_{1},t_{0})^{k},
\end{equation*}
where $A_{1,n,k}$ is a constant which does not depend on $x$ and $t_{0}$.  
\end{corollary}

Our next result concerns the calculation of terms of the form $\partial_{1}^{(h)}( B_{n,k}({\bf \partial_{t}s(x_{1},t)}))$. 
{\color{black}
\begin{lemma}\label{ADhBts}
	Let  $h$ a positive integer  and $t_{0}$  such that $s(x_{1},t_{0})=0$ , then 
	\begin{equation}\label{ADhBts_0}
	\partial_{1}^{(h)} \left(  B_{n,k}({\bf \partial_{t}s(x_{1},t) }) \right) \Big|_{(x,t_{0})} =A^{h}_{n,k} \left( \partial_{t}s(x_{1},t_{0}) \right)^{k},
	\end{equation}
where $A^{h}_{n,k}$ is a constant which does not depend on $x$ and $t_{0}$. Moreover,  
\begin{itemize}
\item if $n$  is even we have that $A^{h}_{n,k}=0$ if $k$ and $h$ have different parity.
\item If $n$ is  odd then $A^{h}_{n,k}=0$ if $k$ and $h$ have the same parity.
\end{itemize}
\end{lemma}
\begin{proof}
The proof is similar to that of Lemma \ref{PEB}. From \eqref{Bell_compacto} and splitting even and odd derivatives,
\begin{equation*}
\begin{split}
&	B_{n,k} \left( {\bf\partial_{t}^{(v)} s(x_{1},t): v} \right)\\
&=\sum\frac{n!}{j_{1}!j_{2}!\cdots j_{n-k+1}!}\prod_{i=1}^{n-k+1}\left(\frac{\partial^{(i)}_{t} s(x_{1},t)}{i!} \right)^{j_{i}} \\
&=\sum \frac{n!}{j_{1}!j_{2}!\cdots j_{n-k+1}!}\prod_{i=1}^{\lfloor \frac{n-k+1}{2} \rfloor}\left(\frac{\partial^{(2i)}_{t} s(x_{1},t)}{(2i)!} \right)^{j_{2i}}\prod_{i=0}^{\lfloor \frac{n-k}{2} \rfloor}\left(\frac{\partial^{(2i+1)}_{t} s(x_{1},t)}{(2i+1)!} \right)^{j_{2i+1}}.
\end{split}
\end{equation*}
where the sum is taken over all sequences $j_{1},\dots,j_{n-k+1}$  of nonnegative integers which satisfy the constitutive properties \eqref{indices} . From \eqref{G_s},
\begin{equation*}
\begin{split}
&B_{n,k} \left( {\bf\partial_{t}^{(v)} s(x_{1},t): v} \right)\\
&=\sum \frac{n!}{j_{1}!j_{2}!\cdots j_{n-k+1}!}\prod_{i=1}^{\lfloor \frac{n-k+1}{2} \rfloor}\left(\frac{(-1)^{i}(2\alpha m)^{2i}s(x_{1},t)}{(2i)!} \right)^{j_{2i}}\prod_{i=0}^{\lfloor \frac{n-k}{2} \rfloor}\left(\frac{(-1)^{i}(2\alpha m)^{2i}\partial_{t}s(x_{1},t) }{(2i+1)!} \right)^{j_{2i+1}}\\
&=\sum \frac{n!}{j_{1}!j_{2}!\cdots j_{n-k+1}!}\prod_{i=1}^{\lfloor \frac{n-k+1}{2} \rfloor}\left(\frac{(-1)^{i}(2\alpha m)^{2i}}{(2i)!} \right)^{j_{2i}} \\
& \qquad \qquad \qquad\qquad\qquad \times \prod_{i=0}^{\lfloor \frac{n-k}{2} \rfloor}\left(\frac{(-1)^{i}(2\alpha m)^{2i} }{(2i+1)!} \right)^{j_{2i+1}}s(x_{1},t)^{	S_{j_{i}}^{e}}\cdot \left(\partial_{t}s(x_{1},t) \right)^{S_{j_{i}}^{o} }\\
&=\sum J_{1,\dots,(n-k+1)}s(x_{1},t)^{	S_{j_{i}}^{e}}\cdot \left(\partial_{t}s(x_{1},t) \right)^{S_{j_{i}}^{o} }.
\end{split}
\end{equation*}
Here
\begin{eqnarray*}
S_{j_{i}}^{e}:=\sum_{i=1}^{\lfloor \frac{n-k+1}{2} \rfloor} j_{2i} \quad \mbox{ and } \quad S_{j_{i}}^{o}:=\sum_{i=0}^{\lfloor \frac{n-k}{2} \rfloor}j_{2i+1}.
\end{eqnarray*}
Additionally, $J_{1,\dots,(n-k+1)}$ is the explicit constant
\begin{eqnarray*}
	J_{1,\dots,(n-k+1)}:=\frac{n!}{j_{1}!j_{2}!\cdots j_{n-k+1}!}\prod_{i=1}^{\lfloor \frac{n-k+1}{2} \rfloor}\left(\frac{(-1)^{i}(2\alpha m)^{2i}}{(2i)!} \right)^{j_{2i}}\prod_{i=0}^{\lfloor \frac{n-k}{2} \rfloor}\left(\frac{(-1)^{i}(2\alpha m)^{2i} }{(2i+1)!} \right)^{j_{2i+1}}.
\end{eqnarray*}
Hence
\begin{equation*}
	\begin{split}
		\partial_{1}^{(h)}B_{n,k}\left( {\bf\partial_{t}^{(v)} s(x_{1},t): v} \right) = \sum J_{1,\dots,(n-k+1)} \partial_{1}^{(h)}\left( s(x_{1},t)^{	S_{j_{i}}^{e}}\cdot \partial_{t}s(x_{1},t)^{S_{j_{i}}^{o} }\right).
	\end{split}
\end{equation*}	
From Lemma \ref{product_law0},
\begin{equation*}
	\partial_{1}^{(h)}\left(s(x_{1},t)^{	S_{j_{i}}^{e}}\cdot \partial_{t}s(x_{1},t)^{S_{j_{i}}^{o} } \right) \Big|_{(x,t_{0})}=s(x_{1},t_{0})^{k} A^{1}_{S_{j_{i}}^{e},S_{j_{i}}^{o},h},
\end{equation*}
where 
\begin{eqnarray*}
A^{1}_{S_{j_{i}}^{e},S_{j_{i}}^{o},h}=0 \mbox{ if } S_{j_{i}}^{e}>h  \mbox{ or } S_{j_{i}}^{e}  \mbox{ and } h \mbox{ have different parity.}
\end{eqnarray*}
On the other hand, from  \eqref{indices}, we have that
\begin{equation*}
S_{j_{i}}^{e}+S_{j_{i}}^{o}=k  \quad \mbox{ and } \quad n=\sum_{i=1}^{n-k+1} ij_{i}=\sum_{i=1}^{\lfloor \frac{n-k+1}{2} \rfloor} 2ij_{2i} +\sum_{i=0}^{\lfloor \frac{n-k}{2} \rfloor}(2i+1)j_{2i+1},
\end{equation*}
which implies that
\begin{equation*}
n=2\sum_{i=1}^{\lfloor \frac{n-k+1}{2} \rfloor} ij_{2i} +2\sum_{i=0}^{\lfloor \frac{n-k}{2} \rfloor}ij_{2i+1}+\sum_{i=0}^{\lfloor \frac{n-k}{2} \rfloor}j_{2i+1}.
\end{equation*}
We conclude that $n$ and $S_{j_{i}}^{o}$ have the same parity.  Finally,  if $n$ is even, then $k$ and $S_{j_{i}}^{e}$ have the same parity. Therefore 
\begin{eqnarray*}
	A^{1}_{S_{j_{i}}^{e},S_{j_{i}}^{o},h}=0,
\end{eqnarray*}
if $k$ and $h$ have different parity. On the other hand,  if $n$ is odd, then $k$ and $S_{j_{i}}^{e}$ have the different parity. Hence
\begin{eqnarray*}
	A^{1}_{S_{j_{i}}^{e},S_{j_{i}}^{o},h}=0. 
\end{eqnarray*}
if $k$ and $h$ have the same parity. This concludes the proof of \eqref{ADhBts_0}.
\end{proof}

}

\begin{lemma}\label{lemma_B_DG}
	Let $t_{0}$  such that $s(x_{1},t_{0})=0$, then 
	\begin{equation}\label{lemma_B_DG_0}
		B_{n,k} \left(  {\bf\partial_{1}^{(v)}G(x,t_{0}): v} \right)=C_{n,k}(x)\partial_{1}s(x_{1},t_{0})^{k},
	\end{equation}
	where $C_{n,k}(x)$ does not depend on $t_{0}$ and it is defined in terms of Bell's polynomials as
\begin{equation}\label{C_nk}
\begin{aligned}	
		C_{n,k}(x):= &~{} B_{n,k} \left( {\bf   \sum_{\ell=0}^{\lfloor \frac{v-1}{2} \rfloor}\binom{v}{2\ell+1}\partial^{(v-(2\ell+1))}_{1}p(x)(-1)^{\ell}(2\alpha)^{2\ell} }: v \right)\\
	= &~{} \sum \frac{n!}{j_{1}!j_{2}!\dots j_{n-k+1}!}\prod_{a=1}^{n-k+1}\left( \frac{ \sum_{\ell=0}^{\lfloor \frac{a-1}{2} \rfloor}\binom{a}{2\ell+1}\partial^{(a-(2\ell+1))}_{1}p(x)(-1)^{\ell}(2\alpha)^{2\ell} }{a!}\right)^{j_{a}}.
\end{aligned}
\end{equation}	
\end{lemma}

\begin{proof}
	For any $a\in\mathbb{N}$ and any $(x,t)\in$ we have
	\begin{equation*}
		\begin{split}
			\partial^{(a)}_{1}G(x,t)& =\partial^{(a)}_{1}(p(x)s(x_{1},t))\\
			& =\sum_{b=0}^{a}\binom{a}{b}\partial^{(a-b)}_{1}p(x)\partial^{(b)}_{1}s(x_{1},t)\\
			&=\sum_{\ell=0}^{\lfloor \frac{a}{2} \rfloor}\binom{a}{2\ell}\partial^{(a-2\ell)}_{1}p(x)\partial^{(2\ell)}_{1}s(x_{1},t)+\sum_{\ell=0}^{\lfloor \frac{a-1}{2} \rfloor}\binom{a}{2\ell+1}\partial^{(a-(2\ell+1))}_{1}p(x)\partial^{(2\ell+1)}_{1}s(x_{1},t)\\
			&=\sum_{\ell=0}^{\lfloor \frac{a}{2} \rfloor}\binom{a}{2\ell}\partial^{(a-2\ell)}_{1}p(x)(-1)^{\ell}(2\alpha)^{2\ell}s(x_{1},t) \\
			& \quad +\sum_{\ell=0}^{\lfloor \frac{a-1}{2} \rfloor}\binom{a}{2\ell+1}\partial^{(a-(2\ell+1))}_{1}p(x)(-1)^{\ell}(2\alpha)^{2\ell}\partial_{1}s(x_{1},t).
		\end{split}
	\end{equation*}
	In particular, if  $t_{0}$ is such that $s(x_{1},t_{0})=0$ we have
	\begin{equation}\label{corcho}
			\partial^{(a)}_{1}G(x,t_{0})=\sum_{\ell=0}^{\lfloor \frac{a-1}{2} \rfloor}\binom{a}{2\ell+1}\partial^{(a-(2\ell+1))}_{1}p(x)(-1)^{\ell}(2\alpha)^{2\ell}\partial_{1}s(x_{1},t_{0}).
	\end{equation}
	Hence, by definition of Bell's polynomials \eqref{Bell_compacto}, and replacing \eqref{corcho},
	\begin{equation*}
		\begin{split}
			&B_{n,k}( {\bf\partial_{1}^{(v)}G(x,t_{0}): v})\\
			&=\sum \frac{n!}{j_{1}!j_{2}!\dots j_{n-k+1}!}\prod_{a=1}^{n-k+1}\left( \frac{ \partial^{(a)}_{1}G(x,t_{0})}{a!}\right)^{j_{a}}\\
			&=\sum \frac{n!}{j_{1}!j_{2}!\dots j_{n-k+1}!}\prod_{a=1}^{n-k+1}\left( \frac{ \sum_{\ell=0}^{\lfloor \frac{a-1}{2} \rfloor}\binom{a}{2\ell+1}\partial^{(a-(2\ell+1))}_{1}p(x)(-1)^{\ell}(2\alpha)^{2\ell}\partial_{1}s(x_{1},t_{0}) }{a!}\right)^{j_{a}}\\
			&=\partial_{1}s(x_{1},t_{0})^{k} \left( \sum \frac{n!}{j_{1}!j_{2}!\dots j_{n-k+1}!}\prod_{a=1}^{n-k+1}\left( \frac{ \sum_{\ell=0}^{\lfloor \frac{a-1}{2} \rfloor}\binom{a}{2\ell+1}\partial^{(a-(2\ell+1))}_{1}p(x)(-1)^{\ell}(2\alpha)^{2\ell} }{a!}\right)^{j_{a}} \right).
		\end{split}
	\end{equation*}
Recalling \eqref{C_nk}, one gets \eqref{lemma_B_DG_0}.
\end{proof}

\subsection{An extended Chain Rule} We finish this section with an important result computing the action of the Chain rule in a particular evaluation of the ZK quasimonochromatic algebra.

\begin{lemma}\label{ADf}
	Let $h,g\geq 1$, $\ell\geq0$ be fixed integers, and let $t_{0}$ be such that $s(x_{1},t_{0})=0$. Then 
	\begin{equation}\label{ADf_0}
	\partial_{1}^{(h)} \left( F^{(\ell)}(G(x,t)) \right) \Big|_{(x,t_{0})} =\sum_{k=1}^{h}F^{(\ell+k)}(0)C_{h,k}(x)\partial_{1}s(x_{1},t_{0})^{k},
	\end{equation}
	where $C_{h,k}(x)$ are coefficients only depending on $x$. Additionally, 
	\begin{equation}\label{ADf_1}
	\partial_{j}^{(g)}( F^{(\ell)}(G(x,t))) \Big|_{(x,t_{0})} = \left\{
	\begin{array}{c l}
		F^{(\ell)}(0) & \mbox{ if  } g=0 \\
		0 & \mbox{ if } g>0,
	\end{array}
	\right.
	\end{equation}
and {\color{black}
\begin{equation}\label{ADf_2}
\begin{aligned}
	&~{}  \partial_{1}\partial_{j}^{(g)} \left(  F^{(\ell)}(G(x,t)) \right) \Big|_{(x,t_{0})} ~{} = \left\{
	\begin{array}{c l}
		F^{(\ell+1)}(0)p(x)\partial_{1}s(x,t_{0}) & \mbox{ if  } g=0 \\
		F^{(\ell+1)}(0)\partial_{1}s(x,t_{0})B_{g,1}({\bf \partial_{j}^{(v)}p(x)}: {\bf v}) & \mbox{ if }g>0.
	\end{array}
	\right.
	\end{aligned}
	\end{equation}
}
\end{lemma}

\begin{proof}
From Lemma \ref{Bell_lema1},
	\[
	\partial_{1}^{(h)}( F^{(\ell)}(G(x,t))) =\sum_{k=1}^{h}F^{(\ell+k)}(G(x,t))B_{h,k} \left( {\bf \partial_{1}^{(v)}G(x,t): v} \right).
	\]
Thanks to Lemma \ref{lemma_B_DG}, one has
	\begin{equation*}
		B_{h,k} \left(  {\bf\partial_{1}^{(v)}G(x,t): v} \right) \Big|_{(x,t_{0})}=C_{h,k}(x)\partial_{1}s(x_{1},t_{0})^{k}.
	\end{equation*}
	Gathering the two previous identities, and using the fact that from \eqref{G_s} one has $G(x,t)\Big|_{(x,t_0)}=0$, we get \eqref{ADf_0}.
	
	\medskip
	
	Note that for $g>0$, we have
	\begin{equation}\label{split1}
		\begin{split}
			\partial_{j}^{(g)}( F^{(\ell)}(G(x,t)))& =\sum_{k=1}^{g}F^{(\ell+k)}(G(x,t))s(x_{1},t)^{k}B_{g,k} \left( {\bf \partial_{j}^{(v)}p(x): (v)} \right)\\
			\partial_{1}\partial_{j}^{(g)}( F^{(\ell)}(G(x,t)))& =\sum_{k=1}^{g}\sum_{b_{1}+b_{2}+b_{3}=1}\partial_{1}^{(b_{1})}\left( F^{(\ell+k)}(G(x,t))\right)  \partial_{1}^{(b_{2})}\left( s(x_{1},t)^{k}\right)\times \\
			&\hspace{3cm}  \partial_{1}^{(b_{3})}\left( B_{g,k} \left( {\bf \partial_{j}^{(v)}p(x): v} \right)\right).
		\end{split}
	\end{equation}
	Hence, from the first equation in \eqref{split1},
	\[
	\partial_{j}^{(g)} \left( F^{(\ell)}(G(x,t)) \right) \Big|_{(x,t_{0})} = \left\{
	\begin{array}{c l}
		F^{(\ell)}(0) & \mbox{ if  } g=0 \\
		0 & \mbox{ if } g>0,
	\end{array}
	\right.
	\]
	proving \eqref{ADf_1}. {\color{black} Finally, if $\{b_{1},b_{2},b_{3}\}$ are nonnegative integers with $b_{1}+b_{2}+b_{3}=1$, we have that
	\[
	\begin{aligned}
	&~{}  \partial_{1}^{(b_{1})}\left( F^{(\ell+k)}(G(x,t))\right)  \partial_{1}^{(b_{2})}\left( s(x_{1},t)^{k}\right) \partial_{1}^{(b_{3})}\left( B_{g,k} \left( {\bf \partial_{j}^{(v)}p(x): v} \right)\right)\Big|_{(x,t_{0})} \\
	&~{} = \left\{
	\begin{array}{c l}
		0 & \mbox{ if } b_{2}=0 \\
		0 & \mbox{ if } b_{2}=1  \mbox{ and } k>1 \\
		F^{(\ell+1)}(0)\partial_{1}s(x,t_{0})B_{g,1}({\bf \partial_{j}^{(v)}p(x)}: {\bf v}) & \mbox{ if } b_{2}=1  \mbox{ and } k=1.
	\end{array}
	\right.
	\end{aligned}
	\]
Hence \eqref{ADf_2}	follows from \eqref{split1}.}
\end{proof}

\section{Propagation of the quasimonochromatic structure}\label{sec:propa1}

The purpose of this section is to propagate the results obtained in Lemma \ref{eq_simplificada_0} in the case of some particular space-time parameters. As in the previous section, the symbol
\[
M(x,t)\Big|_{(x,t_{0})}
\]
indicates evaluation of the function $M$ at the particular point $(x,t_{0})\in\mathbb R^{N}\times \mathbb R$. We first start propagating Lemma \ref{eq_simplificada_0}:

\begin{lemma}[Flat second derivative of $F$]\label{Corf1}
Let $v$ of the form \eqref{def_v} be a smooth solution of \eqref{A2}. Then it holds true that:
\begin{equation}\label{sec1}
	\begin{split}
		&4\alpha\left(F^{(1)}(0)\left( mp(x)-p(x)(2\alpha)^{2}+ \Delta p(x)+ 2\partial_{1}^{(2)}p(x)\right) \right) \\
		&\quad +16\alpha^{3}\left( F^{(3)}(0)p(x)^{3} \right)+ 2\left(1-(-1)^{q} \right)(2\alpha)^{q} \left(F^{(1)}(0) \right)^{q}p(x)^{q} =0,
	\end{split}
\end{equation}
and
\begin{equation}\label{sec2}
	8\alpha^{2}\left(6F^{(2)}(0)p(x)\partial_{1}p(x) \right)+ 2\big(1+(-1)^{q}\big)(2\alpha)^{q}\left(F^{(1)}(0) \right)^{q} p(x)^{q} =0.
\end{equation}
Moreover:
\begin{itemize}
	\item[a)] If $q$ is odd,
	\begin{equation}\label{Corf1a}
		\begin{cases}
		F^{(2)}(0)=0. \\
		\mbox{If } \ F^{(1)}(0)=0 \mbox{ then }F^{(3)}(0)=0.
		\end{cases}
	\end{equation}
	\item[b)] If $q$ is even, then it holds true that:
	\begin{equation}\label{Corf1b}
		\begin{cases}
		F^{(1)}(0)=0\mbox{ if and only if }  F^{(2)}(0)=0.\\
		\mbox{If } \ F^{(1)}(0)=0 \mbox{ then } F^{(3)}(0)=0.
		\end{cases}
	\end{equation}
\end{itemize}  
\end{lemma}

\begin{remark}
Equations \eqref{Corf1a} and \eqref{Corf1b} show that parity in the nonlinearity is a complicated issue. We cannot infer that  $F^{(2)}(0)=0$ in every case. Indeed, in the case $q$ even, the sought condition  $F^{(2)}(0)=0$ leads to the strong condition $F^{(1)}(0)=0$, that formally will lead to a trivial solution $F\equiv 0$. This is still not clear from our current computations, but during next sections will stay clear.
\end{remark}

\begin{proof}[Proof of Lemma \ref{Corf1}]
Consider \eqref{eq_simplificada}. Let $t_{0}\in\mathbb{R}$ be any element such that $s(x_{1},t_{0})=0$. Then from \eqref{G_s} one has $G(x,t_0)=0$ and
\begin{equation*}
	\begin{split}
		F^{(1)}(G(x,t))p(x)\partial^{(1)}_{t}s(x_{1},t) \Big|_{(x,t_{0})} &= F^{(1)}(0)p(x)\partial^{(1)}_{t}s(x_{1},t_{0}).
	\end{split}
\end{equation*}
Now we concentrate ourselves in the following simplified formula for the derivative of the composition. Using Lemma \ref{Bell_lema1} and \eqref{bells} ($B_{2,1}=x_2$), and the fact that $s(x_{1},t_{0})=0$,
\begin{equation*}
	\begin{split}
		\partial_{1}\partial^{2}_{j} (F\circ G)(x,t) \Big|_{(x,t_0)} &=\sum_{k=1}^{2}	F^{(k)}(0)B_{2,k}\left({\bf \partial^{(v)}_{j}p(x):v} \right) ks(x_{1},t_{0})^{k-1}\partial_{1}s(x_{1},t_{0})\\
		&=F^{(1)}(0)\partial^{2}_{j}p(x)\partial_{1}s(x_{1},t_{0}).
	\end{split}
\end{equation*}
Therefore, for $\Delta_{c}:= \sum_{j=2}^N \partial_j^2$,
\begin{equation*}
	\begin{split}
		(\partial_{1}\Delta_{c}v)(x,t_{0})=\partial_{1}(\Delta_{c}(F\circ G))(x,t_{0})=F^{(1)}(0)(\Delta_{c}p(x))\partial_{1}s(x_{1},t_{0}).
	\end{split}
\end{equation*}	
Also, using again Lemma \ref{Bell_lema1} with $\ell=0$, $j=1$ and $n=3$, $s(x_{1},t_{0})=0$ and \eqref{bells}, 
\begin{equation*}
	\begin{split}
		\partial^{3}_{1}v(x,t_{0})&=F^{(1)}(0)B_{3,1} \left( {\bf \partial_{j}^{(v)}G(x,t_0): v} \right) \\
		&\quad  +F^{(2)}(0) B_{3,2} \left( {\bf \partial_{j}^{(v)}G(x,t_0): v} \right) +F^{(3)}(0)B_{3,3} \left( {\bf \partial_{j}^{(v)}G(x,t_0): v} \right) \\
		& =F^{(1)}(0)\left( 3\partial_{1}^{(2)}p(x)+ p(x)(-1)(2\alpha)^{2}\right)\partial_{1}s(x_{1},t_{0}) \\
		&\quad + 6 F^{(2)}(0) p(x) \partial_{1}p(x) \partial_{1}s(x_{1},t_{0})^{} +F^{(3)}(0) p(x)^{3}\partial_{1}s(x_{1},t_{0})^{3}.
	\end{split}
\end{equation*}
Finally, we consider the nonlinear term:
\begin{equation*}
	\begin{split}
		2\left( F^{(1)}(G(x,t))\partial^{(1)}_{1}G(x,t)\right)^{q} \Big|_{(x,t_{0})}  & = 2\left( F^{(1)}(0)\partial_{1}G(x,t_{0})\right)^{q} \\
		&=2F^{(1)}(0)^{q}p(x)^{q}\partial_{1}s(x_{1},t_{0})^{q}. 
	\end{split}
\end{equation*}
Gathering the previous identities in \eqref{eq_simplificada}, we obtain the following list of terms:
\begin{equation*}
	\begin{split}
		&\left( \partial_t v+\partial_{x_1}(\Delta v) +2(\partial_{x_1}v)^q \right)\Big|_{(x,t_{0})} \\
		&\quad = F^{(1)}(0)p(x)\partial^{(1)}_{t}s(x_{1},t_{0})  +F^{(1)}(0)\left( 3\partial_{1}^{(2)}p(x)+ p(x)(-1)(2\alpha)^{2}\right)\partial_{1}s(x_{1},t_{0}) \\
		&\qquad +  6 F^{(2)}(0) p(x) \partial_{1}p(x)  \partial_{1}s(x_{1},t_{0})^{}  +F^{(3)}(0)( p(x)^{3}) \left( \partial_{1}s(x_{1},t_{0}) \right)^{3}\\
		&\qquad +F^{(1)}(0)(\Delta_{c}p(x))\partial_{1}s(x_{1},t_{0}) +2F^{(1)}(0)^{q}p(x)^{q} \left( \partial_{1}s(x_{1},t_{0}) \right)^{q}.
\end{split}
\end{equation*}
Simplifying,
\begin{equation*}
	\begin{split}
	&\big(\partial_t v+\partial_{x_1}(\Delta v) +2(\partial_{x_1}v)^q \big)\Big|_{(x,t_{0})} \\
		&\quad  =\partial_{1}s(x_{1},t_{0})\left(F^{(1)}(0)\left( mp(x)-p(x)(2\alpha)^{2}+ \Delta p(x)+ 2\partial_{1}^{(2)}p(x)\right) \right) \\
		&\qquad +\partial_{1}s(x_{1},t_{0})^{2}\left(6F^{(2)}(0)p(x)\partial_{1}p(x) \right) \\
		&\qquad  +\partial_{1}s(x_{1},t_{0})  ^{3}\left(F^{(3)}(0)p(x)^{3} \right)+\partial_{1}s(x_{1},t_{0})  ^{q}\left(2F^{(1)}(0)^{q}p(x)^{q} \right).
	\end{split}
\end{equation*}
In particular,  if  we choose $t_0= t_{1}$ and $t_0=t_{2}$  such that both $s(x_{1},t_{0})=0$, $\partial_{1} s(x_{1},t_{1})=2\alpha $ and $\partial_{1} s(x_{1},t_{2})=-2\alpha $ (this is indeed possible) we obtain:  
\begin{equation*}
	\begin{split}
		&\left(\partial_t v+\partial_{x_1}(\Delta v) +2(\partial_{x_1}v)^q \right) \Big|_{(x,t_{1})} \\
		&=2\alpha\left(F^{(1)}(0)\left( mp(x)-p(x)(2\alpha)^{2}+ \Delta p(x)+ 2\partial_{1}^{(2)}p(x)\right) \right) \\
		&\quad +4\alpha^{2}\left(6F^{(2)}(0)p(x)\partial_{1}p(x) \right) \\
		&\quad  +8\alpha^{3}\left(F^{(3)}(0)p(x)^{3} \right)+(2\alpha)^{q}2F^{(1)}(0)^{q}p(x)^{q} ,
	\end{split}
\end{equation*}
and 
\begin{equation*}
	\begin{split}
		&\left(\partial_t v+\partial_{x_1}(\Delta v) +2(\partial_{x_1}v)^q\right) \Big|_{(x,t_{2})} \\
		&\quad =-2\alpha\left(F^{(1)}(0)\left( mp(x)-p(x)(2\alpha)^{2}+ \Delta p(x)+ 2\partial_{1}^{(2)}p(x)\right) \right) \\
		&\qquad +4\alpha^{2}\left(6F^{(2)}(0)p(x)\partial_{1}p(x) \right) \\
		&\qquad -8\alpha^{3}\left(F^{(3)}(0)p(x)^{3} \right)+(-1)^{q}(2\alpha)^{q}2F^{(1)}(0)^{q}p(x)^{q}.
	\end{split}
\end{equation*}
Since \eqref{A2} is satisfied, subtracting the two previous identities we obtain \eqref{sec1}. If now one considers the addition of both identities, one gets \eqref{sec2}.

\medskip

{\it Proof of \eqref{Corf1a}.} Now, let us assume that $q$ is odd. From  \eqref{sec1} we obtain
\begin{equation}\label{corta}
	\begin{split}
		0&=4\alpha F^{(1)}(0)\left( mp(x)-p(x)(2\alpha)^{2}+ \Delta p(x)+ 2\partial_{1}^{(2)}p(x)\right)  \\
		&\quad +16\alpha^{3}\left(F^{(3)}(0)p(x)^{3} \right)+ 4(2\alpha)^{q}F^{(1)}(0)^{q}p(x)^{q},
	\end{split}
\end{equation}
and from  \eqref{sec2} we obtain
\[
\alpha^{2} F^{(2)}(0)p(x)\partial_{1}p(x) =0.
\] 
Since $\alpha>0$ and $p$ is nontrivial, this shows $F^{(2)}(0)=0$. Also, from \eqref{corta}, if $F^{(1)}(0)=0$ then $F^{(3)}(0)=0$. 

\medskip

{\it Proof of \eqref{Corf1b}.} On the other hand, if we assume that $q$ is even, then from  \eqref{sec1} we have that
\begin{equation*}
	\begin{split}
		0&=4\alpha\left(F^{(1)}(0)\left( mp(x)-p(x)(2\alpha)^{2}+ \Delta p(x)+ 2\partial_{1}^{(2)}p(x)\right) \right) \\
		&\quad +16\alpha^{3}\left( F^{(3)}(0)p(x)^{3} \right),
	\end{split}
\end{equation*}
and from \eqref{sec2} we have that
\[
8\alpha^{2}\left(6F^{(2)}(0)p(x)\partial_{1}p(x) \right)+ (2)(2\alpha)^{q}2F^{(1)}(0)^{q}p(x)^{q} =0.
\] 
The first equation shows that if $F^{(1)}(0)=0$, then $F^{(3)}(0)=0$; and the second equation shows that  $F^{(1)}(0)=0$ if and only if $F^{(2)}(0)=0$.
This ends the proof.
\end{proof}

Now we propagate time derivatives:
\begin{proposition}[Propagation of arbitrary time derivatives]\label{irradiacion} 
Let $v$ of the form \eqref{def_v}.	If $t_{0}\in\mathbb{R}$ is any element such that $s(x_{1},t_{0})=0$, then for every $n=2,3,\ldots$ it holds true that
	\begin{equation}\label{dtne}
		\begin{aligned}
		&	\left( \partial_{t}^{2n-3}	\left(\partial_t v +\partial_{x_1}(\Delta v) +2(\partial_{x_1}v)^q\right) \right)\Big|_{(x,t_{0})} \\
		& \qquad	=\sum_{\ell=1}^{2n}F^{(\ell)}(0) \left( \partial_{1}s(x_{1},t_{0}) \right)^{\ell}{\bf \zeta_{1,\ell} }(x)\\
		& \qquad \quad + \sum_{\ell=0}^{2n-3}  {\bf \zeta_{2,\ell} }(x) \sum_{k_{1}+k_{2}+\dots+k_{q}=\ell}  \sum_{\substack{ 0 \leq c_{1} \leq k_{1} \\ 0 \leq c_{2} \leq k_{2}  \\ {\small\vdots} \\ 0 \leq c_{n} \leq k_{n} }}  \prod_{i=1}^{q} F^{(c_{i}+1)}(0) \left( \partial_{1}s(x_1,t_{0}) \right)^{c_{i}+1} ,
		\end{aligned}
	\end{equation}	
	where 	${\bf \zeta_{i,\ell} }(x), i=1,2$ are functions that do not depend on $t_{0}$. Moreover,  
	\begin{equation}\label{nonzero}
	\hbox{ ${\bf \zeta_{1,2n-1} }(x)$ and ${\bf \zeta_{1,2n} }(x)$ are nonzero functions.}
	\end{equation}
\end{proposition}

In order to use this result, we define the set
\[
A:= \{ n \geq 2 ~ :  ~ \eqref{dtne} \hbox{ is valid for $n$} \}.
\]
It is not difficult to realize, after a proof by an inductive argument, that $A=\{n \geq 2\}$. This fact ensures that the property \eqref{dtne} is indeed valid for all $n\geq 2$. Therefore, we avoid using $A$ in what follows. An important corollary of the previous result is the following classification for $F$:

\begin{corollary}[Rigidity of the profile $F$]\label{F se anula}
Under the assumption on the real analyticity of $F$ at $0$ described in Theorems \ref{MT1} and \ref{MT2},  one has the following consequences:
\begin{itemize}
	\item[$(i)$]If the nonlinearity exponent $q$ in \eqref{mKdVN} is odd,
	\begin{equation}\label{F se anula 1}
		\begin{cases}
		F^{(2n)}(0)=0 \ \mbox{ for all } n\in\mathbb{N},\\
		\mbox{if} \ F^{(1)}(0)=0 \mbox{ then }F^{(n)}(0)=0 \mbox{ for all } n\in\mathbb{N}.
		\end{cases}
	\end{equation}
	\item[$(ii)$] If now $q$ is even, then it holds true that
	\begin{equation}\label{F se anula 2}
		 \mbox{ if } \ F^{(1)}(0)=0 \mbox{ then }F^{(n)}(0)=0 \mbox{ for all } n\in\mathbb{N}.
	\end{equation}
\end{itemize}  
\end{corollary}
The main conclusion from this corollary is that $F^{(1)}(0)\neq 0$, otherwise the solution is trivial.

\medskip

In the next subsection, we prove this corollary. Subsection \ref{proof_irradiacion} is devoted to the proof of Proposition \ref{irradiacion}.

\subsection{Proof of Corollary \ref{F se anula}}

First, we will prove statement $(ii)$.  Let's assume that $q$ is even. Using Lemma \ref{Corf1}, item $b)$, we have that:
\begin{equation*}
\begin{cases}
	 F^{(1)}(0)=0\quad \mbox{ if and only if }\quad   F^{(2)}(0)=0 \quad \mbox{ and,}\\
	\mbox{if } \quad F^{(1)}(0)=0, \quad \mbox{ then } \quad F^{(3)}(0)=0.
\end{cases}
\end{equation*}
We show that $F^{(1)}(0)=0$ implies $F^{(n)}(0)=0$ for every $n\in\mathbb{N}$. We use mathematical induction on $n$. So, let us assume 
assume that $F^{(1)}(0)=0$ implies $F^{(2)}(0)=F^{(3)}(0)=F^{(4)}(0)=\cdots=F^{(2n-3)}(0)=F^{(2n-2)}(0)=0$; then \eqref{dtne} of Proposition \eqref{irradiacion} implies that:
\begin{equation}\label{previa}
	\begin{aligned}
	\partial_{1}s(x_{1},t_{0})^{2n-1}F^{(2n-1)}(0){\bf \zeta_{1,2n-1} }(x)+ 	\partial_{1}s(x_{1},t_{0})^{2n}F^{(2n)}(0){\bf \zeta_{1,2n} }(x)=0.
	\end{aligned}
\end{equation}	
where 	${\bf \zeta_{1,2n-1} }(x)$ and ${\bf \zeta_{1,2n} }(x)$ are nonzero functions that do not depend on $t_{0}$. If we evaluate \eqref{previa} at $t_0= t_{1}$ and $t_0=t_{2}$  such that both $s(x_{1},t_{0})=0$, $\partial_{1} s(x_{1},t_{1})=2\alpha $ and $\partial_{1} s(x_{1},t_{2})=-2\alpha $, we obtain the following equations:  
\begin{equation*}
	\begin{aligned}
F^{(2n-1)}(0){\bf \zeta_{1,2n-1} }(x)+ 	F^{(2n)}(0){\bf \zeta_{1,2n} }(x)=0,
	\end{aligned}
\end{equation*}	
and
\begin{equation*}
	\begin{aligned}
		-F^{(2n-1)}(0){\bf \zeta_{1,2n-1} }(x)+ 	F^{(2n)}(0){\bf \zeta_{1,2n} }(x)=0.
	\end{aligned}
\end{equation*}	
If we add these last two equations, we obtain:
\begin{equation*}
	\begin{aligned}
		F^{(2n)}(0){\bf \zeta_{1,2n} }(x)=0,
	\end{aligned}
\end{equation*}	
and if we subtract them, we obtain
\begin{equation*}
	\begin{aligned}
		2F^{(2n-1)}(0){\bf \zeta_{1,2n-1} }(x) =0.
	\end{aligned}
\end{equation*}	
Since ${\bf \zeta_{1,2n-1} }(x)$ and ${\bf \zeta_{1,2n} }(x)$ are nonzero functions, we conclude that $F^{(2n-1)}(0)=F^{(2n)}(0)=0$, proving \eqref{F se anula 2}, just as intended.
	
\medskip	
	
Next we will prove statement $(i)$.  Let us assume that $q$ is odd. Using Lemma \ref{Corf1}, item $a)$, we know that $F^{(2)}(0)=0$. We will show that $F^{(2n)}(0)=0$. Again we will use mathematical induction on $n$. So, let us assume that $F^{(2)}(0)=F^{(4)}(0)=\cdots=F^{(2n-2)}(0)=0$. Equation \eqref{dtne} implies that:
	\begin{equation*}
		\begin{split}
			0&=	\sum_{\ell=1}^{n}F^{(2\ell-1)}(0) \left( \partial_{1}s(x_{1},t_{0}) \right)^{2\ell-1}{\bf \zeta_{1,2\ell-1} }(x)+F^{(2n)}(0) \left(\partial_{1}s(x_{1},t_{0}) \right)^{2n}{\bf \zeta_{1,2n} }(x)\\\
			&\quad + \sum_{\ell=0}^{2n-3}  {\bf \zeta_{2,\ell} }(x) \sum_{k_{1}+k_{2}+\dots+k_{q}=\ell}  \sum_{\substack{ 0 \leq 2c_{1} \leq k_{1} \\ 0 \leq 2c_{2} \leq k_{2}  \\ {\small\vdots} \\ 0 \leq 2c_{n} \leq k_{n} }}  \prod_{i=1}^{q} F^{2c_{i}+1}(0)\partial_{1}s(x,t_{0})^{2c_{i}+1} ,
		\end{split}
	\end{equation*}	
	where 	${\bf \zeta_{i,\ell} }(x), i=1,2$ are functions that do not depend on $t_{0}$. Similarly to the previous step,  if we evaluate the previous expression at $t_0= t_{1}$ and $t_0=t_{2}$  such that both $s(x_{1},t_{0})=0$, $\partial_{1} s(x_{1},t_{1})=2\alpha $ and $\partial_{1} s(x_{1},t_{2})=-2\alpha $, we obtain the following equations 
	\begin{equation*}
		\begin{split}
			0&=	\sum_{\ell=1}^{n}F^{(2\ell-1)}(0)(2\alpha)^{(2\ell-1)}{\bf \zeta_{1,2\ell-1} }(x)+F^{(2n)}(0)(2\alpha )^{(2n)}{\bf \zeta_{1,2n} }(x)\\\
			& \quad + \sum_{\ell=0}^{2n-3}  {\bf \zeta_{2,\ell} }(x) \sum_{k_{1}+k_{2}+\dots+k_{q}=\ell}  \sum_{\substack{ 0 \leq 2c_{1} \leq k_{1} \\ 0 \leq 2c_{2} \leq k_{2}  \\ {\small\vdots} \\ 0 \leq 2c_{n} \leq k_{n} }}  \prod_{i=1}^{q} F^{2c_{i}+1}(0)(2\alpha)^{2c_{i}+1},
		\end{split}
	\end{equation*}	
	and
	\begin{equation*}
		\begin{split}
			0&=	-\sum_{\ell=1}^{n}F^{(2\ell-1)}(0)(2\alpha )^{(2\ell-1)}{\bf \zeta_{1,2\ell-1} }(x)+F^{(2n)}(0)(2\alpha )^{(2n)}{\bf \zeta_{1,2n} }(x)\\\
			& \quad -\sum_{\ell=0}^{2n-3}  {\bf \zeta_{2,\ell} }(x) \sum_{k_{1}+k_{2}+\dots+k_{q}=\ell}  \sum_{\substack{ 0 \leq 2c_{1} \leq k_{1} \\ 0 \leq 2c_{2} \leq k_{2}  \\ {\small\vdots} \\ 0 \leq 2c_{n} \leq k_{n} }}  \prod_{i=1}^{q} F^{2c_{i}+1}(0)(2\alpha)^{2c_{i}+1}.
		\end{split}
	\end{equation*}	
	Consequently,
	\begin{equation*}	
		\begin{split}
			0&=2F^{(2n)}(0)(2\alpha m)^{(2n)}{\bf \zeta_{1,2n} }(x)=2	F^{(2n)}(0)(2\alpha m)^{(2n)}{\bf \xi^{2}_{2n}}(x)\\
			&=2	F^{(2n)}(0)(2\alpha m)^{(2n)}m^{-3}p(x)^{2n}, 
		\end{split}
	\end{equation*}	
	for every $x\in \mathbb R^N$. Since $p(x)^{2n}\neq 0$,  we get that $F^{(2n)}(0)=0$. Moreover, the equation \eqref{dtne}  transforms into
	\begin{equation}\label{imder}
		\begin{split}
			0= &	\sum_{\ell=1}^{n-1}F^{(2\ell-1)}(0){\bf \zeta_{1,2\ell-1} }(x)+ F^{(2n-1)}(0){\bf \zeta_{1,2n-1} }(x)\\
			& +\sum_{\ell=0}^{2n-3}  {\bf \zeta_{2,\ell} }(x) \sum_{k_{1}+k_{2}+\dots+k_{q}=\ell}  \sum_{\substack{ 0 \leq 2c_{1} \leq k_{1} \\ 0 \leq 2c_{2} \leq k_{2}  \\ {\small\vdots} \\ 0 \leq 2c_{n} \leq k_{n} }}  \prod_{i=1}^{q} F^{2c_{i}+1}(0)(2\alpha)^{2c_{i}+1}.
		\end{split}
	\end{equation}	
Now we show that $F^{(1)}(0)=0$ implies $F^{(n)}(0)=0$ for every $n\in\mathbb{N}$. From Corollary \ref{Corf1}  we get that $F^{(1)}(0)=0$ implies $F^{(3)}(0)=0$. 
	If we assume that $F^{(1)}(0)=0$ implies $F^{(3)}(0)=F^{(4)}(0)=\cdots=F^{(2n-3)}(0)=0$ then, equation \eqref{imder} implies that
\begin{equation*}
F^{(2n-1)}(0){\bf \zeta_{1,2n-1} }(x)=0.
\end{equation*}	
Since ${\bf \zeta_{1,2n-1}}$ is a nontrivial function we conclude that $ F^{(2n-1)}(0)=0$, proving \eqref{F se anula 1}. 

\subsection{Proof of Proposition \ref{irradiacion}}\label{proof_irradiacion} The terms $\zeta_{\mathbf{1},\ell}$ and $\zeta_{\mathbf{2},\ell}$ in \eqref{dtne} will be related to $\partial_{t}^{2n-3}	\left(\partial_t v +\partial_{x_1}(\Delta v)\right)$ and $2(\partial_{x_1}v)^q$ respectively. Indeed
we start computing  $\partial_{t}^{2n-3}\partial_{t}v $. Using Lemma \ref{Bell_lema1} with $h=2n-2$ and $G$ as in \eqref{G_s}, 
	\begin{equation*}
		\begin{split}
			\partial_{t}^{2n-3}	\partial_{t}v(x,t)&= \partial_{t}^{2n-2}( F\circ G) (x,t)\\
			&=\sum_{\ell=1}^{2n-2} F^{(\ell)} (G(x,t)) B_{2n-2,\ell} \left( {\bf \partial_{t}^{(v)}G(x,t):  v} \right)\\
			&=\sum_{\ell=1}^{2n-2} F^{(\ell)} (G(x,t))p(x)^{\ell} B_{2n-2,\ell}\left( {\bf\partial_{t}^{(v)}s(x_{1},t): v} \right).
		\end{split}
	\end{equation*}
	Evaluating this expression at $t=t_{0}$ (recall that by hypothesis $t_0$ is such that $s(x_{1},t_{0})=0$), and using Lemma \ref{PEB}, we obtain: 
	\begin{equation}\label{dtne1}
		\begin{aligned}
			\partial_{t}^{2n-3}	\partial_{t}v(x,t)\Big|_{(x,t_{0})}&=\sum_{\ell=1}^{2n-2} F^{(\ell)} (0) \left( \partial_{t}s(x_{1},t_{0}) \right)^{\ell}p(x)^{\ell} A_{2n-2,\ell}\\
			&:=\sum_{\ell=1}^{2n-2} F^{(\ell)} (0) \left( \partial_{t}s(x_{1},t_{0}) \right)^{\ell} {\bf \xi^{1}_{\ell}}(x).
		\end{aligned}
	\end{equation}
	In the following, we calculate the term $\partial_{t}^{(2n-3)}\left(\partial_{1}(\Delta v)  \right) $. We start with $\partial_{t}^{(2n-3)}\partial^{3}_{1}v$. Again using Lemma \ref{Bell_lema1},
	\begin{equation*}
		\begin{split}
\partial_{t}^{(2n-3)}\partial^{3}_{1}v(x,t)
			&= \partial^{3}_{1}(\partial_{t}^{(2n-3)}(F\circ G))\\
			&=\partial^{3}_{1} \left( \sum_{\ell=1}^{2n-3}F^{(\ell)}(G(x,t))p(x)^{\ell}B_{(2n-3),\ell}\left(\bf \partial^{(v)}_{t}s(x_{1},t):v\right)   \right).
			\end{split}
	\end{equation*}
	Expanding the last term above using the formula for the derivative of a product (see e.g. \eqref{ejemplo1}),
	\begin{equation*}
		\begin{split}
			&\partial_{t}^{(2n-3)}\partial^{3}_{1}v(x,t)\\
			&=\sum_{\ell=1}^{2n-3}\sum_{k_{1}+k_{2}+k_{3}=3}\binom{3}{k_{1},k_{2},k_{3}} \\
			& \hspace{3cm}\times\partial_{1}^{k_{1}}\left(F^{(\ell)}(G(x,t))\right) \partial_{1}^{k_{2}}\left(p(x)^{\ell} \right) \partial_{1}^{k_{3}}\left(B_{(2n-3),\ell}\left( \bf\partial^{(v)}_{t}s(x_{1},t):v\right) \right),
		\end{split}
	\end{equation*}
where the sum extends over all $3$-tuples $(k_{1},k_{2},k_{3})$ of non-negative integers with $k_{1}+k_{2}+k_{3}=3$. If we adopt the notation $C_{0,0}(x)=1$ and $C_{i,0}(x)=0$ for $i=1,2,3$, evaluating the previous expression at $t_{0}$, from Lemmas \ref{ADf} and \ref{ADhBts},  we obtain: 
	\begin{equation*}
		\begin{split}
			&\partial_{t}^{(2n-3)}\partial^{3}_{1}v(x,t) \Big|_{(x,t_{0})}\\
			&=\sum_{\ell=1}^{2n-3}\sum_{k_{1}+k_{2}+k_{3}=3}\binom{3}{k_{1},k_{2},k_{3}}\left( \sum_{k=0}^{k_{1}}F^{(\ell+k)}(0)C_{k_{1},k}(x)\left(m^{-1}  \partial_{t}s(x_{1},t_{0})\right)^{k}\right)\times \\  & \qquad \ \partial_{1}^{k_{2}}\left(p(x)^{\ell} \right) 
			A^{(k_{3})}_{(2n-3),\ell} \left( \partial_{t}s(x_{1},t_{0}) \right)^{\ell}\\
			&=\sum_{\ell=1}^{2n-3}\sum_{k_{1}+k_{2}+k_{3}=3}\binom{3}{k_{1},k_{2},k_{3}}\left( \sum_{k=0}^{k_{1}}F^{(\ell+k)}(0)C_{k_{1},k}(x)m^{-k} \left(\partial_{t}s(x_{1},t_{0}) \right)^{\ell+k}\right)\times \\  & \qquad \ \partial_{1}^{k_{2}}\left(p(x)^{\ell} \right) 
			A^{(k_{3})}_{(2n-3),\ell}.
		\end{split}
	\end{equation*}
Reorganizing terms in the previous sum, we obtain:
\begin{equation*}
	\begin{split}
		&\partial_{t}^{(2n-3)}\partial^{3}_{1}v(x,t) \Big|_{(x,t_{0})}\\
		&=\sum_{h=1}^{3}F^{(h)}(0) \left(\partial_{t}s(x_{1},t_{0}) \right)^{h}\\
		&\qquad \times\left( \sum_{k=0}^{h-1}\sum_{k_{1}=k}^{3}\left( \sum_{\substack{0\leq k_{2}\leq 3\\ 0\leq k_{3}\leq 3 \\ k_{1}+k_{2}+ k_{3}=3}}\binom{3}{k_{1},k_{2},k_{3}} C_{k_{1},k}(x)(m^{-k})\partial_{1}^{k_{2}}\left(p(x)^{h-k} \right) 
		A^{(k_{3})}_{(2n-3),h-k} \right)  \right) \\
		&+\sum_{4\leq h\leq 2n-3}F^{(h)}(0) \left( \partial_{t}s(x_{1},t_{0}) \right)^{h}\\
		&\qquad \times\left( \sum_{k=0}^{3}\sum_{k_{1}=k}^{3}\left( \sum_{\substack{0\leq k_{2}\leq 3\\ 0\leq k_{3}\leq 3 \\ k_{1}+k_{2}+ k_{3}=3}}\binom{3}{k_{1},k_{2},k_{3}} C_{k_{1},k}(x)(m^{-k})\partial_{1}^{k_{2}}\left(p(x)^{h-k} \right) 
		A^{(k_{3})}_{(2n-3),h-k} \right)  \right) \\
		&+\sum_{h=1}^{3}F^{(2n-3+h)}(0) \left( \partial_{t}s(x_{1},t_{0}) \right)^{2n-3+h}\\
		&\qquad \times \left( \sum_{k=h}^{3}\sum_{k_{1}=k}^{3}\left( \sum_{\substack{0\leq k_{2}\leq 3\\ 0\leq k_{3}\leq 3 \\ k_{1}+k_{2}+ k_{3}=3}}\binom{3}{k_{1},k_{2},k_{3}} C_{k_{1},k}(x)(m^{-k})\partial_{1}^{k_{2}}\left(p(x)^{(2n-3)+h-k} \right) 
		A^{(k_{3})}_{(2n-3),(2n-3)+h-k} \right)  \right) .
	\end{split}
\end{equation*}
We write the last expression as
	\begin{equation}\label{dtne2}
		\begin{split}
			&\partial_{t}^{(2n-3)}\partial^{3}_{1}v(x,t) \Big|_{(x,t_{0})}:=\sum_{\ell=1}^{2n}F^{(\ell)}(0) \left( \partial_{1}s(x_{1},t_{0}) \right)^{\ell}{\bf \xi^{2}_{\ell}}(x).
		\end{split}
	\end{equation}
For some $\xi^{\mathbf{2}}_{\ell}(x)$, with $\ell=1,\cdots, 2n$ only depending on $x$. In particular, the most important terms $\xi^{\mathbf{2}}_{\mathbf{2n}}$, $\xi^{\mathbf{2}}_{\mathbf{2n-1}} $ are described by 
	\begin{equation*}
		\begin{split}
			{\bf \xi^{2}_{2n-1}}(x)&=C_{3,3}(x)m^{-3}(p(x)^{2n-4})A^{0}_{(2n-3),(2n-4)}\\
			&\quad + C_{2,2}(x)m^{-2}\left(\sum_{k_{2}+k_{3}=1}\binom{3}{2,k_{2},k_{3}}\partial_{1}^{k_{2}}\left(p(x)^{2n-3} \right) 
			A^{(k_{3})}_{(2n-3),(2n-3)}  \right) \\
			&\quad  + C_{3,2}(x)m^{-2}p(x)^{2n-3} A^{0}_{(2n-3),(2n-3)},
		\end{split}
	\end{equation*}
	and, since $A^{0}_{2n-3,2n-4}=A^{1}_{2n-3,2n-3}=0$ and $A^{0}_{2n-3,2n-3}=1$, we have that
	\begin{equation}\label{zeta2n-1}
		\begin{split}
			{\bf \xi^{2}_{2n-1}}(x)	&=3 C_{2,2}(x)m^{-2}\partial_{1}\left(p(x)^{2n-3} \right) + C_{3,2}(x)m^{-2}p(x)^{2n-3}\\
			&=3 p(x)^{2}m^{-2}\partial_{1}\left(p(x)^{2n-3} \right) +6p(x)\partial_{1}(p(x))m^{-2}p(x)^{2n-3}\\
			&=3m^{-2}p(x)^{2n-2}\partial_{1}p(x)\left( (2n-3)p(x)^{2}+2\right) \neq 0.
		\end{split}
	\end{equation}
  Also 
	\begin{equation}\label{zeta2n}
		{\bf \xi^{2}_{2n}}(x)= m^{-3}C_{3,3}(x)p(x)^{2n-3}  
		A^{0}_{(2n-3),(2n-3)}\neq 0.
	\end{equation}
This proves the second assertion in \eqref{nonzero}. Now, we shift our focus to the term $\partial_{t}^{(2n-3)}\partial_{1}\partial^{2}_{j} (v)$ with $j\neq 1$. We compute using Lemma \ref{Bell_lema1}: 
	\begin{equation*}
		\begin{split}
			&\partial_{t}^{(2n-3)}\partial_{1}\partial^{2}_{j} (v)(x,t)\\
			& \quad =\partial_{1}\partial^{2}_{j}\partial_{t}^{(2n-3)} (F\circ G)(x,t)\\
			&\quad=\partial_{1}\partial^{2}_{j} \left( \sum_{\ell=1}^{2n-3}F^{(\ell)}(G(x,t))p(x)^{\ell}B_{(2n-3),\ell}\left(\bf \partial^{(v)}_{t}s(x_{1},t):v\right)   \right)\\	
			&\quad=\partial_{1}\left( \sum_{\ell=1}^{2n-3}\partial^{2}_{j}\left(F^{(\ell)}(G(x,t))p(x)^{\ell} \right) B_{(2n-3),\ell}\left(\bf \partial^{(v)}_{t}s(x_{1},t):v\right)   \right).
\end{split}
\end{equation*}
Expanding the derivatives,
	\begin{equation*}
		\begin{split}
		&\partial_{t}^{(2n-3)}\partial_{1}\partial^{2}_{j} (v)(x,t)\\
			&\quad=\partial_{1}\left( \sum_{\ell=1}^{2n-3}\sum_{g=0}^{2}\binom{2}{g}\partial_{j}^{g} (F^{(\ell)}(G(x,t)))\partial_{j}^{2-g}(p(x)^{\ell})  B_{(2n-3),\ell}\left(\bf \partial^{(v)}_{t}s(x_{1},t):v\right)   \right) \\ 		
			&\quad=\sum_{\ell=1}^{2n-3}\sum_{g=0}^{2}\binom{2}{g} \partial_{j}^{g} \left( F^{(\ell)}(G(x,t)) \right)\sum_{a=0}^{1}\binom{1}{a}\partial^{(1-a)}_{1}\left( \partial_{j}^{2-g}(p(x)^{\ell}) \right)  \partial^{(a)}_{1} \left( B_{(2n-3),\ell}\left(\bf \partial^{(v)}_{t}s(x_{1},t):v\right) \right)   \\
			&\quad\quad+\sum_{\ell=1}^{2n-3}\sum_{g=0}^{2}\binom{2}{g}\partial_{1}\left( \partial_{j}^{g} (F^{(\ell)}(G(x,t))) \right) \partial_{j}^{2-g}(p(x)^{\ell})  B_{(2n-3),\ell}\left(\bf \partial^{(v)}_{t}s(x_{1},t):v\right) .
	\end{split}
	\end{equation*}
	If we evaluate this expression at time $t_{0}$, we obtain from \eqref{G_s} that $G(x,t_0)=0$,
	\begin{eqnarray*}
		\begin{split}
			&\partial_{t}^{(2n-3)}\partial_{1}\partial^{2}_{j} (F\circ G)(x,t) \Big|_{(x,t_{0})}\\
			&= \sum_{\ell=1}^{2n-3}\left( F^{(\ell)}(0)\sum_{a=0}^{1}\partial_{1}^{(1-a)}\left( \partial_{j}^{2}(p(x)^{\ell})\right) \partial_{1}^{a}\left(B_{(2n-3),\ell}\left(\bf \partial^{(v)}_{t}s(x_{1},t_0):v\right) \right)  \right) \\
			&\quad  +  \sum_{\ell=1}^{2n-3}\left( \sum_{g=0}^{2}\binom{2}{g}F^{(\ell+1)}(0)\partial_{1}s(x_{1},t_{0})\partial_{j}^{g}(p(x))  \partial_{j}^{2-g}(p(x)^{\ell})  B_{(2n-3),\ell}\left(\bf \partial^{(v)}_{t}s(x_{1},t_0):v\right) \right).
		\end{split}
	\end{eqnarray*}
Using Lemmas \ref{ADhBts} and \ref{PEB}, 
\[
\begin{aligned}
&\partial_{t}^{(2n-3)}\partial_{1}\partial^{2}_{j} (F\circ G)(x,t) \Big|_{(x,t_{0})}\\
			&= \sum_{\ell=1}^{2n-3}\left( F^{(\ell)}(0)\sum_{a=0}^{1}\partial_{1}^{1-a}\left( \partial_{j}^{2}(p(x)^{\ell})\right) \left(A^{a}_{2n-3,\ell}  \partial_{t}s(x_{1},t_{0})^{\ell}\right)  \right) \\
			&\quad +  \sum_{\ell=1}^{2n-3}\left( \sum_{g=0}^{2}\binom{2}{g}F^{(\ell+1)}(0)\partial_{1}s(x_{1},t_{0})\partial_{j}^{g}(p(x))  \partial_{j}^{2-g}(p(x)^{\ell})  A_{2n-3,\ell} \partial_{t}s(x_{1},t_{0})^{\ell} \right)   \\		
			&= \sum_{\ell=1}^{2n-3}F^{(\ell)}(0)\partial_{t}s(x_{1},t_{0})^{\ell}\left( \sum_{a=0}^{1}\partial_{1}^{1-a}\left( \partial_{j}^{2}(p(x)^{\ell})\right) A^{a}_{2n-3,\ell}  \right) \\
			&\quad +  \sum_{\ell=2}^{2n-2}F^{(\ell)}(0)\partial_{t}s(x_{1},t_{0})^{\ell}m^{-1}\left( \sum_{g=0}^{2}\binom{2}{g}\partial_{j}^{g}(p(x))  \partial_{j}^{2-g}(p(x)^{\ell-1})  A_{2n-3,\ell-1}  \right).
\end{aligned}
\]
	Rearranging the terms, we can write
	\begin{eqnarray}\label{dtne3}
		\begin{split}
			\partial_{t}^{(2n-3)}\partial_{1}\partial^{2}_{j} (F\circ G)(x,t) \Big|_{(x,t_{0})}:=  \sum_{\ell=1}^{2n-2}F^{(\ell)}(0)\partial_{t}s(x_{1},t_{0})^{\ell}{\bf \xi^{3}_{\ell}}(x),
		\end{split}
	\end{eqnarray}
with
\[
\begin{aligned}
{\bf \xi^{3}_{1}}(x) =&~{}  \sum_{a=0}^{1}\partial_{1}^{1-a}\left( \partial_{j}^{2}(p(x))\right) A^{a}_{2n-3,1},  \\
{\bf \xi^{3}_{\ell}}(x) =&~{} \sum_{a=0}^{1}\partial_{1}^{1-a}\left( \partial_{j}^{2}(p(x)^{\ell})\right) A^{a}_{2n-3,\ell}  \\
&~{} + m^{-1}\left( \sum_{g=0}^{2}\binom{2}{g}\partial_{j}^{g}(p(x))  \partial_{j}^{2-g}(p(x)^{\ell-1})  A_{2n-3,\ell-1}  \right), \quad 2\leq \ell \leq 2n-3, \\
{\bf \xi^{3}_{2n-2}}(x) =&~{} m^{-1}\left( \sum_{g=0}^{2}\binom{2}{g}\partial_{j}^{g}(p(x))  \partial_{j}^{2-g}(p(x)^{2n-3})  A_{2n-3,2n-3}  \right). 
\end{aligned}
\]
Notice that from \eqref{dtne1}, \eqref{dtne2} and \eqref{dtne3}
\begin{equation}\label{defzeta}\zeta_{\mathbf{1},\ell}(x):=\left\{
  \begin{array}{ll}
     \xi^\mathbf{1}_\ell(x)+\xi^\mathbf{2}_\ell(x)+\xi^\mathbf{3}_\ell(x), & \hbox{if $\ell=1,\cdots,(2n-2)$;} \\
     \xi^\mathbf{2}_\ell(x), & \hbox{if $\ell=2n-1,2n$}.
  \end{array}
\right.\end{equation}
Clearly this relation   and \eqref{zeta2n-1}, \eqref{zeta2n}
prove the assertion in \eqref{nonzero}.

\medskip

We will now address the nonlinear term $\partial_{t}^{2n-3}	((\partial_{1}v)^{q})$:
	\begin{equation*}
		\begin{split}
			&\partial_{t}^{2n-3}	\left( \left(\partial_{1}v(x,t) \right)^{q} \right) \\
			&= \partial_{t}^{2n-3} \left( \left(  \partial_{1}(F\circ G) (x,t) \right)^{q} \right)\\
			&=\partial_{t}^{2n-3}\left( \left( F^{(1)}(G(x,t)) \right)^{q} \left( \partial^{(1)}_{1}G(x,t) \right)^{q}\right)\\
			&=\sum_{\ell=0}^{2n-3} \binom{2n-3}{\ell }\partial_{t}^{\ell}\left(  \left( F^{(1)}(G(x,t)) \right)^{q} \right) \partial_{t}^{(2n-3)-\ell}\left( \left( \partial_{1}G(x,t) \right)^{q} \right).
	\end{split}
	\end{equation*}
Expanding using the classical Leibnitz rule (see \eqref{ejemplo1}),
\begin{equation*}
\begin{split}
			&\partial_{t}^{2n-3}	\left( \left(\partial_{1}v(x,t) \right)^{q} \right)\\
			&=\sum_{\ell=0}^{2n-3} \binom{2n-3}{\ell }\left(\sum_{k_{1}+k_{2}+\dots+k_{q}=\ell}  \binom{\ell}{k_{1},k_{2},\dots,k_{q}}\prod_{i=1}^{q}\partial_{t}^{(k_{i})} \left( F^{(1)}(G(x,t)) \right)  \right)\times\\
			&\qquad \partial_{t}^{(2n-3)-\ell}\Big( \Big(\partial_{1}p(x)s(x,t)+p(x)\partial_{1}s(x,t) \Big)^{q}\Big).
\end{split}
\end{equation*}
On the other hand, using the binomial theorem, we have
\begin{equation*}
	\begin{split}
		&\partial_{t}^{(2n-3)-\ell}\Big( \Big(\partial_{1}p(x)s(x,t)+p(x)\partial_{1}s(x,t) \Big)^{q}\Big)\\
		&=\partial_{t}^{(2n-3)-\ell}\left(\sum_{k=0}^{q}\binom{q}{k}(\partial_{1}p(x)s(x,t) )^{k}(p(x)\partial_{1}s(x,t))^{q-k} \right)\\
		&=\sum_{k=0}^{q}\binom{q}{k}(\partial_{1}p(x))^{k}(p(x))^{q-k}  \partial_{t}^{(2n-3)-\ell}\left(s(x,t)^{k} \left(\partial_{1}s(x,t) \right)^{q-k} \right).
	\end{split}
\end{equation*}
Therefore, invoking Lemma \ref{Bell_lema1}, we have
\begin{equation*}
\begin{split}
			&\partial_{t}^{2n-3}	\left( \left(\partial_{1}v(x,t) \right)^{q} \right)\\
			&=\sum_{\ell=0}^{2n-3} \binom{2n-3}{\ell }\\
			& \qquad \times \left(\sum_{k_{1}+k_{2}+\dots+k_{q}=\ell}  \binom{\ell}{k_{1},k_{2},\dots,k_{q}}\prod_{i=1}^{q} \left(\sum_{k=0}^{k_{i}}F^{(k+1)}(G(x,t))p(x)^{k}B_{k_{i},k}({ \bf \partial_{t}^{(v)}s(x,t):(v) }) \right)   \right)\\
			&\qquad \times \left( \sum_{k=0}^{q}\binom{q}{k}(\partial_{1}p(x))^{k}(p(x))^{q-k}  \partial_{t}^{(2n-3)-\ell}\left(s(x,t)^{k} \left(\partial_{1}s(x,t) \right)^{q-k} \right)    \right).
		\end{split}
	\end{equation*}
First of all, from Lemma \ref{PEB},
\[
\sum_{k=0}^{k_{i}}F^{(k+1)}(G(x,t))p(x)^{k}B_{k_{i},k}({ \bf \partial_{t}^{(v)}s(x,t_0):(v) })  = \sum_{k=0}^{k_{i}}F^{(k+1)}(0) A_{k_{i},k}^0 \left( \partial_{t}s(x,t_{0}) \right)^{k}p(x)^{k}.
\]
We concentrate now in the terms $\partial_{t}^{(2n-3)-\ell}\left(s(x,t)^{k} \left(\partial_{1}s(x,t) \right)^{q-k} \right) $ above. If we set $r=\min\{(2n-3)-\ell, q \}$, then by Lemma \ref{product_law0} we have that:
\begin{equation*}
\begin{split}
		&\left( \sum_{k=0}^{q}\binom{q}{k}(\partial_{1}p(x))^{k}(p(x))^{q-k}  \partial_{t}^{(2n-3)-\ell}\left(s(x,t)^{k}\partial_{1}s(x,t)^{q-k} \right)    \right)\Bigg|_{(x,t_{0})}\\
		&\qquad = \left( \partial_{1}s(x,t_{0}) \right)^{q}	\left( \sum_{k=0}^{r}\binom{q}{k}(\partial_{1}p(x))^{k}(p(x))^{q-k}  A_{k,q-k,(2n-3)-\ell}^{t}   \right).
\end{split}
\end{equation*}	
Consequently, if we evaluate at time $t=t_{0}$, and use that from \eqref{G_s} $\left(\partial_{1}s(x,t_{0}) \right)^{q} = (2\alpha )^q$, we obtain:
	\begin{equation*}
		\begin{split}
			&\partial_{t}^{2n-3}	(\partial_{1}v(x,t)^{q})\Big|_{(x,t_{0})}\\
			&=\sum_{\ell=0}^{2n-3} \binom{2n-3}{\ell }\left(\sum_{k_{1}+k_{2}+\dots+k_{q}=\ell}  \binom{\ell}{k_{1},k_{2},\dots,k_{q}}\prod_{i=1}^{q} \left(\sum_{k=0}^{k_{i}}F^{(k+1)}(0) \left( \partial_{t}s(x,t_{0}) \right)^{k}p(x)^{k} A^{0}_{k_{i},k} \right)   \right)\times\\
			&\qquad   \left(\partial_{1}s(x,t_{0}) \right)^{q}	\left( \sum_{k=0}^{r}\binom{q}{k}(\partial_{1}p(x))^{k}(p(x))^{q-k}  A_{k,q-k,(2n-3)-\ell}^{t}  \right)\\
			&=\sum_{\ell=0}^{2n-3} \binom{2n-3}{\ell }\left(\sum_{k_{1}+k_{2}+\dots+k_{q}=\ell}  \binom{\ell}{k_{1},k_{2},\dots,k_{q}}\prod_{i=1}^{q} \left(\sum_{k=0}^{k_{i}}F^{(k+1)}(0)   \left( \partial_{t}s(x,t_{0}) \right)^{k+1} p(x)^{k} A_{k_{i},k}^{0} \right)   \right)\times\\
			&\qquad \left( \sum_{k=0}^{r}\binom{q}{k}(\partial_{1}p(x))^{k}(p(x))^{q-k} A_{k,q-k,(2n-3)-\ell}^{t}   \right).
		\end{split}
	\end{equation*}
Here we use that
\begin{equation*}
\begin{split}
& \prod_{i=1}^{q} \left(\sum_{k=0}^{k_{i}}F^{(k+1)}(0) \left( \partial_{t}s(x,t_{0}) \right)^{k}p(x)^{k}  A^{0}_{k_{i},k} \right) \left( \partial_{1}s(x,t_{0}) \right)^{q} \\
& \quad  = \prod_{i=1}^{q} \left(\sum_{k=0}^{k_{i}}F^{(k+1)}(0)  \left( \partial_{1}s(x,t_{0}) \right)^{k+1} p(x)^{k} m^{k} A_{k_{i},k}^{0} \right).   
\end{split}
\end{equation*}
Reorganizing the terms in the previous equation allows us to write it in the following form:
	\begin{equation}\label{dtne4}
		\begin{split}
			&\partial_{t}^{2n-3}	\left( \partial_{1}v(x,t)^{3} \right) \Big|_{(x,t_{0})}\\
			& =\sum_{\ell=0}^{2n-3}   \sum_{k_{1}+k_{2}+\dots+k_{q}=\ell}   \left(\sum_{\substack{ 0 \leq c_{1} \leq k_{1} \\ 0 \leq c_{2} \leq k_{2}  \\ \cdots \\ 0 \leq c_{n} \leq k_{n} }}  \prod_{i=1}^{q} F^{(c_{i}+1)}(0) \left( \partial_{1}s(x,t_{0}) \right)^{c_{i}+1} \right) \zeta_{\mathbf{2},\ell}(x).
		\end{split}
	\end{equation}
Finally, equation \eqref{dtne} is deduced by combining equations \eqref{dtne1}, \eqref{dtne2}, \eqref{dtne3}, \eqref{defzeta} and \eqref{dtne4}.

\medskip

We finish this Section with the following useful property:

\begin{lemma}\label{eq_simplificada_1}
Assume $N=1$. Let $q\in\{1,2,\ldots\}$ be odd, and $\alpha>0$ and $m\in\R$ free parameters. Consider $v$ as in \eqref{def_v}, and $G$, $s$ as in \eqref{G_s}. Let $t_0$ be a time under which $s(x_1,t_0)=G(x,t_0)=0$.  Then the following expansion holds true: 
\begin{equation}\label{eq_simplificada_2}
	\begin{split}
  0=&~{} 8 \alpha  F^{(1)}(0) \partial_1^{(3)} p(x)+ 2 \alpha  m F^{(1)}(0) \partial_{1} p(x) -32 \alpha ^3 F^{(1)}(0) \partial_{1} p(x)\\
  &~{} +96 \alpha ^3 F^{(3)}(0) p(x)^2 \partial_{1}p(x) + 2^{q+2} q\alpha^q   \left(F^{(1)}(0) \right)^q \left(  p(x)\right)^{q-1}\partial_{1}p(x).
	\end{split}
\end{equation}
\end{lemma}

\begin{proof}
From \eqref{eq_simplificada} and \eqref{A2}, we get
\[
	\begin{split}
		0&= \partial_1 \left( \partial_t v+\partial_{x_1}(\Delta v)+2(\partial_{x_1}v)^q \right) \\
		& = \partial_1 \left( F^{(1)}(G(x,t))p(x)\partial^{(1)}_{t}s(x_{1},t) \right) +2 \partial_1 \left( \left( F^{(1)}(G(x,t))\partial^{(1)}_{1}G(x,t)\right)^{q} \right)\\
		&\quad +\sum_{k=1}^{3} \partial_1 \left( F^{(k)}(G(x,t))B_{3,k}\left( \bf{\partial^{(v)}_{1}G(x,t):v} \right) \right).
	\end{split}
\]
Since $F^{(2)}(0)=F^{(4)}(0)=0$ (Corollary \ref{F se anula}), we easily get after evaluation at $t=t_0$,
\[
	\begin{split}
		0 & =  F^{(1)}(0) \partial_1 p(x) 2m\alpha  \\
		& \quad +2q   \left( F^{(1)}(0)\partial_1^{(1)} G(x,t_0) \right)^{q-1}   F^{(1)}(0)  \partial_1 \left(\partial^{(1)}_{1}G(x,t_0)\right)\\
		&\quad +  F^{(3)}(0)  \partial_1^{(1)} G(x,t)  B_{3,2}\left( \bf{\partial^{(v)}_{1}G(x,t):v} \right) \Big|_{(x,t_0)} \\
		& \quad +\sum_{k=1,3}  F^{(k)}(0) \partial_1 \left( B_{3,k}\left( \bf{\partial^{(v)}_{1}G(x,t):v} \right) \right)\Big|_{(x,t_0)}.
	\end{split}
\]
Replacing \eqref{bells},
\[
	\begin{split}
		0 & =  2m\alpha F^{(1)}(0) \partial_1 p(x)   \\
		& \quad +2q   \left( F^{(1)}(0) 2\alpha p(x) \right)^{q-1}   F^{(1)}(0)  \partial^{(2)}_{1}G(x,t_0) \\
		&\quad + 6  \alpha F^{(3)}(0)  p(x)  \partial_1^{(1)} G(x,t_0)\partial_1^{(2)} G(x,t_0)  \\
		& \quad +  F^{(1)}(0)  \partial_1^{(4)} G(x,t_0) +  F^{(3)}(0) \partial_1 \left( \partial_1^{(1)} G(x,t_0) \right)^3.
	\end{split}
\]
Finally, using \eqref{G_s} and the fact that $ \partial_1^{(2)} G(x,t_0)=2 \partial_1^{(1)} p(x) (2\alpha) $ and $ \partial_1^{(4)} G(x,t_0)= 4 \partial_1^{(3)} p(x) (2\alpha) - 4 \partial_1^{(1)} p(x) (2\alpha)^3$, we conclude. This proves \eqref{eq_simplificada_2}.
\end{proof}

\section{Additional propagation identities}\label{adicional}

The previous Section was useful to establish some fundamental propagation properties held by the ZK flow under the Bell's algebra. These results are not enough to conclude our main results. In this Section, we expand further this theory. Recall that from Corollary \ref{F se anula}, one has $F^{(1)}(0)\neq 0$.

\begin{lemma}[Integrated propagation, $q$ odd]\label{EcuLapl2}
	Let $v$ of the form \eqref{def_v} be a smooth solution of \eqref{A2} with $q>1$ odd. 
	Let $t_{1}\in\mathbb{R}$ be any time such that $s(x_{1},t_{1})=0$ and $\partial_{t} s(x_{1},t_{1})=2\alpha m$. 
	Then it holds true that:
\begin{equation}\label{integrada}
	\begin{split}
		&2\alpha m \int_{-\infty}^{x_1}\left(\partial_{t}\Big( \partial_t v+\partial_{x_1}(\Delta v) +2(\partial_{x_1}v)^q \Big) \Big|_{(s_1, x_2,\ldots, x_N,t_{1})} \right)  ds_{1}	\\
		& \quad = F^{(1)}(0) \left( \Delta p(x)-12\alpha^{2} p(x)  +p(x)^{3} \left( 12\alpha^{2} \frac{F^{(3)}(0)}{F^{(1)}(0)}\right)+p(x)^{q} 2^{q}\alpha^{q-1} \left(F^{(1)}(0)\right)^{q-1}  \right).
	\end{split}
\end{equation}
\end{lemma}
\begin{proof}
We will begin by calculating the term $\partial_{t}\partial_{1}\partial^{2}_{j} (v)$ for $j\neq 1$:
\begin{equation*}
	\begin{split}
		\partial_{t}\partial_{1}\partial^{2}_{j} (v)(x,t)& =\partial^{2}_{j}\partial_{1}\partial_{t} (F\circ G)(x,t)\\
		&=\partial^{2}_{j}\partial_{1} \left( F^{(1)}(G(x,t))p(x) \partial_{t}s(x,t)  \right)\\			
		&=\partial^{2}_{j} \left( \partial_{1}(F^{(1)}(G(x,t)))p(x) \partial_{t}s(x,t)   \right) \\
		& \quad +\partial^{2}_{j} \left( F^{(1)}(G(x,t))\partial_{1}\left(p(x) \partial_{t}s(x,t) \right) \right) =:A_1+A_2.
	\end{split}
\end{equation*}
Next, we evaluate these two terms at time $t_{1}$, and using that $s(x_{1},t_{1})=0$ and $\partial_{t} s(x_{1},t_{1})=2\alpha m$,
\begin{equation*}
	\begin{split}
		A_1\Big|_{(x,t_{1})}& =\partial^{2}_{j} \left( \partial_{1}(F^{(1)}(G(x,t)))p(x)\partial_{t}s(x,t)   \right)\Big|_{(x,t_{1})}\\
		&= \left(F^{(2)}(G(x,t))\partial_{1}G(x,t) \right) \partial^{2}_{j} \left(p(x)\right)\partial_{t}s(x,t)\Big|_{(x,t_{1})}\\   				
		&\quad + 2\partial^{1}_{j} \left(F^{(2)}(G(x,t))\partial_{1}G(x,t) \right)\partial_{j} (p(x))\partial_{t}s(x,t)\Big|_{(x,t_{1})}   \\
		&\quad +\partial^{2}_{j} \left(F^{(2)}(G(x,t))\partial_{1}G(x,t)\right)p(x)\partial_{t}s(x,t)\Big|_{(x,t_{1})}   \\
		&= \left(F^{(2)}(0)p(x)(2\alpha) \right) \partial^{2}_{j} \left(p(x)\right)(2\alpha m)\\   				
		&\quad + 2\left(F^{(3)}(0)(0)(2\alpha) +F^{(2)}(0)(\partial_{j}p(x)(2\alpha))\right)\partial_{j} (p(x))(2\alpha m)   \\
		&\quad +\left(F^{(4)}(0)(0)+ F^{(3)}(0)(0)+F^{(3)}(0)(0) \right)p(x) (2\alpha m)  =0.
	\end{split}
\end{equation*}
Similarly, we have that:
\begin{equation*}
	\begin{split}
A_2\Big|_{(x,t_{1})}&=\partial^{2}_{j} \Big( F^{(1)}(G(x,t))\partial_{1}\Big(p(x) \partial_{t}s(x_{1},t) \Big) \Big) \Big|_{(x,t_{1})}\\
& =\partial^{2}_{j} \left( F^{(1)}(G(x,t))\Big(\partial_{1}p(x) \partial_{t}s(x_{1},t)+p(x) \partial_{1}\partial_{t}s(x_{1},t) \Big) \right) \Big|_{(x,t_{1})}\\
& =\partial^{2}_{j} \left( F^{(1)}(G(x,t))\right)\Big(\partial_{1}p(x) \partial_{t}s(x_{1},t)+p(x) \partial_{1}\partial_{t}s(x_{1},t)  \Big)\Big|_{(x,t_{1})} \\
&\quad + 2\partial_{j} \left( F^{(1)}(G(x,t))\right)\Big( \partial_{j}(\partial_{1}p(x)) \partial_{t}s(x_{1},t)+\partial_{j}(p(x)) \partial_{1}\partial_{t}s(x_{1},t)  \Big)\Big|_{(x,t_{1})} \\
&\quad + F^{(1)}(G(x,t))\Big(\partial^{2}_{j}(\partial_{1}p(x)) \partial_{t}s(x_{1},t)+\partial^{2}_{j}(p(x)) \partial_{1}\partial_{t}s(x_{1},t) \Big)\Big|_{(x,t_{1})}  \\
&  = 2\alpha m F^{(1)}(0)\partial^{2}_{j}(\partial_{1}p(x)) .
	\end{split}
\end{equation*}
Hence, adding the terms $A_1\big|_{(x,t_{1})}$ and $A_2\big|_{(x,t_{1})}$ and summing on $j$,
\begin{equation*}
(\partial_{t}\partial_{1}\Delta_{c}(v(x,t)))\Big|_{(x,t_{1})}  =2\alpha m F^{(1)}(0) \Delta_{c}(\partial_{1}p(x)) .
\end{equation*}
We continue with the expansion of $\partial_{t}\partial^{3}_{1}v$:
\begin{equation*}
	\begin{split}
		&\partial_{t}\partial^{3}_{1}v(x,t)\Big|_{(x,t_{1})}\\
		&\quad =\partial_{t} (\partial^{3}_{1}(F\circ G))\Big|_{(x,t_{1})}\\
		&\quad =\partial_{t} \left(\sum_{\ell=1}^{3}F^{(\ell)}(G(x,t))B_{3,\ell}\left(\bf\partial^{(v)}_{1}G(x,t):v \right)  \right)\Big|_{(x,t_{1})} \\
	&\quad = \sum_{\ell=1}^{3}F^{(\ell)}(G(x,t))\partial_{t}\left(B_{3,\ell}\left(\bf\partial^{(v)}_{1}G(x,t):v \right) \right)\Big|_{(x,t_{1})} \\
	& \qquad +\sum_{\ell=1}^{3}F^{(\ell+1)}(G(x,t))\partial_{t}G(x,t)B_{3,\ell}\left(\bf\partial^{(v)}_{1}G(x,t):v \right)\Big|_{(x,t_{1})}.
	\end{split}
\end{equation*}
Replacing,
\begin{equation*}
	\begin{split}
			&\partial_{t}\partial^{3}_{1}v(x,t)\Big|_{(x,t_{1})}\\
		&\quad = F^{(1)}(0)\left( \partial^{3}_{1}p(x)(2m\alpha)-3\partial_{1}p(x)(2m\alpha)(2\alpha)^{2}\right)\\
		&\qquad+F^{(2)}(0)\left( 24(\partial_{1}p(x))^{2}m\alpha^{2}+12p(x)\partial_{1}^{2}p(x)m\alpha^{2}-48p(x)^{2}m\alpha^{4}\right) \\
		&\qquad+F^{(3)}(0)\left(3(p(x)^{2}(2\alpha)^{2})(\partial_{1}p(x)(2m\alpha))  \right) \\
		&\qquad+F^{(3)}(0)(p(x)(2\alpha m))(24p(x)\partial_{1} p(x))\\
		&\quad = F^{(1)}(0)\left( \partial^{3}_{1}p(x)(2m\alpha)-3\partial_{1}p(x)(2m\alpha)(2\alpha)^{2}\right)\\
		&\qquad+F^{(3)}(0)\left(3(p(x)^{2}(2\alpha)^{2})(\partial_{1}p(x)(2m\alpha))  \right) +F^{(3)}(0)(p(x)(2\alpha m))(24p(x)\partial_{1} p(x)).
	\end{split}
\end{equation*}
Now, we deal with the nonlinear term:
\begin{equation*}
	\begin{split}
		2&\partial_{t}	(\partial_{1}v(x,t)^{q})\Big|_{(x,t_{1})}\\
		&=2 \partial_{t}(\partial_{1}(F\circ G) (x,t)^{q})\Big|_{(x,t_{1})}\\
		&=2\partial_{t}\left( F^{(1)}(G(x,t))^{q}(\partial_{1}G(x,t))^{q}\right)\Big|_{(x,t_{1})}\\
		&=2\partial_{t}\left( F^{(1)}(G(x,t))^{q}\right)(\partial_{1}G(x,t))^{q}\Big|_{(x,t_{1})}		+2 F^{(1)}(G(x,t))^{q}\partial_{t}\left((\partial_{1}G(x,t))^{q}\right)\Big|_{(x,t_{1})}\\
		&=2\left(q F^{(1)}(G(x,t))^{q-1}\partial_{t}\left(F^{(1)}(G(x,t) \right) \right)(\partial_{1}G(x,t))^{q}\Big|_{(x,t_{1})}		\\
		&\quad +2 F^{(1)}(G(x,t))^{q}\left(q(\partial^{(1)}_{1}G(x,t))^{q-1}\partial_{t}\partial_{1}G(x,t)\right)\Big|_{(x,t_{1})}\\
		&=2F^{(1)}(0)^{q}qp(x)^{q-1}(2\alpha)^{-1}\partial_{1}p(x)(2\alpha m).
	\end{split}
\end{equation*}
Finally,
\begin{equation*}
	\begin{split}
		\partial_{t}	(\partial_{t}v(x,t))\Big|_{(x,t_{1})} &=	\partial^{2}_{t}v(x,t)\Big|_{(x,t_{1})}\\
		&=\sum_{\ell=1}^{2}F^{(\ell)}(G(x,t))p(x)^{\ell}B_{2,\ell}(\bf{ \partial_{t}s(x_{1},t):(v)})\\
		&=F^{(1)}(0)p(x)\partial_{t}^{2}s(x_{1},t)\Big|_{(x,t_{1})}+F^{(2)}(0)p(x)^{2}(\partial_{t}s(x_{1},t))^{2}\Big|_{(x,t_{1})}=0.
	\end{split}
\end{equation*}
Putting together all the information developed above, we get that:
\begin{equation}\label{eqvt}
	\begin{split}
		&\partial_{t}\left( \partial_t v+\partial_{x_1}(\Delta v) +2(\partial_{x_1}v)^3 \right) \Big|_{(x,t_{1})} \\
		&\quad  =2\alpha mF^{(1)}(0)\partial_{1}(\Delta_{c}p(x))   + 2m\alpha F^{(1)}(0) \partial^{3}_{1}p(x)\\
		&\qquad -24m\alpha^{3}F^{(1)}(0) \partial_{1}p(x)+72m\alpha^{3}F^{(3)}(0)p(x)^{2}\partial_{1}p(x) \\
		& \qquad +2^{q+1}q\alpha^{q}m \left( F^{(1)}(0) \right)^{q}p(x)^{q-1}\partial_{1}p(x)  \\
		&\quad = 2\alpha mF^{(1)}(0)\partial_{1}(\Delta p(x))   -24m\alpha^{3}F^{(1)}(0) \partial_{1}p(x) \\
		& \qquad +72m\alpha^{3}F^{(3)}(0)p(x)^{2}\partial_{1}p(x)+2^{q+1}q\alpha^{q}m F^{(1)}(0)^{q}p(x)^{q-1}\partial_{1}p(x).
	\end{split}
\end{equation}
Note that the last expression represents a total derivative in the $x_1$ variable. Integrating the previous equation over the variable $x_{1}$ and using the decay of $p(x)$ in the variable $x_{1}$, we obtain:
\begin{equation*}
	\begin{split}
		&\int_{-\infty}^{x_1}\partial_{t}\left(\partial_t v+\partial_{x_1}(\Delta v) +2(\partial_{x_1}v)^3 \right) \Big|_{(s_1,x_2,\ldots,x_N,t_{1})} ds_{1}\\
		& \quad = 2\alpha mF^{(1)}(0)\Delta p(x)   -24m\alpha^{3}F^{(1)}(0) p(x)+24m\alpha^{3}F^{(3)}(0)p(x)^{3}\\
		&\qquad +2^{q+1}\alpha^{q}m F^{(1)}(0)^{q}p(x)^{q}.
	\end{split}
\end{equation*}
This completes the proof of \eqref{integrada}.	
\end{proof}

\begin{lemma}[Identity for the gradient squared]\label{ecugra1}
Let $v$ of the form \eqref{def_v} be a smooth solution of \eqref{A2} with $q$ odd.
Let $t_{1}\in\mathbb{R}$ be any time such that $s(x_{1},t_{1})=0$ and $\partial_{t} s(x_{1},t_{1})=2\alpha m$. Then it holds true that
\begin{equation*}
	\begin{split}
		& \partial_{t}^{2}	\Big(\partial_{t}v(x,t)+ \partial_{x_1}(\Delta v(x,t))\Big)  \Big|_{(x,t_{1})}\\
		&\quad =48p(x)F^{(3)}(0)m^{2}\alpha^{3}\left( |\nabla p(x)|^{2}+2\left(\partial_{1}p(x)\right)^{2} \right) \\
		&\qquad -16p(x)^{3}F^{(3)}(0) m^{3}\alpha^{3}+2^{q+3}p(x)^{q} \left( F^{(1)}(0) \right)^{q}m^{2}\alpha^{q+2} \\
		&\qquad  -3\cdot2^{q+3}p(x)^{q+2}F^{(3)}(0) \left( F^{(1)}(0) \right)^{q-1}m^{2}\alpha^{q+2}-96p(x)^{5}m^{2}\alpha^{5}\frac{ \left( F^{(3)}(0) \right)^{2}}{F^{(1)}(0)}\\
		&\qquad -192p(x)^{3}F^{(3)}(0)m^{2}\alpha^{5}+F^{(5)}(0)32m^{2}\alpha^{5}p(x)^{5}.
	\end{split}
\end{equation*} 
\end{lemma}
\begin{proof}
This lemma is proved through a direct calculation. We will begin by calculating the term $\partial_{t}^{2}	\partial_{t}v(x,t)\Big|_{(x,t_{1})}$:
\begin{equation*}
	\begin{split}
		\partial_{t}^{2}	\partial_{t}v(x,t)&= \partial_{t}^{3}( F\circ G) (x,t)\\
		&=\sum_{\ell=1}^{3} F^{(\ell)} (G(x,t)) B_{3,\ell} \left({\bf \partial_{t}^{(v)}G(x,t):v} \right)\\
		&=\sum_{\ell=1}^{3} F^{(\ell)} (G(x,t))p(x)^{\ell} B_{3,\ell} \left({\bf \partial_{t}^{(v)}s(x_{1},t):v} \right).
	\end{split}
\end{equation*}
Hence
\begin{equation}\label{graaux1}
	\begin{split}
		& \left(\partial_{t}^{2}	\partial_{t}v(x,t) \right) \Big|_{(x,t_{1})}\\
		&\qquad =-F^{(1)} (0)p(x)(2m\alpha)^{3}+F^{(2)}(0)p(x)^{2} 3(0)(2m\alpha)^{2}+F^{(3)}(0)p(x)^{3} (2m\alpha)^{3}\\
		&\qquad =-8p(x)F^{(1)} (0)m^{3}\alpha^{3}+8p(x)^{3}F^{(3)}(0) m^{3}\alpha^{3}.
	\end{split}
\end{equation}
Now we will focus on the term $\partial_{t}^{2}\partial_{1}(\Delta v)$. We start with  $\partial_{t}^{2}\partial_{1}\partial^{2}_{j} (v)$ for $j\neq 1$:
\begin{equation*}
	\begin{split}
		& \partial_{t}^{2}\partial_{1}\partial^{2}_{j} (v)(x,t)\\
		& \quad =\partial^{2}_{j}\partial_{1}\partial_{t}^{2} (F\circ G)(x,t)\\
		&\quad =\partial^{2}_{j}\partial_{1} \left( \sum_{\ell=1}^{2}F^{(\ell)}(G(x,t))p(x)^{\ell}B_{2,\ell}\left(\bf \partial^{(v)}_{t}s(x_{1},t):v\right)   \right)\\			
		&\quad =\partial^{2}_{j} \left( \sum_{\ell=1}^{2}\partial_{1}(F^{(\ell)}(G(x,t)))p(x)^{\ell}B_{2,\ell}\left(\bf \partial^{(v)}_{t}s(x_{1},t):v\right)   \right)\\			
		&\qquad +\partial^{2}_{j} \left( \sum_{\ell=1}^{2}F^{(\ell)}(G(x,t))\partial_{1}\left( p(x)^{\ell}B_{2,\ell}\left(\bf \partial^{(v)}_{t}s(x_{1},t):v\right)   \right)\right) =: T_1+T_2.		
	\end{split}
\end{equation*}
Next, we expand these two terms. We start with $T_1$:
\begin{equation*}
	\begin{split}
		T_1&=\partial^{2}_{j} \left( \sum_{\ell=1}^{2}\partial_{1}(F^{(\ell)}(G(x,t)))p(x)^{\ell}B_{2,\ell}\left(\bf \partial^{(v)}_{t}s(x_{1},t):v\right)   \right)\\	
		&=\partial^{2}_{j} \left( \sum_{\ell=1}^{2}\left(F^{(\ell+1)}(G(x,t))(\partial_{1}G(x,t))  \right) p(x)^{\ell}B_{2,\ell}\left(\bf \partial^{(v)}_{t}s(x_{1},t):v\right)   \right)\\					
		&=\sum_{\ell=1}^{2}\left(F^{(\ell+1)}(G(x,t))(\partial_{1}G(x,t))  \right) \partial^{2}_{j}(p(x)^{\ell})B_{2,\ell}\left(\bf \partial^{(v)}_{t}s(x_{1},t):v\right)   \\	
		&\quad +\sum_{\ell=1}^{2}2\partial^{1}_{j}\left(F^{(\ell+1)}(G(x,t))(\partial_{1}G(x,t))  \right) \partial^{1}_{j}(p(x)^{\ell})B_{2,\ell}\left(\bf \partial^{(v)}_{t}s(x_{1},t):v\right)   \\	
		&\quad +\sum_{\ell=1}^{2}\partial^{2}_{j}\left(F^{(\ell+1)}(G(x,t))(\partial_{1}G(x,t))  \right) (p(x)^{\ell})B_{2,\ell}\left(\bf \partial^{(v)}_{t}s(x_{1},t):v\right).
	\end{split}
\end{equation*}
Recall \eqref{bells}. Evaluating this expression at time $t=t_{1}$, and since
\[
B_{2,1}\left(\bf \partial^{(v)}_{t}s(x_{1},t):v\right) \Big|_{(x,t_{1})}=0,
\]
and
\[
B_{2,2}\left(\bf \partial^{(v)}_{t}s(x_{1},t):v\right)  \Big|_{(x,t_{1})}=\left(2\alpha m\right)^{2},
\]
we have that
\begin{equation}\label{Egaux1}
	\begin{split}
		&T_1=\partial^{2}_{j} \left( \sum_{\ell=1}^{2}\partial_{1}(F^{(\ell)}(G(x,t)))p(x)^{\ell}B_{2,\ell}\left(\bf \partial^{(v)}_{t}s(x_{1},t):v\right)   \right)\Bigg|_{(x,t_{1})} \\	
		&\quad = F^{(3)}(0) p(x)(2\alpha) \partial^{2}_{j}(p(x)^{2})\left(2\alpha m\right)^{2}   \\	
		&\qquad  +2 F^{(3)}(0)\partial_{j}p(x)(2\alpha) (2p(x)\partial^{1}_{j}p(x))\left(2\alpha m\right)^{2}  +F^{(3)}(0)\partial^{2}_{j}p(x)(2\alpha) (p(x)^{2})\left(2\alpha m\right)^{2}   \\
		&\quad =F^{(3)}(0)8m^{2}\alpha^{3}p(x)(2\partial_{j}p(x)^{2}+2p(x)\partial_{j}^{2}p(x))   \\	
		&\qquad+F^{(3)}(0)32m^{2}p(x)\alpha^{3}\left(\partial_{j}p(x)  \right)^{2}   +F^{(3)}(0)8m^{2}\alpha^{3}p(x)^{2}\partial^{2}_{j}p(x)   \\
		&\quad =p(x)F^{(3)}(0)48m^{2}\alpha^{3}\left(\partial_{j}p(x)  \right)^{2}+p(x)^{2}F^{(3)}(0)24m^{2}\alpha^{3}\partial^{2}_{j}p(x) .
	\end{split}
\end{equation}
Next, we compute the expansion of the second term $T_2$. Thanks to Corollary \ref{F se anula} and the fact that $q$ is odd, $F^{(2)}(0)=0$ and
\[
	\begin{split}
		T_2& =\partial^{2}_{j} \left( \sum_{\ell=1}^{2}F^{(\ell)}(G(x,t))\partial_{1}\left( p(x)^{\ell}B_{2,\ell}\left(\bf \partial^{(v)}_{t}s(x_{1},t):v\right)   \right)\right)\Bigg|_{(x,t_{0})} \\
		&= \sum_{\ell=1}^{2}F^{(\ell)}(0)\cdot\partial^{2}_{j}\left( \partial_{1}\left( p(x)^{\ell}B_{2,\ell}\left(\bf \partial^{(v)}_{t}s(x_{1},t):v\right)   \right)\right)\Bigg|_{(x,t_{0})} \\	
		&= \sum_{\ell=1}^{2}F^{(\ell)}(0)\cdot\left(\partial^{2}_{j}\partial_{1}\left( p(x)^{\ell}\right)B_{2,\ell}\left(\bf \partial^{(v)}_{t}s(x_{1},t):v\right) + \partial^{2}_{j}(p(x)^{\ell})\partial_{1}B_{2,\ell}\left(\bf \partial^{(v)}_{t}s(x_{1},t):v\right)   \right) \Bigg|_{(x,t_{0})} \\	
		&= F^{(1)}(0)\cdot\left(\left(\partial^{2}_{j}\partial_{1} p(x)\right)B_{2,1}\left(\bf \partial^{(v)}_{t}s(x_{1},t):v\right) + \partial^{2}_{j}(p(x))\partial_{1}B_{2,1}\left(\bf \partial^{(v)}_{t}s(x_{1},t):v\right)   \right) \Bigg|_{(x,t_{0})}.
\end{split}
\]
Using now Lemmas \ref{PEB} and \ref{ADhBts} with $(n,k)=(2,1)$,
\begin{equation}\label{Egaux2}
	\begin{split}
		T_2 &= F^{(1)}(0) \Big(\partial^{2}_{j}\partial_{1} p(x)  \partial^{2}_{t}s(x_{1},t) + \partial^{2}_{j} p(x)\partial_{1} \partial^{2}_{t}s(x_{1},t)   \Big) \Bigg|_{(x,t_{0})} \\
		&=-m^{2}(2\alpha)^{3}  F^{(1)}(0) \partial^{2}_{j}(p(x))= -8m^{2}\alpha^{3}F^{(1)}(0) \partial^{2}_{j}(p(x)).							
	\end{split}
\end{equation}
Combining equations \eqref{Egaux1} and \eqref{Egaux2}, we obtain:
\begin{equation}\label{Egaux6}
	\begin{split}
		& \left( \partial_{t}^{2}\partial_{x_1}(\Delta_{c} v(x,t))\right)  \Big|_{(x,t_{1})}\\
		&\quad = -8m^{2}\alpha^{3}F^{(1)}(0) \Delta_{c}(p(x))+48p(x)F^{(3)}(0)m^{2}\alpha^{3}\sum_{j=2}^{N}\left(\partial_{j}p(x)  \right)^{2}   \\	
		&\qquad +24p(x)^{2}F^{(3)}(0)m^{2}\alpha^{3}\Delta_{c}(p(x)).
	\end{split}
\end{equation} 
We will now calculate the term $\partial_{t}^{2}\partial^{3}_{1}v(x,t)$:
\begin{equation*}
	\begin{split}
\partial_{t}^{2}\partial^{3}_{1}v(x,t) &= \partial_{t}^{2}\partial^{3}_{1}(F\circ G))\\
		& =\partial_{t}^{2} \left( \sum_{\ell=1}^{3}F^{(\ell)}(G(x,t))B_{3,\ell}\left(\bf \partial^{(v)}_{1}G(x,t):v\right)   \right)\\
		&=  \sum_{\ell=1}^{3}\partial_{t}^{2}(F^{(\ell)}(G(x,t))) B_{3,\ell}\left(\bf \partial^{(v)}_{1}G(x,t):v\right)   \\
		&\quad +  \sum_{\ell=1}^{3}2\partial_{t}(F^{(\ell)}(G(x,t))) \partial_{t} \left( B_{3,\ell}\left(\bf \partial^{(v)}_{1}G(x,t):v\right) \right)   \\
		&\quad +  \sum_{\ell=1}^{3}F^{(\ell)}(G(x,t)) \partial^{2}_{t} \left( B_{3,\ell}\left(\bf \partial^{(v)}_{1}G(x,t):v\right) \right).
	\end{split}
\end{equation*}
In what follows, we will evaluate each of the previous terms at $t_{1}$. From \eqref{bells}, \eqref{vector_de_derivadas}, \eqref{G_s} and the fact that $s(x_1,t_1)=0$, we obtain
\begin{eqnarray*}
	B_{3,1}\left(\bf \partial^{(v)}_{1}G(x,t):v\right)\Big|_{(x,t_{1})}  &=& \left( \partial^{3}_{1}G(x,t)\right) \Big|_{(x, t_{1})} \\
	&=& 3\partial^{2}_{1}p(x)(2\alpha)-p(x)(2\alpha)^{3}, 
\end{eqnarray*}
\begin{eqnarray*}
	B_{3,2}\left(\bf \partial^{(v)}_{1}G(x,t):v\right)\Big|_{(x,t_{1})}  &=& 3\left( \partial^{1}_{1}G(x,t)\partial^{2}_{1}G(x,t)\right) \Big|_{(x, t_{1})} \\
	&=& 3\left( p(x)(2\alpha)\right) \left( 4\alpha \partial_{1}p(x) \right),
\end{eqnarray*}
and
\begin{eqnarray*}
	B_{3,3}\left(\bf \partial^{(v)}_{1}G(x,t):v\right)\Big|_{(x,t_{1})}  &=& \left( \partial^{1}_{1}G(x,t)\right)^{3} \Big|_{(x,t_{1})} =p(x)^{3}(2\alpha)^{3}.
\end{eqnarray*}
Consequently,
\begin{equation}\label{Egaux3}
	\begin{split}
		& \left(\sum_{\ell=1}^{3}\partial_{t}^{2}(F^{(\ell)}(G(x,t))) B_{3,\ell}\left(\bf \partial^{(v)}_{1}G(x,t):v\right) \right)   \Bigg|_{(x,t_{1})}   \\
		& \quad = \Bigg(\sum_{\ell=1}^{3}\Bigg( F^{(\ell+2)}(G(x,t)) \big( p(x)\partial_{t}s(x_{1},t) \big)^{2}+ F^{(\ell+1)}(G(x,t)) \left( p(x)\partial^{2}_{t}s(x_{1},t) \right)\Bigg)  \\
		& \hspace{2.5cm} B_{3,\ell}\left(\bf \partial^{(v)}_{1}G(x,t):v\right) \Bigg)   \Bigg|_{(x,t_{1})}   \\
		&\quad =\left( F^{(3)}(0)(p(x)^{2}(2m\alpha)^{2})+0 \right)  \left(3\partial^{2}_{1}p(x)(2\alpha)-p(x)(2\alpha)^{3} \right)  \\
		& \qquad+ \left( F^{(4)}(0)(p(x)^{2}(2m\alpha)^{2})+ 0\right) 3\left( p(x)(2\alpha)\right) \left( 4 \alpha \partial_{1}p(x)\right)  \\
		&\qquad+ \left( F^{(5)}(0)(p(x)^{2}(2m\alpha)^{2})+ 0\right)  p(x)^{3}(2\alpha)^{3} \\
		&\quad =  8 m^{2}\alpha^{3} p(x)^{2}\left(  F^{(3)}(0) 3\partial^{2}_{1}p(x)- 4 F^{(3)}(0) \alpha^{2}p(x) \partial^{2}_{1}p(x) + 4 F^{(5)}(0) \alpha^{2}p(x)^{3} \right).
	\end{split}
\end{equation}
Since $q$ is odd, then $F^{(2)}(0)=0$, hence, a straightforward calculation shows that (see \eqref{bells}, \eqref{vector_de_derivadas} and \eqref{G_s})
\begin{equation*}
\begin{aligned}
	\partial_{t}B_{3,1}\left(\bf \partial^{(v)}_{1}G(x,t):v\right)\Big|_{(x,t_{1})}  = &~{} \left( \partial_{t}\partial^{3}_{1}G(x,t)\right) \Big|_{(x,t_{1})} \\
	=&~{}\partial^{3}_{1}p(x)(2m\alpha)-3\partial_{1}p(x)(2m\alpha)(2\alpha)^{2},
\end{aligned}
\end{equation*}
\begin{equation*}	
\begin{aligned}
	\partial_{t}B_{3,2}\left(\bf \partial^{(v)}_{1}G(x,t):v\right)\Big|_{(x,t_{1})}  =&~{} 3\partial_{t}\left( \partial^{1}_{1}G(x,t)\partial^{2}_{1}G(x,t)\right) \Big|_{(x,t_{1})}\\
	=&~{} 3\Big( \partial_{1}p(x)(2m\alpha)(2\partial_{1}p(x)(2\alpha)) \\
	&~{}+ p(x)(2\alpha) \left( \partial_{1}^{2}p(x)(2m\alpha)-p(x)m(2\alpha)^{3}\right)  \Big) \\
	=&~{} 24(\partial_{1}p(x))^{2}m\alpha^{2}+12p(x)\partial_{1}^{2}p(x)m\alpha^{2}-48p(x)^{2}m\alpha^{4},
\end{aligned}
\end{equation*}
and
\begin{equation*}
\begin{aligned}
	\partial_{t}B_{3,3}\left(\bf \partial^{(v)}_{1}G(x,t):v\right)\Big|_{(x,t_{1})}  = &~{} 3\left( \partial^{1}_{1}G(x,t)\right)^{2}\partial_{t}\partial^{1}_{1}G(x,t) \Big|_{(x,t_{1})} \\
	=&~{} 3(p(x)^{2}(2\alpha)^{2})(\partial_{1}p(x)(2m\alpha)).
\end{aligned}
\end{equation*}
Therefore
\begin{equation}\label{Egaux4}
	\begin{split}
		&\sum_{\ell=1}^{3}2\partial_{t}(F^{(\ell)}(G(x,t))) \partial_{t} \left( B_{3,\ell}\left(\bf \partial^{(v)}_{1}G(x,t):v\right) \right)   \\
		&\quad = \sum_{\ell=1}^{3}2\left(F^{(\ell+1)}(G(x,t))p(x)\partial_{t}s(x_{1},t) \right)  \partial_{t} \left( B_{3,\ell}\left(\bf \partial^{(v)}_{1}G(x,t):v\right) \right)   \\
		&\quad = 2\left(F^{(2)}(0)p(x)(2m\alpha) \right)  \left(\partial^{3}_{1}p(x)(2m\alpha)-3\partial_{1}p(x)(2m\alpha)(2\alpha)^{2} \right) \\
		&\qquad + 2\left(F^{(3)}(0)p(x)(2m\alpha) \right)\Big( 24(\partial_{1}p(x))^{2}m\alpha^{2}+12p(x)\partial_{1}^{2}p(x)m\alpha^{2}-48p(x)^{2}m\alpha^{4} \Big) \\
		&\qquad + 2\left(F^{(4)}(0)p(x)(2m\alpha) \right)  3(p(x)^{2}(2\alpha)^{2})(\partial_{1}p(x)(2m\alpha)) \\
		&\quad =  F^{(3)}(0)48p(x)m^{2}\alpha^{3} \left( 2(\partial_{1}p(x))^{2}+p(x)\partial_{1}^{2}p(x)-4p(x)^{2}\alpha^{2} \right)\\
		&\quad =  96F^{(3)}(0)p(x)\partial_{1}p(x)^{2}m^{2}\alpha^{3} +48F^{(3)}(0)p(x)^{2}\partial_{1}p(x)^{2}m^{2}\alpha^{3} \\
		& \qquad -192F^{(3)}(0)p(x)^{3}m^{2}\alpha^{5}.
	\end{split}
\end{equation}
Similarly
\begin{eqnarray*}
	\partial_{t}^{2}B_{3,1}\left(\bf \partial^{(v)}_{1}G(x,t):v\right)\Big|_{(x,t_{1})}  &=& \left( \partial_{t}^{2}\partial^{3}_{1}G(x,t)\right) \Big|_{(x,t_{1})} \\ & =& -3\partial^{2}_{1}p(x)(2m\alpha)^{2}(2\alpha)+p(x)(2\alpha)^{3}(2m\alpha)^{2},
\end{eqnarray*}
\begin{eqnarray*}
	\partial_{t}^{2}B_{3,2}\left(\bf \partial^{(v)}_{1}G(x,t):v\right)\Big|_{(x,t_{1})}  &=& 3\partial_{t}^{2}\left( \partial^{1}_{1}G(x,t)\partial^{2}_{1}G(x,t)\right) \Big|_{(x,t_{1})} \\
	& =& 6\partial_{1}p(x)\partial_{1}^{2}p(x) (2m\alpha)^{2}\\
	&& + 3\partial_{1}p(x) p(x)(2m\alpha)(-(2m\alpha)((2\alpha)^{2})) \\
	&& +6p(x)\partial_{1}p(x)(2(2\alpha)(-(2m\alpha)^{2}(2\alpha))),
\end{eqnarray*}
and
\begin{eqnarray*}
	\partial^{2}_{1}B_{3,3}\left(\bf \partial^{(v)}_{1}G(x,t):v\right)\Big|_{(x,t_{1})}  &=& \partial^{2}_{1}\left( (\partial^{1}_{1}G(x,t) )^{3} \right) \Big|_{(x,t_{1})} \\
	&=& 6p(x)(2\alpha)(\partial_{1}p(x))^{2}(2m\alpha)^{2}-3p(x)^{3}m^{2}(2\alpha)^{5}.
\end{eqnarray*}
Therefore,
\begin{equation}\label{Egaux5}
	\begin{split}
		& \left(\sum_{\ell=1}^{3}F^{(\ell)}(G(x,t)) \partial^{2}_{t} \left( B_{3,\ell}\left(\bf \partial^{(v)}_{1}G(x,t):v\right) \right) \right) \Bigg|_{(x,t_{1})}   \\
		& \quad = F^{(1)}(0) \left( -3\partial^{2}_{1}p(x)(2m\alpha)^{2}(2\alpha)+p(x)(2\alpha)^{3}(2m\alpha)^{2}\right)  \\
		& \qquad+ F^{(2)}(0) \partial^{2}_{t} \left( B_{3,2}\left(\bf \partial^{(v)}_{1}G(x,t):v\right) \right)  \Big|_{(x,t_{1})}  \\
		&\qquad + F^{(3)}(0)\Big( 6p(x)(2\alpha)(\partial_{1}p(x))^{2}(2m\alpha)^{2}-3p(x)^{3}m^{2}(2\alpha)^{5} \Big) \\
		&\quad = -24F^{(1)}(0)\partial^{2}_{1}p(x)m^{2}\alpha^{3} +32F^{(1)}(0)p(x)m^{2}\alpha^{5}  \\
		&\qquad + 48F^{(3)}(0)p(x)(\partial_{1}p(x))^{2}m^{2}\alpha^{3}-96F^{(3)}(0)p(x)^{3}m^{2}\alpha^{5}.
	\end{split}
\end{equation}
From equations \eqref{Egaux3}, \eqref{Egaux4} and \eqref{Egaux5}, we obtain:
\begin{equation}\label{Egaux7}
	\begin{split}
		\left( \partial_{t}^{2}\partial_{x_1}(\Delta_{c} v(x,t))\right)  \Big|_{(x,t_{1})}= & -8m^{2}\alpha^{3}F^{(1)}(0) \Delta_{c}(p(x))\\
		&  +48p(x)F^{(3)}(0)m^{2}\alpha^{3}\sum_{j=2}^{N}\left(\partial_{j}p(x)  \right)^{2}   \\	
		&+24p(x)^{2}F^{(3)}(0)m^{2}\alpha^{3}\Delta_{c}(p(x)) .
	\end{split}
\end{equation} 
From equations \eqref{Egaux6} and \eqref{Egaux7} we obtain:
\begin{equation*}
	\begin{split}
		&\left( \partial_{t}^{2}\partial_{x_1}(\Delta v(x,t))\right)  \Big|_{(x,t_{1})}\\
		&\quad =-8m^{2}\alpha^{3}F^{(1)}(0) \left(\Delta p(x)+2\partial^{2}_{1}p(x) \right) \\
		&\qquad +24p(x)^{2}F^{(3)}(0)m^{2}\alpha^{3}(\Delta p(x)+2\partial^{2}_{1}p(x)) \\
		&\qquad -320p(x)^{3}F^{(3)}(0)m^{2}\alpha^{5}+32F^{(1)}(0)p(x)m^{2}\alpha^{5}\\
		&\qquad +F^{(5)}(0)32m^{2}\alpha^{5}p(x)^{5}\\
		&\qquad +48p(x)F^{(3)}(0)m^{2}\alpha^{3}\left( \sum_{j=2}^{N}\left(\partial_{j}p(x)  \right)^{2}+2\left(\partial_{1}p(x)\right)^{2} \right).
	\end{split}
\end{equation*} 
Since $v$ satisfies \eqref{A2}, from Lemma \ref{Corf1} we have that
\begin{equation*}
	\begin{split}
		\Delta p(x)+ 2\partial_{1}^{(2)}p(x)&=p(x)\left(- m+4\alpha^{2} \right) \\
		& \quad  -p(x)^{3}4\alpha^{2}\left( \frac{F^{(3)}(0)}{F^{(1)}(0)} \right)-p(x)^{q} 2^{q}\alpha^{q-1}F^{(1)}(0)^{q-1}.
	\end{split}
\end{equation*}
Hence:
\begin{equation*}
	\begin{split}
		&\left( \partial_{t}^{2}\partial_{x_1}(\Delta v(x,t))\right)  \Big|_{(x,t_{1})}\\
		&\quad =-8m^{2}\alpha^{3}F^{(1)}(0) \left(p(x)\left( -m+4\alpha^{2}\right)-p(x)^{3}4\alpha^{2} \frac{F^{(3)}(0)}{F^{(1)}(0)} -p(x)^{q} 2^{q}\alpha^{q-1}F^{(1)}(0)^{q-1}\right) \\
		&\qquad +24p(x)^{2}F^{(3)}(0)m^{2}\alpha^{3}\left(p(x)\left( -m+4\alpha^{2}\right)-p(x)^{3}4\alpha^{2} \frac{F^{(3)}(0)}{F^{(1)}(0)}-p(x)^{q} 2^{q}\alpha^{q-1}F^{(1)}(0)^{q-1} \right) \\
		&\qquad -320p(x)^{3}F^{(3)}(0)m^{2}\alpha^{5}+32F^{(1)}(0)p(x)m^{2}\alpha^{5}\\
		&\qquad +F^{(5)}(0)32m^{2}\alpha^{5}p(x)^{5}\\
		&\qquad +48p(x)F^{(3)}(0)m^{2}\alpha^{3}\left( \sum_{j=2}^{N}\left(\partial_{j}p(x)  \right)^{2}+2\left(\partial_{1}p(x)\right)^{2} \right).
\end{split}
\end{equation*} 
Simplifying terms,
\begin{equation*}
	\begin{split}
	&\left( \partial_{t}^{2}\partial_{x_1}(\Delta v(x,t))\right)  \Big|_{(x,t_{1})}\\
		&\quad =8p(x)m^{3}\alpha^{3}F^{(1)}(0)-32p(x)F^{(1)}(0)m^{2}\alpha^{5}\\
		& \qquad +2^{q+3}p(x)^{q}F^{(1)}(0)^{q}m^{2}\alpha^{q+2}+32p(x)^{3}F^{(3)}(0)m^{2}\alpha^{5}  \\
		&\qquad -24p(x)^{3}F^{(3)}(0)m^{3}\alpha^{3}+96p(x)^{3}F^{(3)}(0)m^{2}\alpha^{5}\\
		&\qquad -3\cdot2^{q+3}p(x)^{q+2}F^{(3)}(0)F^{(1)}(0)^{q-1}m^{2}\alpha^{q+2}-96p(x)^{5}m^{2}\alpha^{5}\frac{F^{(3)}(0)^{2}}{F^{(1)}(0)}\\
		&\qquad -320p(x)^{3}F^{(3)}(0)m^{2}\alpha^{5}+32F^{(1)}(0)p(x)m^{2}\alpha^{5}\\
		&\qquad +F^{(5)}(0)32m^{2}\alpha^{5}p(x)^{5}\\
		&\qquad +48p(x)F^{(3)}(0)m^{2}\alpha^{3}\left( \sum_{j=2}^{N}\left(\partial_{j}p(x)  \right)^{2}+2\left(\partial_{1}p(x)\right)^{2} \right).
	\end{split}
\end{equation*}
Finally,
\begin{equation*}
	\begin{split}
	&\left( \partial_{t}^{2}\partial_{x_1}(\Delta v(x,t))\right)  \Big|_{(x,t_{1})}\\
		&\quad =8p(x)m^{3}\alpha^{3}F^{(1)}(0)+2^{q+3}p(x)^{q}F^{(1)}(0)^{q}m^{2}\alpha^{q+2} \\
		&\qquad -24p(x)^{3}F^{(3)}(0)m^{3}\alpha^{3}\\
		&\qquad -3\cdot2^{q+3}p(x)^{q+2}F^{(3)}(0)F^{(1)}(0)^{q-1}m^{2}\alpha^{q+2}-96p(x)^{5}m^{2}\alpha^{5}\frac{F^{(3)}(0)^{2}}{F^{(1)}(0)}\\
		&\qquad -192p(x)^{3}F^{(3)}(0)m^{2}\alpha^{5}\\
		&\qquad +F^{(5)}(0)32m^{2}\alpha^{5}p(x)^{5}\\
		&\qquad +48p(x)F^{(3)}(0)m^{2}\alpha^{3}\left( \sum_{j=2}^{N}\left(\partial_{j}p(x)  \right)^{2}+2\left(\partial_{1}p(x)\right)^{2} \right).
	\end{split}
\end{equation*} 
The proof of the proposition follows by adding the preceding equation to equation \eqref{graaux1}.
\end{proof}

\begin{lemma}\label{ecugra1nl}
Let $v$ of the form \eqref{def_v} be a smooth solution of \eqref{A2} with $q$ odd. If $t_{1}\in\mathbb{R}$ is any element such that $s(x_{1},t_{1})=0$,  $ \partial_t s(x_{1},t_{1})=2\alpha m$, and $ \partial_1 s(x_{1},t_{1})=2\alpha$, then one has
\begin{equation}\label{graaux22}
	\begin{split}
		& 2\partial_{t}^{2}	\big(\partial_{1}v(x,t)^{q} \big)\Big|_{(x,t_{1})}\\
		&\quad  =-2^{q+3}qp(x)^{q} \left(  F^{(1)}(0) \right)^{q}m^{2}\alpha^{q+2}+q(q-1)2^{q+1}(\partial_{1}p(x))^{2}(p(x)^{q-2}) \left( F^{(1)}(0)\right)^{q}m^{2}\alpha^{q}  \\
		&\qquad+2^{q+3}qF^{(3)}(0) \left( F^{(1)}(0) \right)^{q-1}p(x)^{q+2}m^{2}\alpha^{q+2}.
	\end{split}
\end{equation}
\end{lemma}
\begin{proof}
Observe that, since $s(x_{1},t_{1})=0$ and $ \partial_1 s(x_{1},t_{1})=2\alpha$,
\begin{eqnarray*}
	\left( \sum_{k=0}^{q}\binom{q}{k}(\partial_{1}p(x))^{k}(p(x))^{q-k}  \left(s(x,t)^{k}\partial_{1}s(x,t)^{q-k} \right)    \right)\Bigg|_{(x,t_{1})}=(p(x))^{q}(2\alpha)^{q}, 
\end{eqnarray*}
\begin{eqnarray*}
\begin{split}
	&	\sum_{k=0}^{q}\binom{q}{k}(\partial_{1}p(x))^{k}(p(x))^{q-k}  \partial_{t}\left(s(x,t)^{k}\partial_{1}s(x,t)^{q-k} \right)\Bigg|_{(x,t_{1})}	\\
	&\qquad =q (2m\alpha)(2\alpha)^{q-1} \partial_{1}p(x)(p(x))^{q-1},
\end{split}
\end{eqnarray*}
and
\begin{equation*}
	\begin{split}
		&\sum_{k=0}^{q}\binom{q}{k}(\partial_{1}p(x))^{k}(p(x))^{q-k}  \partial_{t}^{2}\left(s(x,t)^{k}\partial_{1}s(x,t)^{q-k} \right)\Bigg|_{(x,t_{1})}\\
		& \qquad =p(x)^{q}(q(2\alpha)^{q-1}(-1)(2\alpha)(2m\alpha)^{2})\\
		&\qquad \quad +\binom{q}{2}(\partial_{1}p(x))^{2} p(x)^{q-2} \left(2(2m\alpha)^{2}(2\alpha)^{q-2} \right)\\
		&\qquad = p(x)^{q} \left( q(2\alpha)^{q-1}(-1)(2\alpha)(2m\alpha)^{2} \right)+q(q-1)(\partial_{1}p(x))^{2} p(x)^{q-2} \left( (2m\alpha)^{2}(2\alpha)^{q-2} \right).
	\end{split}
\end{equation*}
Therefore, using Lemma \ref{Bell_lema1} and expanding, 
\[
	\begin{split}
		2&\partial_{t}^{2}\big( \left( \partial_{1}v(x,t) \right)^{q} \big)\Big|_{(x,t_{1})}\\
		&=2 \partial_{t}^{2}\Big( \big( \partial_{1}(F\circ G) (x,t) \big)^{q} \Big)\Big|_{(x,t_{1})}\\
		&=2\partial_{t}^{2}\left( \left( F^{(1)}(G(x,t)) \right)^{q} \big(\partial_{1}G(x,t) \big)^{q}\right)\Big|_{(x,t_{1})}\\
		&=2\sum_{\ell=0}^{2} \binom{2}{\ell }\partial_{t}^{\ell}\left( F^{(1)}(G(x,t))^{q} \right) \partial_{t}^{2-\ell}\left( (\partial_{1}G(x,t))^{q} \right)\Big|_{(x,t_{1})}.
\end{split}
\]
From the classical Leibnitz rule,
\[
\begin{split}
		2&\partial_{t}^{2}\big( \left( \partial_{1}v(x,t) \right)^{q} \big)\Big|_{(x,t_{1})}\\
		&=2\sum_{\ell=0}^{2} \binom{2}{\ell }\left(\sum_{k_{1}+k_{2}+\dots+k_{q}=\ell}  \binom{\ell}{k_{1},k_{2},\dots,k_{q}}\prod_{i=1}^{q}\partial_{t}^{(k_{i})}(F^{(1)}(G(x,t)))  \right)\Bigg|_{(x,t_{1})}\times\\
		&\qquad \hspace{1.5cm} \partial_{t}^{2-\ell}\Big( \Big(\partial_{1}p(x)s(x,t)+p(x)\partial_{1}s(x,t) \Big)^{q}\Big) \Big|_{(x,t_{1})}\\
		&=2\sum_{\ell=0}^{2} \binom{2}{\ell }\left(\sum_{k_{1}+k_{2}+\dots+k_{q}=\ell}  \binom{\ell}{k_{1},k_{2},\dots,k_{q}}\prod_{i=1}^{q}\partial_{t}^{(k_{i})}(F^{(1)}(G(x,t)))  \right)\Bigg|_{(x,t_{1})}\times\\
		&\qquad \hspace{1.5cm}\left( \sum_{k=0}^{q}\binom{q}{k}(\partial_{1}p(x))^{k}(p(x))^{q-k}  \partial_{t}^{2-\ell}\left(s(x,t)^{k}\partial_{1}s(x,t)^{q-k} \right)    \right)\Bigg|_{(x,t_{1})}.
\end{split}
\]
Evaluating,
\begin{equation*}
	\begin{split}
	& 2 \partial_{t}^{2}	(\partial_{1}v(x,t)^{q})\Big|_{(x,t_{1})}\\
		&\quad =2\left(  F^{(1)}(0)^{q}  \right)\Big(p(x)^{q}(q(2\alpha)^{q-1}(-1)(2\alpha)(2m\alpha)^{2}) \\
		& \qquad \qquad\qquad \qquad +q(q-1)(\partial_{1}p(x))^{2}(p(x)^{q-2})((2m\alpha)^{2}(2\alpha)^{q-2}) \Big) \\
		&\qquad +2\cdot2\left( 3C F^{(2)}(0)F^{(1)}(0)^{q-1}  \right)\left(q(\partial_{1}p(x))(p(x))^{q-1}((2m\alpha)(2\alpha)^{q-1}) \right) \\
		&\qquad +2\left( qF^{(3)}(0)p(x)^{2}(2m\alpha)^{2}F^{(1)}(0)^{q-1}\right)   (p(x))^{q}(2\alpha)^{q}.
	\end{split}
\end{equation*}
Simplifying,
\begin{equation}\label{graaux2}
	\begin{split}
		&\quad =2\left(  F^{(1)}(0)  \right)^{q}\Big( p(x)^{q}(q(2\alpha)^{q-1}(-1)(2\alpha)(2m\alpha)^{2}) \\
		& \qquad \qquad \qquad \qquad +q(q-1)(\partial_{1}p(x))^{2}(p(x)^{q-2})((2m\alpha)^{2}(2\alpha)^{q-2}) \Big) \\
		&\qquad +2^{q+3}qF^{(3)}(0) \left( F^{(1)}(0) \right)^{q-1}p(x)^{q+2}m^{2}\alpha^{q+2} \\
		&\quad =-2^{q+3}qp(x)^{q} \left(  F^{(1)}(0) \right)^{q}m^{2}\alpha^{q+2}+q(q-1)2^{q+1}(\partial_{1}p(x))^{2}p(x)^{q-2} \left( F^{(1)}(0) \right)^{q}m^{2}\alpha^{q}  \\
		&\qquad +2^{q+3}qF^{(3)}(0) \left( F^{(1)}(0) \right)^{q-1}p(x)^{q+2}m^{2}\alpha^{q+2}.
	\end{split}
\end{equation}
The final outcome in \eqref{graaux2} concludes the proof of \eqref{graaux22}.
\end{proof}

\subsection{Conclusions in the case $q$ odd} Now we are ready to state the two main results of this section.

\begin{proposition}\label{ecugrac}
Under $F^{(1)}(0)\neq0$, $F^{(3)}(0)\neq 0$, and $q$ odd, the following is satisfied. Let $v$ be as in \eqref{def_v}, and satisfying \eqref{A2} with $q$ odd. Then, if $t_{1}\in\mathbb{R}$ is any element such that $s(x_{1},t_{1})=0$,  $ \partial_t s(x_{1},t_{1})=2\alpha m$, and $ \partial_1 s(x_{1},t_{1})=2\alpha$, then it holds true that:
\begin{equation*}\label{gradc}
	\begin{split}
		0 & = \left(48F^{(3)}(0)m^{2}\alpha^{3} \right)^{-1}\Big( \partial_{t}^{2}	\Big( \partial_t v+\partial_{x_1}(\Delta v) +2(\partial_{x_1}v)^q\Big) \Big)\Big|_{(x,t_{1})} \\
		& \quad = p(x)\Big( |\nabla p(x)|^{2}+ 2(\partial_{1} p(x))^{2}\Big)+p(x)^{q-2}(\partial_{1} p(x))^{2} \left( \frac{q(q-1)2^{q-3}\alpha^{q-3}F^{(1)}(0)^{q}}{3F^{(3)}(0)} \right)\\
		& \qquad - p(x)^{3}\left(\frac{m}{3} +4\alpha^{2} \right) + p(x)^{5}\left( \frac{2\alpha^{2}F^{(5)}(0)}{3F^{(3)}(0)} -\frac{2\alpha^{2} F^{(3)}(0)}{F^{(1)}(0)}\right) \\
		& \qquad + p(x)^{q}(1-q)\frac{2^{q-1}\alpha^{q-1}F^{(1)}(0)^{q}}{3F^{(3)}(0)}+p(x)^{q+2}2^{q-1} \left( F^{(1)}(0) \right)^{q-1}\alpha^{q-1} \left( 1-\frac{q}{3} \right).
	\end{split}
\end{equation*}
\end{proposition}
\begin{proof}
The proof follows directly from the addition of the results in Lemmas \ref{ecugra1} and \ref{ecugra1nl}, and the fact that \eqref{A2} is valid for all $(t,x)$. In particular, its time derivatives are also zero for all $(t,x)$.
\end{proof}

\begin{proposition}  Under $F^{(1)}(0)\neq0$, $q$ odd, the following equations are satisfied by $p$:
\begin{equation}\label{sec1_repetida}
	\begin{split}
		&   \Delta p(x)+ 2\partial_{1}^{(2)}p(x) + \left( m- 4\alpha^{2}\right)p(x) \\
		&\quad +4 \alpha^{2}\left( \frac{F^{(3)}(0)}{F^{(1)}(0)} \right)p(x)^{3}+  2^q \alpha^{q-1} \left(F^{(1)}(0) \right)^{q-1} p(x)^{q}=0,
	\end{split}
\end{equation}
and
\begin{equation}\label{Lq}
\Delta p(x)-12\alpha^{2} p(x)  +12\alpha^{2}\left(  \frac{F^{(3)}(0)}{F^{(1)}(0)}\right)p(x)^{3}  +2^{q}\alpha^{q-1} \left(F^{(1)}(0)\right)^{q-1} p(x)^{q} =0.
\end{equation}
In consequence, there is a rigidity equation in the $x_1$ variable, 
\begin{equation}\label{L1q}
\partial_{1}^{(2)}p(x) + \left( \frac12m +4 \alpha^{2}\right)p(x)- 4 \alpha^{2}\left( \frac{F^{(3)}(0)}{F^{(1)}(0)} \right)p(x)^{3}=0,
\end{equation}
and the Laplacian for the $(x_2,\ldots,x_N)$ coordinates satisfies
\begin{equation}\label{Lcq}
\Delta_c p(x)- \left( \frac12m+16\alpha^{2}  \right)p(x)  +16\alpha^{2}\left(  \frac{F^{(3)}(0)}{F^{(1)}(0)}\right)p(x)^{3}  +2^{q}\alpha^{q-1} \left(F^{(1)}(0)\right)^{q-1} p(x)^{q}  =0.
\end{equation}
\end{proposition}

\begin{proof}
Equation \eqref{sec1_repetida} is Lemma \ref{Corf1}, equation \eqref{sec1}  in the case $q$ odd. Similarly, \eqref{Lq} is nothing but the result obtained from \eqref{integrada} in Lemma \ref{EcuLapl2}. Finally, \eqref{L1q} and \eqref{Lcq} are obtained after suitable subtraction and addition of \eqref{sec1_repetida} and \eqref{Lq}. 
\end{proof}

%
%
%
%

\begin{corollary}\label{F3}
If $q$ is odd, $F^{(1)}(0)\neq 0$ and $F$ is nontrivial, then $F^{(3)}(0)\neq 0$.
\end{corollary}

\begin{proof}
From \eqref{L1q}, if we assume $F^{(3)}(0)= 0$ and $F^{(1)}(0) \neq 0$,
\begin{equation}\label{corta1}
\partial_{1}^{(2)}p(x) + \left( \frac12m +4 \alpha^{2}\right)p(x) =0.
\end{equation}
Equation \eqref{corta1} is a classical second order ODE in the $x_1$ variable and has solutions depending on the value of $m$, given by exponentials, linearly growing, constants or trigonometric functions. In any case, the condition $p(x)\to 0$ as $|x|\to +\infty$ is satisfied, leading to $p\equiv 0$. Therefore, we get a contradiction.
\end{proof}

\section{Uniqueness in dimension $N=1$, $q$ odd}\label{N1} 

In this section and the next one we prove that the mKdV breather is the unique quasimonochromatic gKdV solution, namely Theorem \ref{MT1}. 

\medskip

Set $N=1$.  If $F^{(1)}(0)=0$, by Corollary \ref{F se anula}  one has $F^{(n)}(0)=0$ for each $n\in \mathbb N$. Since $F$ is supposed analytic one concludes that $F\equiv 0$, leading to a trivial solution. 
Therefore, we can assume $F^{(1)}(0)\neq0$. 

\subsection{The case $q=3$} First we consider the mKdV case.  From now on we denote $x=x_1$.  Notice that from \eqref{Lcq} in the case $N=1$, since $p(x)$ is nontrivial, one gets
\[
- \left( \frac12m+16\alpha^{2}  \right)p(x)  +8\alpha^{2}\left( 2 \frac{F^{(3)}(0)}{F^{(1)}(0)} +  \left(F^{(1)}(0)\right)^{2} \right)p(x)^{3}   =0.
\]
and therefore
\begin{equation}\label{m}
	m=-32\alpha^2.
\end{equation}
and
\begin{equation*}
	 F^{(3)}(0)=-\frac{1}{2}F^{(1)}(0)^3.
\end{equation*}
Replacing in \eqref{L1q} $p$ satisfies 
\[
	-\partial_{1}^{(2)} p(x)=-12 \alpha^2 p(x)+  2 \left( \alpha F^{(1)}(0) \right)^2 p^3(x), \quad x\in \mathbb R.
\]
Performing the rescaling $p(x)=\zeta_0 \widetilde p(2\sqrt{3}\alpha x)$, $\zeta_0=2\sqrt{\frac{3}{F^{(1)}(0)^2}}$, one obtains
\begin{equation}\label{Lions}
	-\partial_{1}^{(2)}  \widetilde p(y) =-\widetilde p(y)+  2   \widetilde p(y)^3, \quad y\in \mathbb R.
\end{equation}
Since $p$ is not trivial, then $\widetilde p$ is the unique (up to translations) solution to the equation \eqref{Lions}, see \cite[Theorem 5]{BL}. This implies that $\widetilde p=\sech ( \cdot + y_0)$, where the convergence to zero at infinity discards a nontrivial dependence of $y_0$ on the rest of variables. Moreover,
\begin{equation}\label{valor_final_p}
p(x)=\zeta_0 \sech \left( 2\sqrt{3}\alpha x + x_2\right), \quad x_2:=2\sqrt{3}\alpha y_0.
\end{equation}
Now we will  show that $F$ is the function $-2\arctan$ or $2\arctan$. Without loss of generality we assume that $x_2=0$.   Performing the ansatz \eqref{ansatz} with \eqref{m}, \eqref{valor_final_p} and replacing  in \eqref{mKdV} and \eqref{mZKqv} respectively, we get after some simplifications and evaluating at $x=0$:
\begin{equation*}
	\begin{aligned}
		0=&~{} 6 \zeta _0^3 \cos ^4\left(64 \alpha ^3 t\right) F^{(1)}\left(\zeta _0 \left(-\sin \left(64 \alpha ^3 t\right)\right)\right){}^2 F^{(2)}\left(\zeta _0 \left(-\sin \left(64 \alpha ^3 t\right)\right)\right)\\
		&-72 \sin \left(64 \alpha ^3 t\right) F^{(1)}\left(\zeta _0 \left(-\sin \left(64 \alpha ^3 t\right)\right)\right) \\
		& +24 \zeta _0^2 \sin \left(64 \alpha ^3 t\right) \cos ^2\left(64 \alpha ^3 t\right) F^{(1)}\left(\zeta _0 \left(-\sin \left(64 \alpha ^3 t\right)\right)\right){}^3\\
		&+48 \zeta _0 \sin ^2\left(64 \alpha ^3 t\right) F^{(2)}\left(\zeta _0 \left(-\sin \left(64 \alpha ^3 t\right)\right)\right) \\
		& -48 \zeta _0 \cos ^2\left(64 \alpha ^3 t\right) F^{(2)}\left(\zeta _0 \left(-\sin \left(64 \alpha ^3 t\right)\right)\right)\\
		&+\zeta _0^3 \cos ^4\left(64 \alpha ^3 t\right) F^{(4)}\left(\zeta _0 \left(-\sin \left(64 \alpha ^3 t\right)\right)\right)\\
		&+24 \zeta _0^2 \sin \left(64 \alpha ^3 t\right) \cos ^2\left(64 \alpha ^3 t\right) F^{(3)}\left(\zeta _0 \left(-\sin \left(64 \alpha ^3 t\right)\right)\right),
	\end{aligned}
\end{equation*}
and
\begin{equation*}
	\begin{aligned}
		0&=-18 F^{(1)}\left(\zeta _0 \left(-\sin \left(64 \alpha ^3 t\right)\right)\right)+2 \zeta _0^2 \cos ^2\left(64 \alpha ^3 t\right) F^{(1)}\left(\zeta _0 \left(-\sin \left(64 \alpha ^3 t\right)\right)\right){}^3\\
		&\quad +12 \zeta _0 \sin \left(64 \alpha ^3 t\right) F^{(2)}\left(\zeta _0 \left(-\sin \left(64 \alpha ^3 t\right)\right)\right)+\zeta _0^2 \cos ^2\left(64 \alpha ^3 t\right) F^{(3)}\left(\zeta _0 \left(-\sin \left(64 \alpha ^3 t\right)\right)\right).
	\end{aligned}
\end{equation*}
Now, if we set $ z=-\zeta_0\sin \left(64 \alpha ^3 t\right)$ then $\cos \left(64 \alpha ^3 t\right)=\sqrt{1-\frac {z^2} {\zeta_0^2}}$, provided that $1-\frac {z^2} {\zeta_0^2}\ge 0$. Then from the two last equalities we conclude
\begin{equation}\label{F1}
	\begin{aligned}
		0=&~{} 6 \zeta _0^3 \left(1-\frac{z^2}{\zeta _0^2}\right)^2 F^{(1)}(z)^2 F^{(2)}(z)-24 \zeta _0 z \left(1-\frac{z^2}{\zeta _0^2}\right) F^{(1)}(z)^3\\
		&~{} -48 \zeta _0 \left(1-\frac{z^2}{\zeta _0^2}\right) F^{(2)}(z) +\frac{48 z^2 F^{(2)}(z)}{\zeta _0}\\
		&~{}+\zeta _0^3 F^{(4)}(z) \left(1-\frac{z^2}{\zeta _0^2}\right)^2-24 \zeta _0 z F^{(3)}(z) \left(1-\frac{z^2}{\zeta _0^2}\right)+\frac{72 zF^{(1)}(z)}{\zeta _0},
	\end{aligned}
\end{equation}
and
\begin{equation}\label{F2}
	\begin{aligned}
		0&=2 \zeta _0^2 \left(1-\frac{z^2}{\zeta _0^2}\right) F^{(1)}(z)^3+\zeta _0^2 F^{(3)}(z) \left(1-\frac{z^2}{\zeta _0^2}\right)-18 F^{(1)}(z) -12 z F^{(2)}(z).
	\end{aligned}
\end{equation}
Multiplying by $\zeta _0 \left(1-\frac{z^2}{\zeta _0^2}\right)$ the $z$ derivative of \eqref{F1} and subtracting \eqref{F2}, we get
\begin{equation}\label{F3b}
\begin{aligned}
0=	&~{} -\frac{20 z^3 F^{(1)}(z)^3}{\zeta _0}-\frac{66 z^2 F^{(2)}(z)}{\zeta _0}+F^{(3)}(z) \left(10 \zeta _0 z-\frac{10 z^3}{\zeta _0}\right) \\
&~{} +20 \zeta _0 z F^{(1)}(z)^3-\frac{72 z F^{(1)}(z)}{\zeta _0}+18 \zeta _0 F^{(2)}(z).
\end{aligned}
\end{equation}
Also multiplying \eqref{F2} by $\zeta _0^2$ and subtracting \eqref{F3b} multiplied by $10z\zeta_0$ we have
\begin{equation*}
0=	\zeta _0 \left(\left( 3  z^2 +\zeta_0^2 \right) F^{(2)}(z)+ 6  z F^{(1)}(z) \right),
\end{equation*}
which solution is given by
\begin{equation*}
F(z)=	\frac{c_1}{\sqrt{3} \zeta _0} \arctan\left(\frac{\sqrt{3} z}{\zeta _0}\right) +c_2,
\end{equation*}
for some constants $c_1, c_2\in \mathbb{R}$. Since $F(0)=0$, $c_2=0$, then from \eqref{ansatz} and the fact that $\zeta_0=2\sqrt{\frac{3}{F^{(1)}(0)^2}}$,
\begin{equation}\label{F}
	F(z)=2\frac{F^{(1)}(0)}{|F^{(1)}(0)|}  \arctan\left(\frac{|F^{(1)}(0)| z}{2}\right).
\end{equation}
Thus, from \eqref{valor_final_p} and \eqref{F},
\begin{equation}\label{ansatzfinal}
\begin{aligned}
	v(t,x)&=F\left(\zeta_0\sech \left( 2\sqrt{3}\alpha x+x_2 \right)\sin \left( 2\alpha(x+m t) \right) \right)\\
&=F\left(\frac{2\sqrt{3}}{|F^{(1)}(0)|}\sech \left( 2\sqrt{3}\alpha x+x_2 \right)\sin \left( 2\alpha(x-32\alpha^2t) \right)\right)\\
&=\pm2  \arctan\left(\sqrt{3}\sech\left(2\sqrt{3}\alpha x+x_2\right)\sin \left( 2\alpha(x-32\alpha^2t) \right)\right).
\end{aligned}
\end{equation}
Hence from \eqref{ansatzfinal} we obtain the solution has the form as in \eqref{ansatzlamb}.

\subsection{The case $q\neq 3$} Finally we consider the case $q\neq 3$. Recall that $N=1$. We only focus in the case $q$ odd since $q$ even yields a suitable ``blow-up'', see the proof of Theorem \ref{MT2} in the next section. 

\medskip

From \eqref{eqvt} 
one has  
\begin{equation}\label{vqt}
	\begin{split}
		0 = &~{} 2\alpha mF^{(1)}(0)\partial_{1}^{(3)} p(x)   -24m\alpha^{3}F^{(1)}(0) \partial_{1}p(x) \\
		& +72m\alpha^{3}F^{(3)}(0)p(x)^{2}\partial_{1}p(x)+2^{q+1}q\alpha^{q}m \left( F^{(1)}(0) \right)^{q} \left(p(x) \right)^{q-1}\partial_{1}p(x).
	\end{split}
\end{equation}
Also from \eqref{eq_simplificada_2} in Lemma \ref{eq_simplificada_1}, 
\begin{equation}\label{vqx}
\begin{aligned}
  0=&~{} 8 \alpha  F^{(1)}(0) \partial_1^{(3)} p(x)+ 2 \alpha  m F^{(1)}(0) \partial_{1} p(x) -32 \alpha ^3 F^{(1)}(0) \partial_{1} p(x)\\
  &~{} +96 \alpha ^3 F^{(3)}(0) p(x)^2 \partial_{1}p(x) + 2^{q+2} q\alpha^q   \left(F^{(1)}(0) \right)^q \left(  p(x)\right)^{q-1}\partial_{1}p(x) .
  \end{aligned}
\end{equation}
Now multiplying by 4 equation \eqref{vqt} and adding  \eqref{vqx} multiplicated by $-m$ we get
after some simplifications 
\[
0= F^{(1)}(0) \left( \left( m-32 \right) \alpha ^2+2^{q+1} q \left(\alpha  F^{(1)}(0) p(x)\right)^{q-1} \right) + 96 \alpha ^2 F^{(3)}(0) p(x)^2 .
\] 
Assume $F^{(1)}(0)\neq 0$. Denoting $s:=\alpha  F^{(1)}(0) p(x)$ we obtain the equivalen equation for $s$ in the image of $\alpha  F^{(1)}(0) p$:
\[
0=m-32 \alpha ^2+\frac{96 F^{(3)}(0) s^2}{F^{(1)}(0)^3}+2^{q+1} q s^{q-1}.
\]
Recall that $q\neq 3$. Above equality seen as a zero polynomial in $s$, that yields to a contradiction. Therefore $F^{(1)}(0)= 0$ and from Corollary \ref{F se anula} one gets $F\equiv 0$, leading to the undesired trivial solution.
Now the proof of Theorem \ref{MT1} in the case $N=1$ and $q$ odd is complete. The remaining part $q$ even will be proved in the following section, since it is independent of the dimension.

\section{Proof of Theorems \ref{MT2} and \ref{MT1}, $q$ even case}\label{Ngen}

We finally prove  Theorem \ref{MT2}, namely nonexistence of quasimonochormatic breathers in dimensions $N\geq 2$.  

\medskip

Assume $N\geq1$. First we consider  the case in where $q$ is even. 

\subsection{Case $q$ even} Notice that from \eqref{sec2} in Lemma \ref{Corf1},
\begin{equation}\label{EDOpar}
	6F^{(2)}(0)p(x)\partial_{1}p(x) +  2^{q-1}\alpha^{q-2} \left( F^{(1)}(0) \right)^{q}p(x)^{q} =0.
\end{equation} 
By  Corollary \ref{F se anula}, we can assume $F^{(1)}(0) \neq 0$, and consequently, $F^{(2)}(0)\neq 0$.
Equation \eqref{EDOpar} is a nonlinear first order ODE for $p$ in the variable $x_1$, the rest of variables being free parameters. Assuming that $p$ is nontrivial the solution is given by 
\begin{equation*}
 p(x)=
 \begin{cases}
\displaystyle{ \left((q-2) \left(\frac{\left(2\alpha  F^{(1)}(0) \right)^q }{12 \alpha ^2 F^{(2)}(0)}x_1-c_1(x_2,\ldots,x_N)\right)\right)^{\frac{1}{2-q}}}, \quad q> 2,\\
\displaystyle{c_1(x_2,\ldots,x_N) \exp\left( -  \left(\frac{\left(2\alpha  F^{(1)}(0) \right)^q}{12 \alpha ^2 F^{(2)}(0)}  \right) x_1 \right)}, \quad q=2.
 \end{cases} 
\end{equation*}
In the first case $p$ has a blow-up for some $x_1$ fixed, and in the second it does not converge to 0 as $x_1\to \pm\infty$ (one of the two directions), which is a contradiction. Therefore, for any $N$ and $q$ even, we discard quasimonochormatic breather solutions to ZK, proving Theorem \ref{MT1} in the remaining $q$ even, and Theorem \ref{MT2} for all $N\geq 2$ and $q$ even.

\subsection{Case $q$ odd}
 This part is inspired by the steps followed by Mandel in \cite[Proof of Thm. 1]{RainerMandel2021}.   If $F^{(1)}(0)=0$ then  by Corollary \ref{F se anula}   $F^{(n)}(0)=0$ for all $n\in \mathbb{N}$ thus $F\equiv 0$, since $F$ is supposed analytic. In case $F^{(1)}(0)\neq0$ by Corollary \ref{F3} we have  $F^{(3)}(0)\neq 0$. On the one hand we affirm $p$ does not changes  sign. Indeed from Lemma \ref{EcuLapl2} and Proposition \ref{ecugrac} one has 
\begin{equation}\label{hop}
  \begin{aligned}
0=&~{}\mu_1{p}^q+\mu_2{p}^3+\mu_3 p+\Delta p
\end{aligned}
\end{equation}
\begin{equation}\label{hop2}
  \begin{aligned}
0=&~{}\lambda_1{p}^3+\lambda_2{p}^{q+2}+\lambda_3{p}^{q}+\lambda_4p^5+p\left(|\nabla_c {p}|^2+3(\partial_{x_1}p)^2\right)+\lambda_5p^{q-2}(\partial_{x_1}p)^2.\\
\end{aligned}
\end{equation}
where
\begin{equation*}\begin{aligned}
  \mu_1&= 2^{q}\alpha^{q-1} F^{(1)}(0)^{q-1}, \quad \mu_2 = 12\alpha^{2}\frac{F^{(3)}(0)}{F^{(1)}(0)}, \quad \mu_3 = -12\alpha^2,\\
  \lambda_1&= -\frac{m}{3}-4\alpha^{2} , \quad \lambda_2 = 2^{q-1}F^{(1)}(0)^{q-1}\alpha^{q-1}\left( 1-\frac{q}{3} \right), \\ \lambda_3 &= (1-q)\frac{2^{q-1}\alpha^{q-1}F^{(1)}(0)^{q}}{3F^{(3)}(0)}, \quad \lambda_4 = \left(-\frac{2\alpha^{2} F^{(3)}(0)}{F^{(1)}(0)}+ \frac{2\alpha^{2}F^{(5)}(0)}{3F^{(3)}(0)} \right).\\
  \lambda_5&= \frac{q(q-1)2^{q-3}\alpha^{q-3}F^{(1)}(0)^{q}}{3F^{(3)}(0)}.
  \end{aligned}
\end{equation*}

Since $p$ is nontrivial, there is $x_0$ in where $p(x_0)\neq 0$. Without loss of generality we assume that $p(x_0)>0$. Suppose $p$ changes  sign. Let $B$ the maximal ball with center $x_0$ such that $p(x)>0$ for each $x\in B\subset \mathbb R^N$ and let $x^*\in \partial B$ such that $p(x^*)=0$. Since $p$ satisfies  \eqref{hop} in $B$ by  Hopf's Lemma \cite[Chapter 3, Lemma 3.4]{GilTru} we have $|\nabla {p(x^*)}|\neq 0$. Now let  $(x_n)\subset B$  a sequence such that $x_n\to x^*$. By \eqref{hop} and \eqref{hop2} one gets 

\begin{equation}\label{hop1}
  \begin{aligned}
0=&~{}\mu_1{p(x_n)}^q+\mu_2{p(x_n)}^3+\mu_3 p(x_n)+\Delta p(x_n)\\
&+\lambda_1{p(x_n)}+\lambda_2{p(x_n)}^{q}+\lambda_3{p(x_n)}^{q-2}+\lambda_4p(x_n)^3\\
&+\frac{1}{p(x_n)}\left(|\nabla_c {p(x_n)}|^2+3(\partial_{x_1}p(x_n))^2\right)+ \lambda_5p^{q-4}(\partial_{x_1}p(x_n))^2.\\
\end{aligned}
\end{equation}
  In the case $q\geq 5$ taking limit in \eqref{hop1} as $n\to \infty$ yields $\Delta p (x^*)$ unbounded, since $|\nabla_c {p(x_n)}|^2\neq 0$ or $(\partial_{x_1}p(x_n))^2\neq 0$, which is a contradiction.  In case $q=3$ the situation is more delicate and the sign condition of  $\left(\frac{2 F^{(1)}(0)^3}{F^{(3)}(0)}+3\right)>0$ is key. In fact \eqref{hop1} reads as follows 
\begin{equation}\label{explosion}
  \begin{aligned}
0=&~{}(\mu_1+\lambda_2+\lambda_4){p(x_n)}^3+\mu_2{p(x_n)}^3+\Delta p(x_n)\\
&+(\lambda_1+\lambda_3+\mu_3){p(x_n)}\\
&+\frac{1}{p(x_n)}\left(|\nabla_c {p(x_n)}|^2+\left(\frac{2 F^{(1)}(0)^3}{F^{(3)}(0)}+3\right)(\partial_{x_1}p(x_n))^2\right).
\end{aligned}
\end{equation}

\begin{lemma}\label{no_da}
One has
\begin{equation}\label{signo_cool}
\left( F^{(1)}(0) \right)^3> -\frac32 F^{(3)}(0).
\end{equation}
\end{lemma}
\begin{proof}
	Without loss of generality we can assume $F^{(1)}(0)>0$ (since $u$ is solution implies $-u$ is also a solution in the case $q$ odd). Assume, by contradiction, that $0< \left( F^{(1)}(0) \right)^3 \leq -\frac32 F^{(3)}(0)$, so that $F^{(3)}(0)<0.$ Recall \eqref{Lq}, and the fact that $q=3$.	Then we have that $p$ solves
	\begin{equation*}
		-\Delta p(x)=f(p(x)),
	\end{equation*} 
where 
\[
f(p) : = -12\alpha^{2} p  +4 \alpha^{2}  \left(F^{(1)}(0)\right)^{2} \left( 3\left(  \frac{F^{(3)}(0)}{(F^{(1)}(0))^3}\right)+2\right)  p^{3}.
\]
Clearly $f$ is smooth. 

\medskip

First, we prove that $p \leq 0$. Indeed, assume $p(y_1)>0$ for some $y_1\in\mathbb R^N$. Since $p\in C^2(\mathbb R^N)$ and $|p|\to 0$ as $|x|\to +\infty$, $p$ is bounded in $\mathbb R^N$. Therefore, $\sup_{x\in\mathbb R^N}p(x)>0$ is finite. Since $p(x)\to 0$ as $|x|\to +\infty$, given $\varepsilon>0$ there is $R>0$ such that $|p(y)|<\varepsilon$ if $|y|>R$. Assume $\varepsilon = \frac12p(y_1)$ and choose a corresponding $R$. Now consider the ball $\overline{B}(0,R)$. By making $R>0$ larger if necessary, $y_1\in \overline{B}(0,R)$. Therefore, the maximum value of $p$ is attained inside a compact set of $\mathbb R^N$. Therefore, we can assume $y_1$ a local maximum, and $Dp(y_1)=0$, $\Delta p(y_1)\leq 0$. Consequently
\[
0< -\Delta p(y_1) +12\alpha^2 p(y_1) = 4 \alpha^{2}  \left(F^{(1)}(0)\right)^{2} \left( 3\left(  \frac{F^{(3)}(0)}{(F^{(1)}(0))^3}\right)+2\right)  p(y_1)^{3},
\]
implying that 
\[
3\left(  \frac{F^{(3)}(0)}{(F^{(1)}(0))^3}\right)+2>0,
\]	
a contradiction. Therefore, $p\leq0$.

\medskip

Now, we prove that $p\geq 0$. Once again, assuming the contrary, for some $y_2\in\mathbb R^N$, one has $p(y_2)<0$, $Dp(y_2)=0$, $\Delta p(y_2) \geq 0$. In this case, 
\[
0> -\Delta p(y_2) +12\alpha^2 p(y_2) = 4 \alpha^{2}  \left(F^{(1)}(0)\right)^{2} \left( 3\left(  \frac{F^{(3)}(0)}{(F^{(1)}(0))^3}\right)+2\right)  p(y_2)^{3},
\]
implying that
\[
3\left(  \frac{F^{(3)}(0)}{(F^{(1)}(0))^3}\right)+2 >0.
\]
In both cases we have arrived to a contradiction, so necessarily \eqref{signo_cool} holds.
\end{proof}


From Lemma \ref{no_da},  $|\nabla_c {p(x^*)}|^2+\left(\frac{2 F^{(1)}(0)^3}{F^{(3)}(0)}+3\right)(\partial_{x_1}p(x^*))^2>0$.   We conclude from \eqref{explosion} that $|\Delta p (x^*)|=\infty$, a contradiction.  Thus $p$ does not change sign.

%
%
%
%
%
%

\medskip

 Without loss of generality in the next we assume that $p>0$. Notice that  $p$ is radial. Indeed from \eqref{hop} one has 
\begin{equation*}
  \begin{aligned}
0=&\mu_1{p}^q+\mu_2{p}^3+\mu_3 p+\Delta p.
\end{aligned}
\end{equation*}
Since $\mu_3<0$, $p>0$ and $p$ decays if $|x|\to\infty$ from \cite[Thm. 2 Rmk.]{Gidas} we have $p$ is $x_0$-radial for some $x_0\in \mathbb R$, that is $p(x)=p_0(|x-x_0|)$. For simplicity we consider the center $x_0$ at the origin. Now we consider the spherical change of variable on $\mathbb R^N$, ($N\geq2$) given by 
\begin{equation*}
\begin{aligned}
x_1= &~{} r \cos \left(\varphi_1\right), \\
x_2= &~{} r \sin \left(\varphi_1\right) \cos \left(\varphi_2\right), \\
x_3= &~{} r \sin \left(\varphi_1\right) \sin \left(\varphi_2\right) \cos \left(\varphi_3\right), \\
& \vdots \\
x_{n-1}= &~{} r \sin \left(\varphi_1\right) \cdots \sin \left(\varphi_{n-2}\right) \cos \left(\varphi_{n-1}\right), \\
x_n= &~{} r \sin \left(\varphi_1\right) \cdots \sin \left(\varphi_{n-2}\right) \sin \left(\varphi_{n-1}\right).
\end{aligned}
\end{equation*}
From \eqref{hop2} one gets
\begin{equation*}
\begin{aligned}
  0=&~{}\lambda_1 p_0^3(r)+\lambda_2 p_0^{q+2}(r)+\lambda_3 p_0^q(r)+\lambda_4 p_0^5(r)\\
    &~{} +p_0(r)\left(p'_0(r)(1-\cos^2(\varphi_1))+3p_0'(r)\cos^2(\varphi_1))\right)+\lambda_5 p_0(r)^{q-2}p_0'(r)\cos^2(\varphi_1).
\end{aligned}
\end{equation*}
for all $r>0$ and $\varphi_1\in \mathbb R$.
In case $q\ge5$ we have 
\begin{equation*}
\begin{aligned}
  0=&~{}\lambda_1 p_0^3(r)+\lambda_2 p_0^{q+2}(r)+\lambda_3 p_0^q(r)+\lambda_4 p_0^5(r)+p_0(r)p'_0(r)(1+2\cos^2(\varphi_1))\\
    &~{}+\lambda_5 p_0(r)^{q-2}p_0'(r)\cos^2(\varphi_1).
\end{aligned}
\end{equation*}
Thus $p_0(r)p'_0(r)=0$ and $\lambda_5p_0(r)^{q-2}p_0'(r)=0$. This implies $p'_0(r)=0$ is constant, a contradiction by the decay condition of $p$. The case $q=3$ follows similarly. Hence the proof of Theorem \ref{MT2} is complete.
\medskip

\end{document}